\documentclass[
	a4paper,
	11pt
]{amsart}

\usepackage[british]{babel}
\usepackage{enumitem}

\usepackage{times}
\usepackage[T1]{fontenc}

\usepackage[leqno]{amsmath}
\usepackage{amssymb}
\usepackage{amsthm}
\usepackage{amsaddr}
\usepackage{mathtools}
\usepackage{stmaryrd}
\usepackage[mathcal]{euscript}
\usepackage{mathrsfs}%\mathscr

\usepackage{tikz-cd}
\usetikzlibrary {shapes.arrows, arrows.meta, calc, intersections, decorations.pathmorphing}
\tikzset{bg/.style={fill=gray, fill opacity=0.05}}
\definecolor{figcolor}{HTML}{3E4953}
\definecolor{accentcolor}{HTML}{53403E}
\tikzset{bg/.style={fill=gray, fill opacity=0.05}}
\tikzset{textbg/.style={fill=white, fill opacity=.8, text opacity=1, inner sep=0pt}}

\usepackage[
backend=biber,
style=alphabetic,
sorting=nyt,
date=year,
giveninits=true,
useprefix=true,
doi=false,isbn=false,url=false,
maxnames=99
]{biblatex}
\bibliography{bibliography.bib}

\usepackage[dvipsnames,svgnames]{xcolor}
\definecolor{TodoColor}{HTML}{b7636b}
\definecolor{LinkColor}{HTML}{849f78}
\definecolor{CiteColor}{HTML}{d3a951}
\definecolor{URLColor}{HTML}{5a6d81}
\usepackage[
	colorlinks=true,
	linkcolor=LinkColor,
	citecolor=CiteColor,
	urlcolor=URLColor
]{hyperref}
\usepackage{aliascnt}
\usepackage[capitalise,nameinlink,noabbrev]{cleveref}

\theoremstyle{plain}
\newtheorem{theorem}{Theorem}[section]
\newtheorem*{theorem*}{Theorem}
\crefname{theorem}{Theorem}{Theorems}

\newcommand{\newaliastheorem}[3]{%
	\newaliascnt{#1}{theorem}%
	\newtheorem{#1}[#1]{#2}%
	\aliascntresetthe{#1}%
	\crefname{#1}{#2}{#3}%
	\Crefname{#1}{#2}{#3}%
}

\newaliastheorem{proposition}{Proposition}{Propositions}
\newaliastheorem{corollary}{Corollary}{Corollaries}
\newaliastheorem{lemma}{Lemma}{Lemmas}
\newaliastheorem{conjecture}{Conjecture}{Conjectures}

\theoremstyle{definition}
\newaliastheorem{definition}{Definition}{Definitions}
\newaliastheorem{example}{Example}{Examples}
\newaliastheorem{counterexample}{Counterexample}{Counterexamples}
\newaliastheorem{construction}{Construction}{Constructions}

\theoremstyle{remark}
\newaliastheorem{remark}{Remark}{Remarks}
\newaliastheorem{intuition}{Intuition}{Intuitions}

\renewcommand{\leq}{\leqslant}
\renewcommand{\geq}{\geqslant}
\renewcommand{\epsilon}{\varepsilon}

\newcommand{\Opens}{\mathop{\mathcal{O}\mspace{-4mu}}}
\newcommand{\Powerset}{\mathop{\mathcal{P}\mspace{-1mu}}}

\newcommand{\Sl}{\mathop{\mathrm{S}\ell}\mspace{-1mu}}
\DeclareMathOperator{\im}{im}
\DeclareMathOperator{\id}{id}
\newcommand{\pr}{\mathrm{pr}}
\newcommand{\sqleq}{\leq}%{\sqsubseteq} 
\newcommand{\sqgeq}{\geq}
\newcommand{\op}{{\mathrm{op}}}
\newcommand{\tensor}{\otimes}
\newcommand{\pftensor}{\odot}

\newcommand{\cc}[1]{\nabla_{\mspace{-3mu}#1}\mspace{2mu}}
\DeclareMathOperator{\swap}{\mathsf{swap}}
\newcommand{\ocomplement}[1]{u_{#1}}
\newcommand{\cocolon}{\mathpunct{}\mathord{:}}
\NewDocumentCommand{\bool}{g}{%
	\IfNoValueTF{#1}
	{\operatorname{\mathsf{B}}}
	{\mathsf{B}_{\mspace{-1mu}#1}}%
}
\DeclareMathOperator{\boolrec}{\varrho}

\newcommand{\tworightarrow}{\begin{tikzcd}[ampersand replacement=\&,cramped, column sep=1em]  \phantom{}\ar[r,two heads] \& \phantom{} \end{tikzcd}}
\let\rightarrowtail\relax\newcommand{\rightarrowtail}{\begin{tikzcd}[ampersand replacement=\&,cramped, column sep=1.2em]  \phantom{}\ar[r,tail] \& \phantom{} \end{tikzcd}}
\usetikzlibrary{arrows.meta,decorations.pathmorphing}
\newcommand{\squigimplies}{%
	\mathrel{%
		\tikz[baseline=-.5ex]{%
		\draw[
			Implies-,
			double,
			double distance=2pt,
			decorate,
			decoration={
			zigzag,
			amplitude=.7pt,
			segment length=1.7mm,
			pre=lineto,
			pre length=4pt
			}
		] (1.4em,0) -- (0,0);
		}%
	}%
}

\newcommand{\up}{\mathop{\uparrow}\mspace{-2mu}}
\newcommand{\down}{\mathop{\downarrow}\mspace{-2mu}}

\newcommand{\preUp}{\scalebox{1.2}[1.05]{\rotatebox[origin=c]{90}{$\rightarrowtriangle$}}}
\newcommand{\preDown}{\scalebox{1.2}[1.05]{\rotatebox[origin=c]{-90}{$\rightarrowtriangle$}}}
\newcommand{\Up}{\text{\preUp}}
\newcommand{\Down}{\text{\preDown}}

\newcommand{\Downsub}[1]{\Down_{\mspace{-2mu} #1}}
\newcommand{\RDown}{R_{\raisebox{0.2ex}[0pt][0pt]{$\scriptstyle\Down$}}}
\newcommand{\KDown}{K_{\raisebox{0.2ex}[0pt][0pt]{$\scriptstyle\Down$}}}
\NewDocumentCommand{\Kdc}{g}{%
	\IfNoValueTF{#1}
	{K_{\mspace{-2mu}\dc}}
	{K_{\mspace{-2mu}\dc{#1}\mspace{-2mu}}}%
}

\newcommand{\arrowRscaled}[3]{%
	\tikz[baseline={#1*#3}, scale=#3]{%
		\begin{scope}[rotate=#2]
			\def\base{.13em}
			\def\tip{.47em}
			\def\half{.2em}
			
			\draw[
			line width=.5pt,
			line cap=round,
			line join=round
			] (-.47em,0) -- (\base,0);
			
			\path[fill, rounded corners=.4pt]
			(\base,\half) -- (\tip,0) -- (\base,-\half) -- cycle;
		\end{scope}
	}%
}
\newcommand{\arrowR}[2][-3pt]{%
	\mathchoice
	{\arrowRscaled{#1}{#2}{1}}      % display
	{\arrowRscaled{#1}{#2}{1}}      % text
	{\arrowRscaled{#1}{#2}{.7}}     % script
	{\arrowRscaled{#1}{#2}{.5}}     % scriptscript
}

\newcommand{\DownR}{\arrowR[-2.6pt]{-90}}

\newcommand{\uc}{\mathop{\vartriangle}\mspace{-2mu}}
\NewDocumentCommand{\dc}{g}{%
	\mathop{\triangledown}\mspace{-2mu}%
	\IfNoValueF{#1}{_{\!\! #1}}%
}

\newcommand{\dcx}{\overline{\mspace{-1.5mu}\dc}}
\newcommand{\Uc}{\blacktriangle}
\newcommand{\Dc}{\blacktriangledown}

\newcommand{\Dcx}{\overline{\Dc\mspace{-.4mu}}}

\newcommand{\cat}[1]{\textnormal{\textbf{#1}}}

\newcommand{\Frm}{\cat{Frm}}
\newcommand{\Loc}{\cat{Loc}}

\newcommand{\dcFrm}{\raisebox{.09ex}{$\dc$}\Frm}
\newcommand{\rLoc}{\cat{rLoc}}

\newcommand{\rocLoc}{\cat{rocLoc}}
\newcommand{\rosLoc}{\cat{rosLoc}}
\newcommand{\rotLoc}{\cat{rotLoc}}

\newcommand{\DLat}{\cat{DLat}}
\newcommand{\CohFrm}{\cat{CohFrm}}

\newcommand{\Bool}{\cat{Bool}}
\newcommand{\Heyt}{\cat{Heyt}}
\newcommand{\EFrm}{\cat{EFrm}}
\newcommand{\ELoc}{\cat{ELoc}}
\newcommand{\HFrm}{\cat{HFrm}}
\newcommand{\OStoneLoc}{\cat{OSLoc}}

\newcommand{\weak}{{\circ}}
\newcommand{\strong}{{\bullet}}
\newcommand{\mode}{\diamond}

\DeclareMathOperator{\idl}{\mathsf{idl}}
\newcommand{\pidl}[1]{\langle #1\rangle}
\DeclareMathOperator{\jidl}{\mathsf{V}\mspace{-1mu}}
\DeclareMathOperator{\cone}{\mathsf{cone}}
\DeclareMathOperator{\rel}{\mathsf{rel}}

\DeclareMathOperator{\townsendB}{\mathscr{B}\mspace{-1mu}}
\DeclareMathOperator{\townsendC}{\mathscr{C}\!}
\DeclareMathOperator{\patch}{\mathsf{patch}}
\DeclareMathOperator{\ClopUp}{\mathsf{ClopUp}}

\newcommand{\calB}{\mathcal{B}}

\makeatletter
\patchcmd{\@maketitle}
	{\global\topskip42\p@}
	{\global\topskip\dimexpr42\p@-1\baselineskip\relax}
	{}{}
\makeatother
\makeatletter
\let\amsart@setauthors\@setauthors % save original AMS author formatting
\makeatother
\makeatletter
\let\@titleblockemails\@empty
\let\@titleblockdate\@empty
\let\@date\@empty
\renewcommand{\email}[2][]{%
	\ifx\@empty\@titleblockemails
	\gdef\@titleblockemails{\href{mailto:#2}{\nolinkurl{#2}}}%
	\else
	\g@addto@macro\@titleblockemails{,\space\href{mailto:#2}{\nolinkurl{#2}}}%
	\fi
}
\renewcommand{\date}[1]{%
	\gdef\@titleblockdate{#1}%
	\global\let\@date\@empty
}
\def\@set@authors@addresses{%
	\amsart@setauthors
	\begingroup
	\par\vspace{.15\baselineskip}
	\centering
	\def\author##1{}% ignore author entries in \addresses list
	\def\\{\protect\linebreak}%
	\def\address##1##2{\par{\footnotesize\itshape ##2\par}}%
	\addresses
	\ifx\@empty\@titleblockemails\else
	\par{\footnotesize\@titleblockemails\par}%
	\fi
	\ifx\@empty\@titleblockdate\else
	\par\smallskip{\footnotesize\@titleblockdate\par}%
	\fi
	\endgroup
	\par
}
\makeatother

\title[Localic Esakia Duality via Conic Frames]{Localic Esakia Duality\\via Conic Frames}
\author{Nesta van der Schaaf}
\address{%
Université Paris-Saclay, CNRS, CentraleSupélec, ENS Paris-Saclay, Inria,\\ Laboratoire Méthodes Formelles, 91190, Gif-sur-Yvette, France
\\
\textnormal{\color{URLColor}\texttt{nesta.van-der-schaaf@inria.fr}}
\\
\textnormal{\url{https://nvds.site/}}
}
\date{26 August 2026}

\begin{document}
	\begin{abstract}
		Esakia duality is the dual equivalence between Heyting algebras and Esakia spaces. However, the traditional proof uses the Prime Ideal Theorem to recover the algebra from its spectrum, a choice principle that is not constructively valid.

		We build on Townsend's localic Priestley duality to describe a fully constructive, localic Esakia duality.
		On the algebraic side, we use the recently introduced Heyting frames as the point-free version of Heyting algebras. On the spatial side, Esakia spaces are modeled by two point-free alternatives.
		First, Esakia locales are defined as a subclass of the ordered Stone locales introduced by Townsend, and it is shown directly that Townsend's equivalence restricts to a duality between Heyting frames and Esakia locales.

		Second, using the theory of conic frames, in which join-preserving closure operators on frames model localic preorders, we introduce Esakia frames as a frame-theoretic analogue of Esakia spaces. It is shown that Esakia frames are equivalent to Heyting frames and dually equivalent to Esakia locales, and that this factorises the restriction of Townsend's equivalence. This yields a fully constructive, localic Esakia duality.
	\end{abstract}
	\maketitle
	\vspace{-5ex}
	\setcounter{tocdepth}{1}
	\tableofcontents
	\vspace{-7ex}

%~~~~~~~~~~~~~~~~~~~~~~~~~~~~~
\section{Introduction}
\label{section:introduction}
\emph{Stone duality}~\cite{stone1936TheoryRepresentationBoolean,stone1938TopologicalRepresentationsDistributive} represents Boolean algebras in terms of clopen subsets of Stone spaces~$X$. This duality is extended by \emph{Priestley duality}~\cite{priestley1970RepresentationDistributiveLattices,priestley1972OrderedTopologicalSpaces}, where arbitrary (bounded) distributive lattices are represented by clopen upper subsets in Priestley spaces~$(X,\leq)$, consisting of a Stone space~$X$ equipped with a partial order~$\leq$ that satisfies the \emph{Priestley separation axiom}~(\cref{definition:priestley space}). In turn, Priestley duality restricts to \emph{Esakia duality}~\cite{esakia1974topologicalKripkemodels}, where Heyting algebras are represented in the same way in terms of Esakia spaces~$(X,\leq)$, which are the Priestley spaces where~$\down U$ is (cl)open for any (cl)open subset~$U\subseteq X$.

In all three cases, starting from an algebra, the corresponding space is obtained by its \emph{spectrum} of prime ideals. However, in the traditional proofs, to recover the algebra from the clopen (upper) sets of its spectrum, the Prime Ideal Theorem is used. This is a choice-principle that is not constructively valid~\cite[\S 10.19]{davey2002IntroductionLatticesOrder}. In this work, we build on Townsend's localic Priestley duality~\cite{townsend1997LocalicPriestleyDuality} to describe a constructive, localic Esakia duality.
In particular, our main contribution is to show that Townsend's duality restricts to a duality between \emph{Heyting frames}~\cite{bezhanishvili2023FrametheoreticPerspectiveEsakia} and \emph{Esakia locales}, and that this duality moreover factorises through a new category of \emph{Esakia frames} using the framework of \emph{conic frames}~\cite{schaaf2026LocalicRelationsOpenCones}.

The advantage of the latter is that it `externalises' the structure of a localic preorder into a join-preserving closure operator on a frame. In this way, internal properties of localic relations, which can be technically difficult to handle, are expressed in frame-theoretic language. This allows for more direct point-free analogues of familiar spatial constructions from Esakia duality. For example, the formula for the Heyting implication on the point-free `clopen upper' elements is directly analogous to the spatial case~(see below, or~\cref{proposition:Kdc Heyting algebra}). Thus the factorisation elucidates how the order-theoretic and algebraic ingredients of the duality interact.

Our results are \emph{constructive} in the sense that they do not rely on the law of the excluded middle, or any non-constructive choice principle.
Throughout, finite means \emph{Kuratowski finite}~\cite[\S D5.4]{johnstone2002Elephant2}, which is the convention followed by~\cite{townsend1997LocalicPriestleyDuality}.
Besides being constructively valid, our work could provide a new spatially motivated but point-free account of Pitts's constructive amalgamation theorem~\cite{pitts1983AmalgamationInterpolation}.
Possible relations to other choice-free or point-free approaches to Stone-type dualities, such as~\cite{korostenski2007LaxProperMaps,bezhanishvili2020ChoiceFreeStone,hartonas2023ChoicefreeTopologicalDuality}, are discussed in~\cref{section:discussion}. 
A more detailed description of this work is as follows.

\vspace{.5\baselineskip}
To describe the localic story, we first need to define point-free analogues on both the algebraic and spatial sides of the duality. (Point-free preliminaries are recalled in~\cref{section:frames}.)
For Priestley duality, on the algebraic side distributive lattices are replaced by \emph{coherent frames}. The constructive equivalence~$\CohFrm\simeq\DLat$ between coherent frames and distributive lattices plays a key role in this work, and is recalled in detail in~\cref{section:coherent and stone frames,section:ideals}~(following~\cite[\S II.3]{johnstone1982StoneSpaces}). On the spatial side, Townsend introduced the notion of an \emph{ordered Stone locale}~$(X,R)$, which is a suitable internalisation of the definition of a Priestley space into the category~$\Loc$ of locales.
In particular,~$X$ is a \emph{Stone locale}, the point-free analogue of a Stone space~(\cref{definition:stone}), and~$R$ is a partial order internal to~$\Loc$~(\cref{definition:localic relation properties}) that satisfies a suitable localic analogue of the Priestley separation axiom~(\cref{definition:ordered stone locale}).
From an ordered Stone locale a distributive lattice~$K_{\down}(\Opens X)$ of `clopen upper' elements is constructed~(\cref{definition:K down}), the ideals of which then form the desired coherent frame.

Conversely, Townsend obtains a Stone locale from a coherent frame~$L$ by taking the frame of ideals of the free Boolean algebra~(\cref{section:free boolean algebras}) generated by the compact elements~$K(L)$ of~$L$, and equips this with the closed localic partial order defined by a certain explicit formula for its open complement~(recalled in~\cref{section:townsends functors}).
Combined, this gives an equivalence~$\CohFrm^\op\simeq \OStoneLoc$ between coherent frames and ordered Stone locales.
Note that while Townsend only defines functors at the point-free level on coherent frames and ordered Stone locales, the fundamental constructions factorise through the distributive lattices~$K_{\down}(\Opens X)$ and~$K(L)$.
In this work we will permit ourselves a similar strategy, explained further below.

The goal is to describe an analogous localic story for Esakia duality.
On the algebraic side, we use the \emph{Heyting frames} introduced in~\cite{bezhanishvili2023FrametheoreticPerspectiveEsakia} as the point-free stand-in for Heyting algebras. This is justified by their result that the equivalence between coherent frames and distributive lattices restricts to an equivalence~$\HFrm\simeq\Heyt$ between Heyting frames and Heyting algebras.
On the spatial side, we introduce two point-free analogues of Esakia spaces: the locale-theoretic \emph{Esakia locales}, and frame-theoretic \emph{Esakia frames}. 

For the localic version, using that Esakia spaces are precisely the Priestley spaces where~$\!\down$ preserves open subsets, we define \emph{Esakia locales} as the ordered Stone locales~$(X,R)$ whose \emph{source map}~$s\colon R\to X$ is open~(\cref{definition:esakia locale}). We prove that Townsend's duality restricts to an equivalence~${\HFrm^\op\simeq\ELoc}$ between Heyting frames and Esakia locales~(\cref{section:townsends functors}).
As in~\cite{bezhanishvili2023FrametheoreticPerspectiveEsakia}, we treat two classes of morphisms. The `weak' class corresponds to the distributive lattice morphisms between Heyting algebras, and continuous monotone functions between Esakia spaces. The `strong' class corresponds to Heyting algebra morphisms, and \emph{p-morphisms} between Esakia spaces. The theory of p-morphisms internal to the category of locales is developed in~\cref{section:internal p morphisms}. All equivalences mentioned hold when interpreted with either type of morphism, respectively.

Our main contributions are the introduction of \emph{Esakia frames}~(\cref{definition:esakia frame}), and a proof that Townsend's restricted duality factorises through the resulting category:
\[
	\HFrm^\op \simeq \EFrm^\op \simeq \ELoc.
\]
As mentioned, the idea is to `externalise' the localic partial order~$R$ of an ordered Stone locale into a join-preserving closure operator~$\Down$ on the underlying frame, using the more general framework of \emph{conic frames} from~\cite{schaaf2026LocalicRelationsOpenCones}, described and slightly generalised in~\cref{section:conic frames}.
We define a category~$\dcFrm$ of \emph{conic frames}~$(L,\dc)$, where~$L$ is a frame equipped with a join-preserving function~${\dc\colon L\to L}$ called its \emph{cone}, modeling the down-closure operator of a would-be localic relation.
There is an adjunction~$\cone\dashv\rel$ between the categories~$\rosLoc$ of locales equipped with open source relations and~$\dcFrm$~(\cref{section:adjunction}). All \emph{closed} localic relations with open source define fixed points of this adjunction~(\cref{section:closed relations are fixed points}), so in particular any localic partial order from an Esakia locale is fully described in terms of the cone~$\Down$ it induces.
We can then translate the axioms of an Esakia locale~$(X,R)$ along the adjunction~$\cone\dashv\rel$ to obtain the notion of an \emph{Esakia frame}~$(E,\dc)$~(\cref{definition:esakia frame}):
\begin{enumerate}[label=(E\arabic*),start=0]
	\item $(E,\dc)$ is a conic frame;
	\hfill ($R$ has open source)
	\item $E$ is a Stone frame;
	\hfill ($X$ is Stone)
	\item $\dc$ is a closure operator;
	\hfill ($R$ is a preorder)
	\item $(E,\dc)$ is closed;
	\hfill ($R$ is closed)
	\item $\Kdc(E)$ Boolean generates~$K(E)$.
	\hfill ($R$ is anti-symmetric)
\end{enumerate}
Here~$\Kdc(E)$ is the frame-theoretic analogue of the lattice~$K_{\down}(\Opens X)$ of `clopen upper' elements~(\cref{definition:Kdc}):
\[
\Kdc(E)
:=
\{c\in K(E):\dc \neg c = \neg c\}.
\]
This is really the Heyting algebra of open elements coming from the closure algebra~$(K(E),\dc)$ in the sense of~\cite{mckinsey1944AlgebraTopology,mckinsey1946ClosedElementsClosure}. See also~\cite[\S 2.2]{esakia2019HeytingAlgebrasDuality}.

Unlike frames and locales, the categories~$\EFrm$ and~$\ELoc$ of Esakia frames and locales are not formally dual, but the equivalence~$\EFrm^\op\simeq\ELoc$ is derived from the fixed point equivalence of the adjunction~${\cone\dashv\rel}$~(\cref{section:equivalence EFrm ELoc}).

To complete the factorisation of Townsend's restricted duality between Heyting frames and Esakia locales, it now suffices to prove~$\HFrm\simeq\EFrm$.
In turn, using~$\Heyt\simeq\HFrm$ from~\cite{bezhanishvili2023FrametheoreticPerspectiveEsakia}, it suffices to construct an equivalence between Esakia frames and Heyting algebras. This is analogous to how Townsend's functors implicitly factorise through distributive lattices. Having externalised the notion of an Esakia locale into frame-theoretic language, where the cone~$\dc$ plays the dominant role, the desired equivalence follows using much the same spatial strategy of Esakia duality~\cite{esakia2019HeytingAlgebrasDuality}. Starting with an Esakia frame~$(E,\dc)$ we prove that~$\Kdc(E)$ defines a Heyting algebra~(\cref{proposition:Kdc Heyting algebra}), where for~$c,d\in \Kdc(E)$ the Heyting implication~$c\to d$ is defined following the spatial intuition of the Heyting implication of clopen upper sets~$C,D\subseteq X$ in an Esakia space:
\[
C\to D
:=
X\setminus \down (C\setminus D)
\qquad\squigimplies\qquad
c\to d
:=
\neg \dc (c\wedge \neg d).
\]

Conversely, starting with a Heyting algebra~$A$, we follow~\cite{townsend1997LocalicPriestleyDuality} and define a Stone frame~$E_A:=\idl(\bool{A})$ as the ideal completion of the free Boolean algebra~$\bool{A}$ generated by the underlying distributive lattice of~$A$. It remains to define a cone~$\dc{A}$ on this frame. For this, note that the frame~$E_A$ is generated by principal ideals of \emph{patches}~$\pidl{u_c\wedge \ell_d}$, where~$c,d\in A$, which model the complements~$C\setminus D$ of clopen upper sets in an Esakia space. Reusing the spatial intuition from the previous equation gives:
\[
\down (C\setminus D)
=
X\setminus (C\to D)
\qquad\squigimplies\qquad
\dc{A}\pidl{u_c\wedge \ell_d}
=
\pidl{\ell_{c\to d}},
\]
and extending~$\dc{A}$ to arbitrary ideals via join-preservation then gives the desired Esakia frame~$(E_A,\dc{A})$~(\cref{proposition:heyting algebra to esakia frame}). Combined, this yields the desired equivalence~$\EFrm\simeq \Heyt$~(\cref{theorem:EFrm Heyt equivalence}), and combining with~\cite{bezhanishvili2023FrametheoreticPerspectiveEsakia} we immediately get~$\EFrm\simeq \HFrm$~(\cref{theorem:EFrm HFrm equivalence}).

All in all, the paper can be summarised in terms of the following commutative diagram of equivalences, where the bottom-left triangle commutes by construction, and the top-right triangle is shown to commute in~\cref{proposition:townsend compatible}:
\newcommand{\summarydiagramlongcolumnsep}{11.5em}
\newcommand{\summarydiagramtightcolumnsep}{.5em}
\newcommand{\summarydiagramlongrowsep}{10em}
\newcommand{\summarydiagramshortrowsep}{1em}
\definecolor{DiagramFunctorColor}{HTML}{3f6c97}
\definecolor{DiagramMapColor}{HTML}{457737}
\newcommand{\summarydiagramfunctorcolor}{DiagramFunctorColor}
\newcommand{\summarydiagrammapcolor}{DiagramMapColor}
\newcommand{\summarydiagrambottomrowshift}{1}
\newcommand{\citedsimeq}[1]{\simeq\mathrlap{\text{\raisebox{.12ex}{\scriptsize\cite{#1}}}}\mspace{12mu}}
\newcommand{\unlinkedcitedsimeq}[1]{\simeq\mathrlap{\text{\raisebox{.12ex}{\scriptsize\mkbibbrackets{\textcolor{CiteColor}{#1}}}}}\mspace{12mu}}
\newcommand{\refsimeq}[1]{\simeq\mathrlap{\text{\raisebox{.12ex}{\scriptsize(\ref{#1})}}}\mspace{12mu}}
\newcommand{\unlinkedrefsimeq}[1]{\simeq\mathrlap{\text{\raisebox{.12ex}{\scriptsize(\textcolor{LinkColor}{\NoHyper\ref{#1}\endNoHyper})}}}\mspace{12mu}}
\[\begin{tikzcd}[
	column sep=0pt,
	row sep=0pt,
	cells={
		/tikz/outer xsep=.7ex,
		/tikz/outer ysep=.7ex
	},
	labels={/tikz/font=\small},
	marking/.append style={/tikz/font=\small},
	execute before arrows={
		\node[anchor=center, outer xsep=.6ex, outer ysep=.6ex] (\tikzcdmatrixname-3-1) at
			([xshift={-\summarydiagrambottomrowshift*(\summarydiagramlongcolumnsep/4)}]\tikzcdmatrixname-3-1.center)
			{$\Heyt^\op$};
		\node[anchor=center, outer xsep=.6ex, outer ysep=.6ex] (\tikzcdmatrixname-3-2) at
			([xshift={-\summarydiagrambottomrowshift*(\summarydiagramlongcolumnsep/4)}]\tikzcdmatrixname-3-2.center)
			{$\EFrm^\op$};
		\node[anchor=center, outer xsep=.6ex, outer ysep=.6ex] (\tikzcdmatrixname-3-3) at
			([xshift={-\summarydiagrambottomrowshift*(\summarydiagramlongcolumnsep/4)}]\tikzcdmatrixname-3-3.center)
			{$\dcFrm^\op$};
	}
]
	{\CohFrm^\op} &[\summarydiagramlongcolumnsep] \OStoneLoc &[\summarydiagramtightcolumnsep] \rLoc \\[\summarydiagramshortrowsep]
	{\HFrm^\op} & \ELoc & \rosLoc \\[\summarydiagramlongrowsep]
	{} & {} & {}
	\arrow[shift left=2, from=1-1, to=1-2]
	\arrow["{\citedsimeq{townsend1997LocalicPriestleyDuality}}"{marking, allow upside down}, draw=none, from=1-1, to=1-2]
	\arrow["{\color{\summarydiagramfunctorcolor}\townsendC}"{marking, font=\normalsize, allow upside down}, shift right=5, draw=none, from=1-1, to=1-2]
	\arrow["{\color{\summarydiagramfunctorcolor}\townsendB}"{marking, font=\normalsize, allow upside down}, shift left=5, draw=none, from=1-1, to=1-2]
	\arrow[shift left=2, from=1-2, to=1-1]
	\arrow["\subseteq"{anchor=center}, draw=none, from=1-2, to=1-3]
	\arrow["\subseteq"{marking, allow upside down}, draw=none, from=2-1, to=1-1]
	\arrow[shift left=2, from=2-1, to=2-2]
	\arrow["{\refsimeq{theorem:HFrm ELoc equivalence}}"{marking, allow upside down}, draw=none, from=2-1, to=2-2]
	\arrow[shift right=2, from=2-1, to=3-1]
	\arrow["{\color{\summarydiagrammapcolor}\idl\Kdc(E)\longmapsfrom(E,\dc)}"{marking, allow upside down}, shift left=5, draw=none, from=2-1, to=3-2]
	\arrow["{\color{\summarydiagrammapcolor} H\longmapsto (E_{K(H)},\dc{K(H)})}"{marking, allow upside down, pos=.45}, shift right=5, draw=none, from=2-1, to=3-2]
	\arrow[shift right=2, from=2-1, to=3-2]
	\arrow["{\unlinkedrefsimeq{theorem:EFrm HFrm equivalence}}"{marking, allow upside down}, draw=none, from=2-1, to=3-2]
	\arrow["{\hyperref[theorem:EFrm HFrm equivalence]{\phantom{\rule[-.78em]{.95em}{1.56em}}}}"{anchor=center, xshift=6pt, yshift=-7.2pt}, draw=none, from=2-1, to=3-2]
	\arrow["\subseteq"{marking, allow upside down}, draw=none, from=2-2, to=1-2]
	\arrow[shift left=2, from=2-2, to=2-1]
	\arrow["\subseteq"{anchor=center}, draw=none, from=2-2, to=2-3]
	\arrow[shift left=2, from=2-2, to=3-2]
	\arrow["\subseteq"{marking, allow upside down}, draw=none, from=2-3, to=1-3]
	\arrow[""{name=0, anchor=center, inner sep=0}, shift right=2, from=2-3, to=3-3]
	\arrow[shift right=2, from=3-1, to=2-1]
	\arrow["{\color{\summarydiagrammapcolor}A\longmapsto \idl(A)}"{marking, allow upside down}, shift right=5, draw=none, from=3-1, to=2-1]
	\arrow["{\color{\summarydiagrammapcolor}K(H)\longmapsfrom H}"{marking, allow upside down}, shift left=5, draw=none, from=3-1, to=2-1]
	\arrow["{\unlinkedcitedsimeq{BCM23}}"{marking, allow upside down}, draw=none, from=3-1, to=2-1]
	\arrow["{\hyperlink{cite.0@bezhanishvili2023FrametheoreticPerspectiveEsakia}{\phantom{\rule[-2.1em]{1.66em}{4.2em}}}}"{anchor=center, xshift=2.8pt, yshift=11pt}, draw=none, from=3-1, to=2-1]
	\arrow["{\color{\summarydiagrammapcolor}\Kdc(E)\longmapsfrom(E,\dc)}"{marking, allow upside down, pos={.5-.08*\summarydiagrambottomrowshift}}, shift left=5, draw=none, from=3-1, to=3-2]
	\arrow[shift right=2, from=3-1, to=3-2]
	\arrow["{\color{\summarydiagrammapcolor}A\longmapsto(E_A,\dc{A})}"{marking, allow upside down, pos={.5-.08*\summarydiagrambottomrowshift}}, shift right=5, draw=none, from=3-1, to=3-2]
	\arrow["{\refsimeq{theorem:EFrm Heyt equivalence}}"{marking, allow upside down}, draw=none, from=3-1, to=3-2]
	\arrow[shift right=2, from=3-2, to=2-1]
	\arrow["{\color{\summarydiagrammapcolor} (E,\dc)\longmapsto (X,R_{\dc})}"{marking, allow upside down}, shift left=5, draw=none, from=3-2, to=2-2]
	\arrow["{\color{\summarydiagrammapcolor} (\Opens X,\Down)\longmapsfrom (X,R)}"{marking, allow upside down}, shift right=5, draw=none, from=3-2, to=2-2]
	\arrow[shift left=2, from=3-2, to=2-2]
\arrow["\simeq"{marking, allow upside down}, draw=none, from=3-2, to=2-2]
\arrow["{\text{\scriptsize(\ref{theorem:ELoc EFrm equivalence})}}"{anchor=center, xshift={-\summarydiagrambottomrowshift*.12em}, yshift=-2.1ex, fill=white, inner sep=.5pt}, draw=none, from=3-2, to=2-2]
	\arrow[shift right=2, from=3-2, to=3-1]
	\arrow["\subseteq"{anchor=center}, draw=none, from=3-2, to=3-3]
	\arrow[""{name=1, anchor=center, inner sep=0}, shift right=2, from=3-3, to=2-3]
	\arrow["{\color{\summarydiagramfunctorcolor}\cone}"{marking, font=\normalsize, allow upside down}, shift left=5, draw=none, from=3-3, to=2-3]
	\arrow["{\color{\summarydiagramfunctorcolor}\rel}"{marking, font=\normalsize, allow upside down}, shift right=5, draw=none, from=3-3, to=2-3]
	\arrow["\dashv"{anchor=center, sloped, allow upside down}, draw=none, from=0, to=1]
	\arrow["{\text{\scriptsize(\ref{theorem:cone rel adjunction})}}"{anchor=center, xshift={-\summarydiagrambottomrowshift*.12em}, yshift=-2.1ex, fill=white, inner sep=.5pt}, draw=none, from=0, to=1]
\end{tikzcd}\]

\bigskip
\noindent\emph{Overview.} The paper is organised as follows. \cref{section:preliminaries,section:localic partial orders} establish preliminaries, and basic definitions around localic partial orders. \cref{section:conic frames} develops the theory of conic frames and proves the adjunction~$\cone\dashv\rel$. The theory of internal p-morphisms is introduced in~\cref{section:internal p morphisms}, which is then used in~\cref{section:esakia locales and esakia frames} to define categories of Esakia locales. The same section introduces Esakia frames and proves the equivalence~$\EFrm^\op\simeq \ELoc$. \cref{section:heyting frames and esakia frames} recalls the definition of Heyting frames, and contains the main construction~$\HFrm\simeq\EFrm$. \cref{section:townsends functors} ties this all together by proving Townsend's duality restricts to~$\HFrm^\op\simeq\ELoc$, and that this factorises through Esakia frames via the previous two equivalences. \cref{section:conclusion} concludes with directions for future research and connections to the literature.

%~~~~~~~~~~~~~~~~~~~~~~~~~~~~~
\section{Preliminaries}
\label{section:lattices and locales}
\label{section:preliminaries}
We briefly recap definitions on lattices, Heyting algebras, coherent frames, sublocales, and Esakia duality. For general order- and lattice-theoretic background we refer to~\cite{davey2002IntroductionLatticesOrder}. A reference covering the necessary category theory is~\cite{maclane1998CategoriesWorkingMathematician}.

%~~~~~~~~~~~~~~~~~~~~~~~~~~
\subsection{Lattice theory}
A \emph{distributive lattice} is a partially ordered set~$A$, whose order we denote by~$\sqleq$, admitting all finite~\emph{meets}~$\wedge$ (greatest lower bounds) and all finite~\emph{joins}~$\vee$ (least upper bounds) such that the distributivity law holds:
\[
x\wedge (y\vee z) = (x\wedge y)\vee (x\wedge z).
\]
Hence~$A$ is always assumed to be \emph{bounded}, meaning it admits a bottom element~$\bot$ and top element~$\top$. A morphism of distributive lattices~$h\colon A\to B$ is a function that preserves finite meets and finite joins. In particular, they are monotone. The category of distributive lattices and their morphisms is denoted~$\DLat$.

Monotone maps~$f\colon P\to Q$ and~$g\colon Q\to P$ between partially ordered sets define a~\emph{(Galois) adjunction}~$f\dashv g$ if it holds for all~$x\in P$ and~$y\in Q$ that:
\[
f(x)\leq y
\quad\iff\quad
x\leq g(y).
\]
An important result we need later is that a monotone function between complete lattices preserves meets/joins iff it admits a left/right adjoint~\cite[\S 7.34]{davey2002IntroductionLatticesOrder}.

A~\emph{Heyting algebra} is a distributive lattice~$A$ such that for every fixed~$x\in A$ the function~$x\wedge -$ has a right adjoint, denoted~$x\to -$, called the \emph{Heyting implication}. Explicitly, it is characterised by the following universal property:
\[
x\wedge y \sqleq z
\quad\iff\quad
y\sqleq x\to z.
\]
A morphism of Heyting algebras~$g\colon A\to B$ is a map of distributive lattices that also preserves the implication:~${g(x\to y)=g(x)\to g(y)}$. We define categories of Heyting algebras in~\cref{definition:Heyt}.

The \emph{Heyting negation} of~$x\in A$ is denoted~$\neg x:= x\to\bot$. A \emph{Boolean algebra} is a Heyting algebra~$A$ such that~$\neg \neg x=x$, or equivalently~$x\vee \neg x=\top$, for all~$x\in A$. In this case~$\neg$ is sometimes called the~(Boolean)~complement.
A morphism of Boolean algebras~$g\colon A\to B$ is a map of distributive lattices that also preserves complements:~$g(\neg x)=\neg g(x)$. We denote the category of Boolean algebras and Boolean morphisms by~$\Bool$.

%~~~~~~~~~~~~~~~~~~~~~~~~~~~~
\subsection{Frames and locales}
\label{section:frames}
For textbook accounts, see~\cite{johnstone1982StoneSpaces,picado2012FramesLocalesTopology}.
% A \emph{complete lattice}~$L$ is a partial order that has all meets and joins. They define bounded lattices, where the bottom is the empty join and the top is the empty meet.
A \emph{frame} is a complete lattice~$L$ in which the infinite distributivity law holds:
\[
x \wedge \bigvee y_i = \bigvee x\wedge y_i.
\]
A morphism of frames~$h\colon L\to M$ is a function that preserves finite meets and arbitrary joins. We denote the category of frames by~$\Frm$. Dually, the category of locales is defined as the opposite:~$\Loc:=\Frm^\op$. Explicitly, a \emph{locale}~$X$ is formally described by its \emph{frame of opens}~$\Opens X$, and a map of locales~$f\colon X\to Y$ is just a morphism of frames~$f^{-1}\colon \Opens Y\to \Opens X$ going the other way.

The infinite distributivity law says that for a fixed element~$x\in L$ the map~$x\wedge -$ is join-preserving, so admits a right adjoint~(\cite[\S 7.34]{davey2002IntroductionLatticesOrder})~$x\to -$, turning~$L$ canonically into a (complete) Heyting algebra.

%~~~~~~~~~~~~~~~~~~~~~~~~~~~~~~~
\subsection{Sublocales}
\label{section:sublocales}
In order to state the definition of a partial order on a locale, we need a localic analogue of what it means to be a subset. \emph{Sublocales} play this role. Structurally, a sublocale of~$X$ is an (equivalence class of an) extremal monomorphism~$A\rightarrowtail X$ in the category~$\Loc$. There are several equivalent ways of unpacking this to frame-theoretic language, for more details on which we refer to~\cite[\S III]{picado2012FramesLocalesTopology}. For our purposes, note that since locales are formally dual to frames, a sublocale of~$X$ will correspond to a \emph{quotient} of the frame~$\Opens X$, and since frames are algebraic objects, these in turn correspond to frame congruences on~$\Opens X$. Recall that a~\emph{frame congruence} on~$L$ is an equivalence relation~$\equiv$ that preserves finite meets and arbitrary joins. Its~\emph{quotient}~$L/{\equiv}$ is the frame of~$\equiv$-equivalence classes. Using the complete structure, each~$\equiv$-equivalence class can be uniquely represented by its maximal element, and so we may identify~$L/{\equiv}$ with the subset of~$L$ containing these maximal elements. This is known as its~\emph{sublocale set}, which in general is a subset~$S\subseteq L$ of a frame that is closed under all meets, and closed under the Heyting implication in that for~$s\in S$ and~$x\in L$ we have~$x\to s\in S$.

For a relation~$\sim$ on a frame~$L$ we denote by~$\langle\sim\rangle$ the smallest frame congruence containing~$\sim$. Rather than computing~$\langle\sim\rangle$ first and then taking its quotient, we use the techniques from~\cite[\S III.11]{picado2012FramesLocalesTopology} and~\cite{moshier2017GeneratingSublocalesSubsets} to describe the quotient~$L/\langle\sim\rangle$ directly in terms of~$\sim$.

\begin{definition}
	\label{definition:saturated}
	% Fix a binary relation ${\sim}\subseteq L\times L$ on a frame.
	An element~$s\in L$ is called \emph{$\sim$-saturated} if:
	\[
	\forall a,b,c\in L: a \sim b \quad\implies\quad \left(a\wedge c \sqleq s \text{ iff } b\wedge c\sqleq s\right).
	\]
	The set of $\sim$-saturated elements is denoted $L/{\sim}$.
\end{definition}

% \begin{remark}[{\cite[\S III.11.2]{picado2012FramesLocalesTopology}}]
	% \label{remark:quotient map}
	The set~$L/{\sim}$ defines a sublocale set of~$L$, and equals the sublocale set corresponding to the frame congruence~$\langle\sim\rangle$.
	We obtain the~\emph{quotient map} of~$\sim$ as
	\[
	\mu_\sim \colon L \longrightarrow L/{\sim};
	\qquad
	\mu_\sim(x)=\bigwedge\{s\in L/{\sim}: x\sqleq s\},
	\]
	taking an element in the frame and producing the smallest $\sim$-saturated element containing it. This is a map of frames. Seen as a map on~$L$ itself,~$\mu_\sim$ is equivalently the \emph{nucleus} of the induced sublocale, so in particular it is a closure operator~(recalled in~\cref{section:preorders in terms of cones}), and~$L/{\sim}$ is precisely the set of fixed points of~$\mu_\sim$.
% \end{remark}

The next theorem shows that~$\mu_\sim$ indeed behaves like a quotient. Say a function~$h$ defined on~$L$ \emph{equalises} the relation~$\sim$ if~$x\sim y$ implies $h(x)=h(y)$.

\begin{theorem}[{\cite[\S III.11.3.1]{picado2012FramesLocalesTopology}}]
	\label{theorem:frame quotient theorem}
	Let $\sim$ be a binary relation on a frame $L$. Then the function $\mu_\sim\colon L\to L/{\sim}$ is a frame map that equalises~$\sim$. If~$h\colon L\to M$ is a frame map equalising $\sim$, there exists a unique frame map $\bar h$ making the following diagram commute:
	\[
	% https://q.uiver.app/#q=WzAsMyxbMCwwLCJMIl0sWzEsMSwiTSJdLFsxLDAsIkwve1xcc2ltfSJdLFswLDIsIlxcbXVfXFxzaW0iXSxbMCwxLCJoIiwyXSxbMiwxLCJcXGV4aXN0cyEgXFxiYXIgaCIsMCx7InN0eWxlIjp7ImJvZHkiOnsibmFtZSI6ImRhc2hlZCJ9fX1dXQ==
	\begin{tikzcd}[cramped, row sep = 1.5em]
		L & {L/{\sim}} \\
		& M.
		\arrow["{\mu_\sim}", from=1-1, to=1-2]
		\arrow["h"', from=1-1, to=2-2]
		\arrow["{\exists! \bar h}", dashed, from=1-2, to=2-2]
	\end{tikzcd}
	\]
	%and moreover for all $s\in L/{\sim}$ we have $\bar h (s)= h(s)$.
\end{theorem}

%~~~~~~~~~~~~~~~~~~~~~~~~~~~~~
\subsection{Coproduct frames}
\label{section:coproduct frames}
A relation on a locale~$X$ is to be a sublocale of the product locale~$X\times X$, whose frame of opens is the \emph{co}product frame~$\Opens X\tensor \Opens X$. If~$L,M$ are frames then their \emph{coproduct}~$L\tensor M$ is the frame freely generated by the symbols~$x\tensor y$, where~$x\in L$ and~$y\in M$, subject to the following equations~\cite{dowker1977SumsCategoryFrames,johnstone1991PreframePresentationsPresent,picado2015NotesProductLocales}:
	\begin{gather*}
		(x_1\tensor y_1)\wedge (x_2\tensor y_2) = (x_1\wedge x_2)\tensor (y_1\wedge y_2),
		\\
		\bigvee (x_i\tensor y) = \left(\bigvee x_i\right)\tensor y,
		\quad\text{and}\quad
		\bigvee(x\tensor y_i)= x\tensor \left(\bigvee y_i\right).
	\end{gather*}
Recall that a~\emph{join-basis} is a subset~$\calB\subseteq L$ of a complete lattice so that every element in~$L$ can be written as a join over elements in~$\calB$. Thus $L\tensor M$ has join-basis $\{x\tensor y: x\in L,y\in M\}$, whose elements are called \emph{(basic) rectangles}. The following result shows how to explicitly calculate the coproduct injections~$\iota_1,\iota_2$ and coproducts of frame maps~$h,k$ on basic rectangles.

\begin{lemma}[{\cite[\S IV.5.5]{picado2012FramesLocalesTopology}}]
	\label{lemma:coproduct structure}
	The coproduct structure is given on rectangles by
	\begin{align*}
		\iota_1(x)&=x\tensor\top,
		&
		\iota_2(y)&=\top\tensor y,
		\\
		(h\tensor k)(x\tensor y)&=h(x)\tensor k(y),
		&
		[h,k](x\tensor y)&=h(x)\wedge k(y).
	\end{align*}
	% In particular, the codiagonal~$[\id_L,\id_L]\colon L\tensor L\to L$ maps~$x\tensor y\mapsto x\wedge y$.
\end{lemma}

%~~~~~~~~~~~~~~~~~~~~~~~
\subsection{Coherent and Stone frames}
\label{section:coherent and stone frames}
For the general theory of coherent frames we refer to~\cite[\S II.3]{johnstone1982StoneSpaces}, from which we collect some definitions and results. For background on Stone frames we refer to~\cite{banaschewski1989UniversalZero,korostenski2007LaxProperMaps,bezhanishvili2020ChoiceFreeStone}. In the following we formulate all definitions frame-theoretically, but take all definitions to apply dually to the localic setting. So, a Stone locale is a locale whose underlying frame is Stone.

Recall that an element~$x$ in a distributive lattice~$A$ is called \emph{complemented} if there exists a~(necessarily unique) element~$y$ such that~${x\wedge y =\bot}$ and~$x\vee y =\top$. The element~$y$ is called the \emph{complement} of~$x$. In fact,~$A$ is a Boolean algebra iff every element is complemented, so we also denote the complement of~$x$ by~$\neg x$.

An element $x\in L$ of a frame is called \emph{compact} if for every subset~$I\subseteq L$ such that~$x\sqleq\bigvee I$, there exists a finite subset~$F\subseteq I$ such that $x\sqleq \bigvee F$.
The set of compact elements of~$L$ is denoted~$K(L)$.

\begin{lemma}[{\cite[\S II.3.2]{johnstone1982StoneSpaces}}]
	\label{lemma:finite joins of compact elements are compact}
	Compact elements are closed under finite joins.
\end{lemma}

\begin{definition}
	\label{definition:coherent}\label{definition:stone}
	A frame~$L$ is called:
	\begin{itemize}
		\item \emph{algebraic} if~$K(L)$ forms a join-basis for~$L$;
		\item \emph{coherent} if it is algebraic and~$K(L)$ is a bounded sublattice~of~$L$;
		\item \emph{compact} if~$\top\in L$ is a compact element;
		\item \emph{zero-dimensional} if its complemented elements form a join-basis;
		\item \emph{Stone} if it is compact and zero-dimensional.
	\end{itemize}
\end{definition}

Hence for any coherent frame~$L$ we get a distributive lattice~$K(L)$, whose lattice operations are the same as those of the ambient frame.
Moreover, a frame~$L$ is Stone iff it is coherent and~$K(L)$ forms a Boolean algebra, in which case the compact elements are precisely the complemented ones.
Thus, a coherent frame~$L$ is Stone iff every compact element is complemented (in which case its complement is necessarily compact)~\cite[p.258]{banaschewski1989UniversalZero}.

%~~~~~~~~~~~~~~~~~~~~~~~~~~~~~~
\subsection{Ideals}
\label{section:ideals}
We briefly recall the constructive equivalence of categories between coherent frames~$\CohFrm$ and distributive lattices~$\DLat$ from~\cite[\S II.3]{johnstone1982StoneSpaces}, since the resulting natural transformations are used in subsequent sections.

An~\emph{ideal} in a distributive lattice~$A$ is a downclosed subset~$I\subseteq A$ closed under finite joins. Denote the set of ideals of~$A$ by~$\idl (A)$, which is a frame under subset inclusion. If~$g\colon A\to B$ is a map of distributive lattices, we get a map of frames~$\idl(g)\colon \idl(A)\to \idl(B)$ defined by
\[
	\idl(g)\colon I\longmapsto \down g[I] = \{b\in B: \exists a\in I: b\sqleq g(a)\}.
\]
Together, this forms a functor~$\idl\colon \DLat\to \Frm$, left adjoint to the forgetful functor~${\Frm\to\DLat}$.

Define on the other hand the category~$\CohFrm$ consisting of coherent frames and \emph{coherent (frame) morphisms}: those frame maps that preserve compact elements. Thus a coherent morphism~$h\colon L\to M$ induces a map of distributive lattices~${h|_{K(L)}\colon K(L)\to K(M)}$, and so we get a functor~${K\colon \CohFrm\to \DLat}$.

The functors~$K$ and~$\idl$ define an equivalence of categories~$\CohFrm\simeq\DLat$ by~\cite[Corollary~III.3.3]{johnstone1982StoneSpaces}. First, there is an isomorphism~${\pidl{-}\colon \id\Longrightarrow K\circ \idl}$ that identifies the elements of a distributive lattice~$A$ with their principal ideals, which are precisely the compact elements in~$\idl(A)$:
\[
	\pidl{-}_A\colon A\longrightarrow K \idl(A);
	\qquad
	a\longmapsto \pidl{a}:= \{b\in A: b\sqleq a\}.
\]
On the other hand, there is a natural isomorphism~$\jidl\colon \idl\circ K\Longrightarrow \id$ that for a coherent frame~$L$ just takes joins over ideals:
\[
	\jidl_L\colon \idl K(L)\longrightarrow L;
	\qquad
	I\longmapsto \bigvee I.
\]

%~~~~~~~~~~~~~~~~~~~~~~~~~~~~~~
\subsection{Priestley and Esakia duality}
\label{section:esakia duality}
We recall some basics of Priestley and Esakia duality that will provide the spatial intuition for the point-free definitions and results to come in later sections. For details we refer to the recent accounts~\cite{esakia2019HeytingAlgebrasDuality,gehrke2024TopologicalDualityDistributive}.
% A brief overview is also in~\cite{landsman2021LogicQuantumMechanics}, and a discussion from the point of view of bitopology is in~\cite{bezhanishvili2010BitopologicalDualityDistributive}.
For this work we follow the conventions of~\cite[\S 2]{bezhanishvili2023FrametheoreticPerspectiveEsakia}, aligning best with those in~\cite{townsend1997LocalicPriestleyDuality}.

In terms of general order-theoretic notation, if~$X$ is a set equipped with a relation~$R$, we define its~\emph{cones}~$\up,\down\colon\Powerset(X)\to \Powerset(X)$ on the powerset of~$X$ by:
\[
\up A
:=
\{ y\in X : \exists a\in A : a R y \},
\quad\text{and}\quad
\down A
:=
\{ x\in X : \exists a\in A : x R a\}.
\]

To define Priestley and Esakia spaces, recall first that a subset of a topological space is called \emph{clopen} if it is both open and closed, and that a \emph{Stone space} is a compact space that is \emph{totally separated}, meaning that clopen subsets separate distinct points~\cite[\S II.4.2]{johnstone1982StoneSpaces}. 

\begin{definition}
\label{definition:priestley space}
	We say a partially ordered space~$(X,\leq)$ is:
	\begin{itemize}
		\item a \emph{Priestley space} if~$X$ is Stone and \emph{Priestley separation} holds: for all $x\not\leq y$ there exists a clopen $U=\up U$ such that~$x\in U$ and~$y\notin U$;
		\item an \emph{Esakia space} if it is a Priestley space and~$\down$~preserves open subsets.
	\end{itemize}
\end{definition}

\begin{remark}
	\label{remark:open source motivation}
	For this work, the condition that~$\down$ preserves open subsets will play a central role. But we need a translation that is more susceptible to internalisation. For that, take a space~$X$ equipped with a relation~$R$, which we in turn equip with the subspace topology inherited from the product space~$X\times X$. It can be seen from a straightforward calculation that if~$s\colon R \to X$ is the \emph{source projection} map~$(x,y)\mapsto x$, then for subsets~$A,B\subseteq X$ we get 
	\[
	s\left[(A\times B)\cap R\right] = A\cap \down B.
	\]
	From this it can be seen that~$s$ is an \emph{open} continuous function, meaning its image map preserves open subsets, iff~$\down$ preserves open subsets. See~\cite[Lemma~3.7]{schaaf2026LocalicRelationsOpenCones} for more details. Hence Esakia spaces are precisely the Priestley spaces with open source projection.
\end{remark}

Important for this work is the spatial motivation behind the construction of a distributive lattice from a Priestley space, and a Heyting algebra from an Esakia~space.

\begin{definition}
	\label{definition:clopup}
	Let~$(X,\leq)$ be a Priestley space. A subset~$C\subseteq X$ is called a \emph{clopen upper set} if~$C$ is clopen and~$\up C= C$. The collection of clopen upper sets is denoted~$\ClopUp(X)$, and forms a distributive lattice under set-theoretic intersections and unions.

	If~$(X,\leq)$ is an Esakia space, then~$\ClopUp(X)$ moreover forms a Heyting algebra, with implication given for clopen upper sets~$A,B\subseteq X$ by~(\cite[\S 2]{bezhanishvili2023FrametheoreticPerspectiveEsakia}):
	\[
	A\to B
	:=
	X\setminus \down (A\setminus B).
	\]
\end{definition}

Lastly, we recall the notion of a~\emph{p-morphism}. The standard class of morphisms between Priestley spaces are the continuous monotone functions, which induce distributive lattice morphisms on the clopen upper sets. However, for Esakia spaces these do not correspond to Heyting morphisms on the algebraic side. Instead we need the following notion.

\begin{definition}
	\label{definition:p morphism pointwise}
	A function~$f\colon (S,R)\to (T,Q)$ between preordered spaces is called a~\emph{p-morphism} if for~$x\in S$ and~$y\in T$:
	\[
		f(x)Q y
		\quad\iff\quad
		\exists x'\in S: xR x' \text{~and~} f(x')=y.
	\]
\end{definition}

Esakia spaces with continuous monotone functions are dual to the category of Heyting algebras, but where the morphisms are maps of distributive lattices. The category of Esakia spaces with p-morphisms is dual to the category of Heyting algebras and Heyting morphisms. In this work, we shall similarly distinguish between these `weak' and `strong' levels of morphisms for the point-free versions of Heyting algebras and Esakia spaces. The internalised, localic analogue of p-morphisms is defined in~\cref{section:internal p morphisms}, and the corresponding weak and strong categories are defined in~\cref{section:ELoc and EFrm,section:HFrm and Heyt}.

%~~~~~~~~~~~~~~~~~~~~~~~~~~~~~
\section{Localic partial orders}
\label{section:localic partial orders}
A relation~$R$ on a space~$X$ is just a subset~$R\subseteq X\times X$ of the product space, usually equipped with the subspace topology. The notion of a localic relation is just the localic analogue of this, using instead sublocales and product locales. In this section we set up the basic theory of localic relations, giving the category~$\rLoc$, and show how the open source assumption gives rise to the induced cone~$\Down$ operation, and a full subcategory~$\rosLoc\subseteq\rLoc$. Important for the present work are localic preorders and partial orders, which are defined following the standard internalisation of relations and their properties into categories with finite limits, see for example~\cite[p.120]{riehl2016CategoryTheoryContext} or~\mbox{\cite[\S D]{schaaf2024TowardsPointFreeSpacetimes}}.

\begin{definition}
	A \emph{(localic) relation} on a locale~$X$ is a sublocale $r\colon R\rightarrowtail X\times X$.
\end{definition}

\begin{definition}
	\label{definition:source and target}
	The \emph{source} and \emph{target maps} of a localic relation $r\colon R\rightarrowtail X\times X$ are defined as the maps $s,t\colon R\to X$ given by $s:=\pr_1\circ r$ and $t:=\pr_2\circ r$, where~$\pr_i\colon X\times X\to X$ are the product projections. For~$U,V\in \Opens X$ we get in particular that~(\cref{lemma:coproduct structure}):
	\begin{gather*}
		r^{-1}(U\tensor V) = s^{-1}(U)\wedge t^{-1}(V),\\
		s^{-1}(U) = r^{-1}(U\tensor \top)
		\qquad\text{and}\qquad
		t^{-1}(V) = r^{-1}(\top\tensor V).
	\end{gather*}
\end{definition}

\begin{example}
	\label{example:diagonal relation}
	The \emph{diagonal sublocale}~${\delta\colon X\rightarrowtail X\times X}$ is the sublocale ~$\Delta$ defined by the unique map such that $\pr_i\circ \delta=\id_X$~\cite[\S IV.5.3]{picado2012FramesLocalesTopology}, with frame map given on rectangles by~(\cref{lemma:coproduct structure}):
	\[
	\delta^{-1}(U\tensor V) = U\wedge V.
	\]
\end{example}

\begin{definition}
	\label{definition:monotone internal}
	A \emph{monotone map} $f\colon (X,R)\to (Y,Q)$ between locales equipped with localic relations consists of a locale map~${f\colon X\to Y}$ satisfying:
	\[
	% https://q.uiver.app/#q=WzAsNCxbMCwwLCJSIl0sWzAsMSwiWFxcdGltZXMgWCJdLFsxLDEsIllcXHRpbWVzIFkuIl0sWzEsMCwiUSJdLFswLDMsIlxcZXhpc3RzIFxcYmFyIGYiLDAseyJzdHlsZSI6eyJib2R5Ijp7Im5hbWUiOiJkYXNoZWQifX19XSxbMCwxLCIiLDIseyJzdHlsZSI6eyJ0YWlsIjp7Im5hbWUiOiJtb25vIn19fV0sWzMsMiwiIiwwLHsic3R5bGUiOnsidGFpbCI6eyJuYW1lIjoibW9ubyJ9fX1dLFsxLDIsImZcXHRpbWVzIGYiLDJdXQ==
	\begin{tikzcd}[cramped]
		R & Q \\
		{X\times X} & {Y\times Y.}
		\arrow["{\exists \bar f}"', dashed, from=1-1, to=1-2]
		\arrow[tail, from=1-1, to=2-1]
		\arrow[tail, from=1-2, to=2-2]
		\arrow["{f\times f}"', from=2-1, to=2-2]
	\end{tikzcd}
	\]
	The map $\bar f$ is unique if it exists. The commutativity of the square above is equivalent to saying that the \emph{source} and \emph{target squares} commute:
	\[\begin{tikzcd}[cramped]
		R & Q \\
		X & Y
		\arrow["{\bar f}", from=1-1, to=1-2]
		\arrow["{s_R}"', from=1-1, to=2-1]
		\arrow["{s_Q}", from=1-2, to=2-2]
		\arrow["f", from=2-1, to=2-2]
	\end{tikzcd}
	\qquad\text{and}\qquad
	\begin{tikzcd}[cramped]
		R & Q \\
		X & {Y.}
		\arrow["{\bar f}", from=1-1, to=1-2]
		\arrow["{t_R}"', from=1-1, to=2-1]
		\arrow["{t_Q}", from=1-2, to=2-2]
		\arrow["f", from=2-1, to=2-2]
	\end{tikzcd}\]
\end{definition}

\begin{remark}
	Note that $f\colon (X,R)\to (Y,Q)$ is monotone iff there is an inclusion~$(f\times f)[R]\subseteq Q$, which in point-set intuition just unpacks to the usual monotonicity condition:~$xRy\implies f(x)Qf(y)$.
\end{remark}

\begin{definition}
	\label{definition:rLoc}
	The category consisting of locales equipped with localic relations and monotone locale maps between them is denoted~$\rLoc$.
\end{definition}

This category forms the general backdrop of localic order theory, including localic preorders and partial orders, which we define next.

%~~~~~~~~~~~~~~~~~~~~~~~~~~~~~~~~~~~~~~~~
\subsection{Partial orders}
Recall that a set-theoretic relation~$R$ on a space is called:
\begin{itemize}
	\item \emph{reflexive} if $xRx$ for all $x$;
	\item \emph{transitive} if $xRy$ and $yRz$ implies $xRz$;
	\item \emph{anti-symmetric} if $xRy$ and $yRx$ implies $x=y$.	
\end{itemize}
These axioms can be internalised using internal composition of relations in any category with finite limits. This is done as follows. Recall the diagonal relation~$\Delta$ from~\cref{example:diagonal relation}. Given any localic relation~$R$ with inclusion~$r=(s,t)$, its opposite~$R^\op$ is the relation defined by the map~$(t,s)$. For localic relations~$R$ and~$Q$ on~$X$, their composition $R\circ Q$ is defined in the standard way via the pullback~$R\times_X^{t_R,s_Q}Q$. See for example~\cite{klein1970RelationsCategories} or~\cite[\S 10.1]{schaaf2026LocalicRelationsOpenCones}.

\begin{samepage}
\begin{definition}
	\label{definition:localic relation properties}
	A localic relation~$R$ is called:
	\begin{itemize}
		\item \emph{reflexive} if $\Delta\subseteq R$;
		\item \emph{transitive} if $R\circ R\subseteq R$;
		\item \emph{anti-symmetric} if $R\cap R^\op \subseteq \Delta$.
	\end{itemize}
	Combined, we say~$R$ is a:
	\begin{itemize}
		\item \emph{preorder} if it is reflexive and transitive;
		\item \emph{partial order} if it is an anti-symmetric preorder.
	\end{itemize}
\end{definition}
\end{samepage}

\begin{remark}
	Localic preorders and partial orders are defined and used in~\cite{townsend1996preframeTechniquesConstructiveLocale}, and localic equivalence relations are studied in~\cite{kock1989GodementTheoremLocales}.
\end{remark}

We will see in~\cref{proposition:Rdc preorder iff closure operator} that for fixed point relations reflexivity and transitivity can be characterised in terms of the cone~$\Down$. Anti-symmetry cannot in general be characterised in terms of cones, but we will see an alternative characterisation in the setting of ordered Stone locales in~\cref{section:esakia locales and esakia frames}.

%~~~~~~~~~~~~~~~~~~~~~~~~~~~~~~~~~~~~~~
\subsection{Open source}
For this work, we particularly need the class of localic relations whose source map is open. This corresponds to the fact that the downwards cone~$\down$ of an Esakia space preserves open subsets~\cref{definition:priestley space}. To state the definition, we first recall the notion of open maps for locales.

\begin{definition}
	\label{definition:open map}
	A map of locales $f\colon X\to Y$ is called \emph{open} if it admits a left adjoint~${f_!\dashv f^{-1}}$ satisfying \emph{Frobenius reciprocity}:
	\[
	f_!\left(U\wedge f^{-1}(V)\right)=f_!(U)\wedge V.
	\]
\end{definition}

\begin{remark}
	Open maps of locales were introduced and studied in~\cite{joyal1984ExtensionGaloisTheory}. Relating back to spatial intuition, that~$f$ is open can equivalently be characterised by saying that its image map preserves open sublocales, or put yet differently, that~$f^{-1}$ is a morphism of complete Heyting algebras~\cite[Proposition~V.1.1]{joyal1984ExtensionGaloisTheory}.
\end{remark}

\begin{definition}
	\label{definition:open cones}
	A localic relation~$R$ is said to have \emph{open source} if the source map~$s\colon R\to X$ is open.
	
	Denote by~$\rosLoc$ the full subcategory of~$\rLoc$ of locales equipped with open source localic relations.
\end{definition}

\begin{remark}
	Defining analogously the class of \emph{open target} localic relations, and denoting their resulting full subcategory by~$\rotLoc$, we get the full subcategory ${\rosLoc\cap\rotLoc=\rocLoc}$ of \emph{open cone} relations studied in~\cite{schaaf2026LocalicRelationsOpenCones}.
\end{remark}

The openness of the source map gives explicitly a left adjoint~$s_!\dashv s^{-1}$ that we use to define the localic down closure operator. For more motivation and intuition on this definition we refer to~\cite[\S 3.1]{schaaf2026LocalicRelationsOpenCones}. The basic idea is that in the spatial setting, the downwards closure cone~$\down$ of a relation on a set is recovered via point-set images and preimages as~$\down A = s\left[t^{-1}[A]\right]$. Using openness, we can adopt this purely in frame-theoretic language.

\begin{definition}
	\label{definition:induced cones}
	Let $X$ be a locale equipped with an open source relation $R$. The~\emph{(induced) cone} is defined as the function
	\[
	\Down := s_!\circ t^{-1}\colon \Opens X\longrightarrow \Opens X.
	\]
\end{definition}

\begin{example}
	\label{example:diagonal relation cones}
	We saw that the diagonal relation~$\Delta\rightarrowtail X\times X$ from~\cref{example:diagonal relation} has source and target map~$\id_X$, which are clearly open. The induced cone is just~$\Down =\id_{\Opens X}$.
\end{example}

\begin{proposition}
	\label{proposition:induced cone is join-preserving}
	If~$R$ is an open source relation, the cone~$\Down$ preserves all joins.
\end{proposition}
\begin{proof}
	This follows since~$t^{-1}$ is a frame map and~$s_!$ is a left adjoint.
\end{proof}

%~~~~~~~~~~~~~~~~~~~~~~~~~~~~~
\section{Conic frames}
\label{section:conic frames}
\emph{Conic frames}~$(L,\uc,\dc)$ were studied in~\cite{schaaf2026LocalicRelationsOpenCones} and provide a frame-theoretic dual to localic relations whose source and target maps are open. Here we briefly describe a slightly generalised version of this theory, using instead pairs of the form~$(L,\dc)$. The resulting localic relation~$R_{\dc}$ has an open source map, but not necessarily an open target map. The conclusion of this section is an adjunction~$\cone\dashv\rel$ between~$\rosLoc$ and the opposite of the category~$\dcFrm$ of conic frames. The general proof strategy here is directly analogous to the two-sided theory, so our exposition here will be light, and we refer to~\cite{schaaf2026LocalicRelationsOpenCones} for more intuition. However, since some proofs need to be genuinely altered, we still present all technical details.

Conic frames should be the frame-theoretic analogue of the induced cone~$\Down$ coming from an open source localic relation. By~\cref{proposition:induced cone is join-preserving} this cone preserves joins, and for the purposes of this paper this is all the properties we axiomatise.

\begin{definition}
	A \emph{conic frame} $(L,\dc)$ is a frame~$L$ equipped with a join-preserving function~$\dc\colon L\to L$, called its \emph{cone}.
\end{definition}

\begin{example}
	By definition, if~$(X,R)$ is a locale~$X$ equipped with an open source localic relation~$R$, then the pair~$(\Opens X,\Down)$ forms a conic frame. We will see in~\cref{theorem:induced relation} that in fact all conic frames are of this form.
\end{example}

\begin{definition}
	\label{definition:conic morphism}
	A \emph{conic morphism} $h\colon (L,\dc)\to (M,\dcx)$ between conic frames consists of a frame homomorphism~$h\colon L\to M$ such that
	\[
	\dcx\circ h \sqleq h\circ \dc.
	\]

% \begin{definition}
	The category of conic frames and morphisms is denoted~$\dcFrm$.
% \end{definition}
\end{definition}

\begin{remark}
		The spatial intuition is that if~$f$ is a function between sets equipped with relations, then~$f$ is monotone iff~$\down \circ f^{-1}\subseteq f^{-1}\circ \down$ holds in the powerset of the domain. This is a standard result, see for example~\cite[Lemma~3.11]{schaaf2026LocalicRelationsOpenCones}.
\end{remark}

%~~~~~~~~~~~~~~~~~~~~~~~~~~~
\subsection{The induced relation}
\label{section:induced relation}
\label{section:Sigma}
The main task of this section is to describe the construction of an open source localic relation~$R_{\dc}$ from a conic frame~$(L,\dc)$. Since this should be a sublocale on the product locale, the relation~$R_{\dc}$ will be represented by a frame congruence on the coproduct frame~${L\tensor L}$~(recall~\cref{section:sublocales}). In turn, we can generate such a frame congruence by a more elementary binary relation~$\sim$, in this case defined on the basic rectangles. The spatial motivation for~$\sim$ follows from the point-set fact that for rectangles in the product space~${A\times B\cap R = C\times D\cap R}$ iff~${\up A\cap B=\up C\cap D}$ and~${A\cap \down B = C\cap \down D}$. The former equation describes the frame congruence~(on rectangles) corresponding to the sublocale induced by the subspace~$R$, and the latter condition is expressed purely in terms of the cones of~$R$, and hence has an analogue in any conic frame. This leads to the following definition for~$\sim$. See~\cite[\S 3.2]{schaaf2026LocalicRelationsOpenCones} for more details on the spatial motivation.
To help notation, we define the \emph{reduced cone} of~$\dc$ as the function
\[
\Dc\colon L\times L\longrightarrow L;
\qquad
(x,y)\longmapsto x\wedge \dc y.
\]

\begin{definition}
	\label{definition:sim}
	For any conic frame~$(L,\dc)$, the \emph{generating relation}~$\sim$ is the relation on the coproduct~${L\tensor L}$ defined by
	\[
	\forall x,y\in L: x\tensor y \sim \Dc(x,y)\tensor y.
	\] 
\end{definition}

Suppose now that we have a locale~$X$ equipped with the structure of a conic frame~$(\Opens X,\dc)$. We use the~$\sim$-saturated elements from~\cite{moshier2017GeneratingSublocalesSubsets} described in~\cref{section:sublocales} to define the induced relation.

\begin{definition}
	\label{definition:induced relation}
	The \emph{induced relation} is the sublocale~$R_{\dc}$ generated by~$\sim$:
	\[
	\Opens R_{\dc} := \Opens X\tensor \Opens X/{\sim}.
	\]
	Denote its sublocale inclusion by~$r_{\dc}\colon R_{\dc}\rightarrowtail X\times X$, with~$r_{\dc}^{-1}=\mu_\sim$, and the corresponding source and target maps by~$s_{\dc},t_{\dc}\colon R_{\dc}\to X$.
\end{definition}

The rest of this section stands in service of proving~\cref{theorem:induced relation}, showing that~$R_{\dc}$ has open source and that its induced cone~$\Down$ precisely recovers~$\dc$. One of the main ingredients is the construction of a left adjoint~$s_!\dashv s_{\dc}^{-1}$ satisfying Frobenius reciprocity. This map~${s_!\colon \Opens R_{\dc}\to \Opens X}$ should mirror the pointwise intuition that~$s\left[(A\times B)\cap R\right]=A\cap \down B$~(\cref{remark:open source motivation}), so we expect:
\[
s_!(x\tensor y)= x\wedge \dc y.
\]
For the rest of this section, we fix a conic frame~$(L,\dc)$, and establish the technical results necessary to prove the existence of the desired~$s_!$.

\begin{lemma}
	\label{lemma:Sigma exists}
	There is a join-preserving~map~$\Sigma\colon L\tensor L\to L$ with~${\Sigma(x\tensor y)=\Dc(x,y)}$.
\end{lemma}
\begin{proof}
	By~\cite[\S II.2.12]{johnstone1982StoneSpaces} we know that the frame coproduct~$L\tensor L$ is really their suplattice tensor product. See also~\cite[\S 1]{johnstone1991PreframePresentationsPresent}. Since~$\dc$ preserves joins we see that~$\Dc\colon L\times L\to L$ is a suplattice bihomomorphism, so extends to the desired join-preserving map~$\Sigma\colon L\tensor L\to L$.
\end{proof}

It should be clear that~$s_!$ is to be the restriction of the map~$\Sigma$ to the sublocale set~$\Opens R_{\dc}$. On the other hand, we need to ensure that its right adjoint agrees with the preimage map~$s_{\dc}^{-1}$. For this, first we calculate the right adjoint of~$\Sigma$.

\begin{definition}
	For~$z\in L$, its \emph{separating open} in~$L\tensor L$ is defined as
	\[
	O^\Dc_z:=\bigvee\{x\tensor y: \Dc(x,y)\sqleq z \}.
	\]
\end{definition}

\begin{lemma}
	The right adjoint~$\Sigma\dashv \sigma$ is given by~$\sigma\colon z\mapsto O^\Dc_z$.
\end{lemma}
\begin{proof}
	Since~$\Sigma$ preserves joins, there exists a right adjoint~$\Sigma\dashv \sigma$ determined by the general formula~(see for example~\cite[\S 7.34]{davey2002IntroductionLatticesOrder})
	\[
	\sigma(z) = \bigvee \{o\in L\tensor L: \Sigma(o)\sqleq z\},
	\] 
	and using the rectangle join-basis and~$\Sigma(x\tensor y)=\Dc(x,y)$ we see the right-hand side is equal to the separating open~$O^\Dc_z$. 
\end{proof}

\begin{corollary}
	\label{corollary:separating open rectangle adjunction}
	We have~$x\tensor y\sqleq O^\Dc_z$ iff~$\Dc(x,y)\sqleq z$.
\end{corollary}
	
\begin{lemma}
	\label{lemma:separating opens are sim saturated}
	The separating opens~$O^\Dc_z$ are~$\sim$-saturated.
\end{lemma}
\begin{proof}
	We need to show that for any~$x,y,z\in L$ and arbitrary~$o\in L\tensor L$ we have
	\[
	x\tensor y \wedge o \sqleq O^\Dc_z
	\iff
	\Dc(x,y)\tensor y \wedge o \sqleq O^\Dc_z.
	\]
	Since~$\Dc(x,y)\sqleq x$, the `only if' direction is immediate. For the converse, suppose~$\Dc(x,y)\tensor y \wedge o \sqleq O^\Dc_z$. Write~$o$ as the join over basic rectangles~$\bigvee x_i\tensor y_i$,%
	\pagebreak[3]
	so that~$x\tensor y \wedge o = \bigvee (x\wedge x_i)\tensor (y\wedge y_i)$, from which we see it suffices to show $(x\wedge x_i)\tensor (y\wedge y_i)\sqleq O^\Dc_z$ for all~$i$, which by~\cref{corollary:separating open rectangle adjunction} is equivalent to~$\Dc(x\wedge x_i,y\wedge y_i)\sqleq z$. 
	Note that for each~$i$ we get
	\[
	(\Dc(x,y)\wedge x_i)\tensor (y\wedge y_i) \sqleq \Dc(x,y)\tensor y \wedge o \sqleq O^\Dc_z,
	\]
	which again by~\cref{corollary:separating open rectangle adjunction} is equivalent to~$\Dc\left(\Dc(x,y)\wedge x_i,y\wedge y_i\right)\sqleq z$. Using that~$\dc$ is monotone, the term on the left simplifies to
	\begin{align*}
		\Dc\left(\Dc(x,y)\wedge x_i,y\wedge y_i\right)
		&=
		\Dc(x,y)\wedge x_i \wedge \dc(y\wedge y_i)
		\\&=
		x\wedge \dc y \wedge x_i \wedge \dc(y\wedge y_i)
		\\&=
		x\wedge x_i\wedge \dc(y\wedge y_i)
		\\&=
		\Dc(x\wedge x_i,y\wedge y_i),
	\end{align*}
	and the desired inclusion follows.
\end{proof}

\begin{lemma}
	\label{lemma:mu recovers separating opens}
	We have~$\mu_\sim(z\tensor \top)= s_{\dc}^{-1}(z) = O^\Dc_z$.
\end{lemma}
\begin{proof}
	Note that~$\Dc(z,\top) =z\wedge \dc\top\sqleq z$, so~$z\tensor\top \sqleq O^\Dc_z$. Since the separating opens~$O^\Dc_z$ are~$\sim$-saturated~(\cref{lemma:separating opens are sim saturated}) we get~$\mu_\sim(z\tensor \top)\sqleq \mu_\sim(O^\Dc_z)=O^\Dc_z$.
	 
	Conversely, for any rectangle~$x\tensor y$ with $\Dc(x,y)\sqleq z$, we get
	\[
	x\tensor y \sim \Dc(x,y)\tensor y \sqleq z\tensor \top \sqleq \mu_\sim(z\tensor \top).
	\]
	Since~$\mu_\sim(z\tensor \top)$ is~$\sim$-saturated it follows that~$x\tensor y\sqleq \mu_\sim(z\tensor\top)$, and taking joins over~$x\tensor y$ we get~$O^\Dc_z\sqleq \mu_\sim(z\tensor\top)$.
\end{proof}

Thus~$s_{\dc}^{-1}$ is just the codomain restriction of~$\sigma$, hence preserves all meets, and so admits a left adjoint~$s_!\dashv s_{\dc}^{-1}$~(\cite[\S 7.34]{davey2002IntroductionLatticesOrder}). We claim this is the desired map.

\begin{lemma}
	\label{lemma:left adjoint is Sigma up to mu}
	We have~$s_!\circ \mu_\sim = \Sigma$.
\end{lemma}
\begin{proof}
	Write the subset inclusion~$i\colon \Opens R_{\dc}\hookrightarrow L\tensor L$ of saturated elements, forming the right adjoint~$\mu_\sim\dashv i$. For elements~$z\in L$ and~$o\in L\tensor L$ then consider:
	\begin{align*}
		s_!\circ \mu_\sim (o)\sqleq z
		&\iff
		\mu_\sim(o) \sqleq s_{\dc}^{-1}(z)
		&&(s_!\dashv s_{\dc}^{-1})
		\\&\iff
		o\sqleq i\circ s_{\dc}^{-1}(z)
		&&(\mu_\sim\dashv i)
		\\&\iff
		o\sqleq \sigma(z)
		&&(i\circ s_{\dc}^{-1}=\sigma)
		\\&\iff 
		\Sigma(o)\sqleq z,
		&&(\Sigma\dashv \sigma)
	\end{align*}
	from which the desired equation follows by taking meets over~$z$.
\end{proof}

\begin{corollary}
	\label{corollary:left adjoint on rectangles}
	For all~$x,y\in L$ we have~$s_!\circ \mu_\sim(x\tensor y)=\Dc(x,y)$.
\end{corollary}

\begin{lemma}
	\label{lemma:frobenius}
	The pair~$s_!\dashv s_{\dc}^{-1}$ satisfies Frobenius reciprocity.
\end{lemma}
\begin{proof}
	We need to show for all~$o\in \Opens R_{\dc}$ and~$z\in L$ that
	\[
	s_!\left(o\wedge s_{\dc}^{-1}(z)\right) = s_!(o)\wedge z.
	\]
	Both sides preserve joins in both arguments, so it suffices to check equality in the case that~$o$ is a basic rectangle~$\mu_\sim(x\tensor y)$ for some~$x,y\in L$. Since~$\mu_\sim$ preserves binary meets we get
	\[
	o\wedge s_{\dc}^{-1}(z) = \mu_\sim(x\tensor y)\wedge \mu_\sim(z\tensor\top) = \mu_\sim((x\wedge z)\tensor y).
	\]
	Applying~$s_!$ on both sides and using~\cref{corollary:left adjoint on rectangles} gives the desired equation:
	\[
	s_! \mu_\sim((x\wedge z)\tensor y) = \Dc(x\wedge z,y) = \Dc(x,y)\wedge z = s_! \mu_\sim(x\tensor y)\wedge z.
	\qedhere
	\]
\end{proof}
% \begin{proof}
% 	We need to show for all~$O\in \Opens R_{\dc}$ and~$U\in \Opens X$ that
% 	\[
% 	s_!\left(O\wedge s_{\dc}^{-1}(U)\right) = s_!(O)\wedge U.
% 	\]
% 	Both sides preserve joins in both arguments, so it suffices to check equality in the case that~$O$ is a basic rectangle~$\mu_\sim(A\tensor B)$ for some~$A,B\in\Opens X$. Since~$\mu_\sim$ preserves binary meets we get
% 	\[
% 	O\wedge s_{\dc}^{-1}(U) = \mu_\sim(A\tensor B)\wedge \mu_\sim(U\tensor\top) = \mu_\sim((A\wedge U)\tensor B).
% 	\]
% 	Applying~$s_!$ on both sides and using~\cref{corollary:left adjoint on rectangles} gives the desired equation:
% 	\[
% 	s_!\mu_\sim((A\wedge U)\tensor B) = \Dc(A\wedge U,B) = \Dc(A,B)\wedge U = s_!\mu_\sim(A\tensor B)\wedge U.
% 	\qedhere
% 	\]
% \end{proof}

\begin{theorem}
	\label{theorem:induced relation}
	The relation~$R_{\dc}$ has open source with induced cone~$\Down = \dc$.
\end{theorem}
\begin{proof}
	That~$R_{\dc}$ has open source follows from~\cref{lemma:frobenius}, and its cone can be calculated using~\cref{corollary:left adjoint on rectangles}:
	\[
	\Down x = s_! t_{\dc}^{-1}(x) = s_! \mu_\sim(\top\tensor x) = \Dc(\top,x) = \dc x.
	\qedhere
	\]
\end{proof}

%~~~~~~~~~~~~~~~~~~~~~~~~
\subsection{Induced relation of~$\Down$}
We can apply the general construction of~$R_{\dc}$ in~\cref{section:induced relation} to the specific case where~$\dc=\Down$ is the cone coming from some open source relation~$R$. The resulting relation is denoted~$\RDown$. Here we show that~$R\subseteq \RDown$.

For this section, fix an open source relation~$R$ on~$X$, with inclusion map~$r$ and source and target maps~$s,t$. By~\cref{proposition:induced cone is join-preserving} the cone~$\Down$ defines the structure of a conic frame~$(\Opens X,\Down)$. Denote its reduced cone by~$\DownR\colon \Opens X\times \Opens X\to \Opens X$, and the corresponding generating relation on~$\Opens X\tensor \Opens X$ by~$\sim$. From~\cref{theorem:induced relation} we get an open source relation~$\RDown$ on~$X$ defined by the map~$\mu_\sim$. First we prove some lemmas relating the primitive structure of~$R$ to the induced structure~$\Down$ and~$\sim$.

\begin{lemma}
	\label{lemma:reduced cone from s r}
	We have $s_!r^{-1}(U\tensor V)= \DownR(U,V)$.
\end{lemma}
\begin{proof}
	This follows straightforwardly from Frobenius reciprocity of~$s_!\dashv s^{-1}$:
	\[
	s_! r^{-1}(U\tensor V) = s_!\left(s^{-1}(U)\wedge t^{-1}(V)\right)
	= s_!t^{-1}(V)\wedge U =\DownR(U,V).
	\qedhere
	\]
\end{proof}

\begin{lemma}
	\label{lemma:r equalises sim}
	We have $r^{-1}(U\tensor V) = r^{-1}\left(\DownR(U,V)\tensor V\right)$.
\end{lemma}
\begin{proof}
	Since $s_!\dashv s^{-1}$, we have $O\sqleq s^{-1}s_!(O)$ for all $O\in\Opens R$, and so:
	\[
	s^{-1}(U)\wedge t^{-1}(V)
	\sqleq
	s^{-1}s_!\bigl(s^{-1}(U)\wedge t^{-1}(V)\bigr)
	=
	s^{-1}(\DownR(U,V)),
	\]
	where we used that $r^{-1}(U\tensor V)=s^{-1}(U)\wedge t^{-1}(V)$~(\cref{lemma:coproduct structure}) together with the previous~\cref{lemma:reduced cone from s r}. From this we get using~$\DownR(U,V)\sqleq U$ that:
	\begin{align*}
		r^{-1}(U\tensor V)
		&=
		s^{-1}(U)\wedge t^{-1}(V)
		\\&=
		s^{-1}(\DownR(U,V))\wedge s^{-1}(U)\wedge t^{-1}(V)
		\\&=
		s^{-1}(\DownR(U,V))\wedge t^{-1}(V)
		\\&=
		r^{-1}(\DownR(U,V)\tensor V).
		\qedhere
	\end{align*}
\end{proof}

\begin{proposition}
	\label{proposition:R in RDown}
	If~$R$ is a localic relation with open source, then~$R\subseteq \RDown$.
\end{proposition}
\begin{proof}
	By~\cref{lemma:r equalises sim} the frame map $r^{-1}\colon \Opens X\tensor \Opens X\to \Opens R$ equalises the generating relation~$\sim$. Therefore, by the frame quotient~\cref{theorem:frame quotient theorem}, there exists a unique frame map $i^{-1}\colon \Opens X\tensor \Opens X/{\sim}\to \Opens R$ such that $i^{-1}\circ \mu_\sim = r^{-1}$, exhibiting the desired sublocale inclusion~$i\colon R\rightarrowtail \RDown$.
\end{proof}

%~~~~~~~~~~~~~~~~~~~~~~~~
\subsection{The adjunction}
\label{section:adjunction}
We put together the maps~${R\mapsto \Down}$ and~${\dc\mapsto R_{\dc}}$, make them functorial, and show they define an adjunction~${\cone\dashv\rel}$. Lastly, we show that closed localic relations with open source maps form fixed points of the adjunction.

The key technical lemma is the following result that shows conic morphisms are precisely those that preserve the generating relations. More precisely, take conic frames~$(L,\dc)$ and~$(M,\dcx)$. Denote by~$\simeq$ the equivalence relation generated by~$\sim$ coming from~$\dcx$ on~$M$. We then have the following.

\begin{lemma}
	\label{lemma:conic morphism iff preserves sim}
	A frame map $h\colon (L,\dc)\to (M,\dcx)$ is conic iff~for~all~${x,y\in L}$:
	\[
	h(x)\tensor h(y)\simeq  h(\Dc(x,y))\tensor h(y).
	\]
\end{lemma}
\begin{proof}
	Suppose first that~$h$ is a conic morphism, fix elements~$x,y\in L$, and abbreviate~$\Dc=\Dc(x,y)$. We claim that under the reduced cone of~$\dcx$ we have
	\[
	\Dcx(h(x),h(y))=\Dcx(h(\Dc),h(y)).
	\]
	To see this, first observe that~$\Dc\sqleq x$, so by monotonicity we get the~$\sqgeq$ inclusion. For the converse inclusion, we use~$\dcx\circ h\sqleq h\circ \dc$ to get
	\[
	\Dcx(h(x),h(y)):= h(x)\wedge \dcx h(y) \sqleq h(x)\wedge h(\dc y) =: h(\Dc),
	\]
	which together with the trivial inclusion~$\Dcx(h(x),h(y))\sqleq \dcx h(y)$ gives
	\[
	\Dcx(h(x),h(y))\sqleq h(\Dc)\wedge \dcx h(y) =: \Dcx(h(\Dc),h(y)),
	\]
	proving the claim. Applying this to the definition of the generating relation~$\sim$ of~$(M,\dcx)$ we find
	\[
	h(x)\tensor h(y)
	\sim
	\Dcx(h(x),h(y))\tensor h(y)
	=
	\Dcx(h(\Dc),h(y))\tensor h(y)
	\simeq
	h(\Dc)\tensor h(y),
	\]
	implying the relation under~$\simeq$, as desired.
	
	Conversely, suppose the displayed condition holds. Setting~$x=\top$ gives in particular that~$\top\tensor h(y)\simeq h(\dc y)\tensor h(y)$. Denote by~$\mu_\sim$ the quotient map and by~$s_!\dashv s_{\dcx}^{-1}$ the left adjoint corresponding to the induced relation of~$\dcx$. Then calculate:
	\begin{align*}
		\dcx h(y)
		&=:
		\Dcx(\top,h(y))
		\\&=
		s_!\mu_\sim (\top\tensor h(y))
		&&(\text{\cref{corollary:left adjoint on rectangles}})
		\\&=
		s_!\mu_\sim(h(\dc y)\tensor h(y))
		&&(\text{$\mu_\sim$ equalises $\simeq$})
		\\&=
		\Dcx(h(\dc y),h(y))
		&&(\text{\cref{corollary:left adjoint on rectangles}})
		\\&:=
		h(\dc y)\wedge \dcx h(y),
	\end{align*}
	giving the inclusion $\dcx\circ h\sqleq h\circ \dc$, as desired.
\end{proof}

We now construct the desired functors and prove that they are adjoint. Note that in terms of functoriality there is nothing to check, since on morphisms we are really just taking the opposite categories. All there is to prove is well-definedness on morphisms.

\begin{proposition}
	\label{proposition:functor ucdcFrm to rosLoc}
	There is a functor
	\begin{align*}
		\rel\colon \dcFrm^\op &\longrightarrow \rosLoc;
		\\
		(\Opens X,\dc)&\longmapsto (X, R_{\dc});
		\\
		f^{-1} &\longmapsto f.
	\end{align*}
\end{proposition}
\begin{proof}
	That this is well-defined on objects is~\cref{theorem:induced relation}. Now take a map of locales $f\colon X\to Y$, where $\Opens X$ and $\Opens Y$ are equipped with join-preserving cones~$\dc$ and~$\dcx$, respectively, and denote the induced relations by~$R_{\dc}$ and~$R_{\dcx}$. We need to construct a frame map $\bar f^{-1}\colon \Opens R_{\dcx}\to \Opens R_{\dc}$ such that
	\[
	\bar f^{-1}\circ r_{\dcx}^{-1} = r_{\dc}^{-1}\circ (f\times f)^{-1}.
	\]
	That $f^{-1}$ is a conic morphism means equivalently by~\cref{lemma:conic morphism iff preserves sim} that
	\[
	f^{-1}(V)\tensor f^{-1}(W) \simeq f^{-1}(\Dcx(V,W))\tensor f^{-1}(W)
	\]
	for all~$V,W\in\Opens Y$, where~$\simeq$ is the equivalence relation on~$\Opens X\tensor \Opens X$ generated by~$\dc$. Now~$r_{\dc}^{-1}=\mu_\sim$ equalises~$\simeq$, in turn~$r_{\dc}^{-1}\circ(f\times f)^{-1}$ equalises the generating relation of~$\dcx$, so we get the desired~$\bar{f}^{-1}$ by the frame quotient~\cref{theorem:frame quotient theorem}.
\end{proof}

\begin{proposition}
	\label{proposition:functor rosLoc to ucdcFrm}
	There is a functor
	\begin{align*}
		\cone\colon \rosLoc &\longrightarrow \dcFrm^\op;
		\\
		(X,R)&\longmapsto (\Opens X,\Down);
		\\
		f &\longmapsto f^{-1}.
	\end{align*}
\end{proposition}
\begin{proof}
	That this is well-defined on objects is just~\cref{proposition:induced cone is join-preserving}. We are left to show that if~$f\colon (X,R)\to (Y,Q)$ is internally monotone, then its underlying frame map is a conic morphism. Monotonicity gives~$\bar f\colon R\to Q$ that at the level of frames satisfies the commutative square~$r^{-1} \circ(f\times f)^{-1} = \bar f^{-1} \circ q^{-1}$. For~${V\in\Opens Y}$ we get by~\cref{lemma:r equalises sim} in particular that~${q^{-1}(\top\tensor V)=q^{-1}(\Downsub{Q}V\tensor V)}$. Applying~$\bar f^{-1}$ and simplifying using~\cref{lemma:coproduct structure} and the commutative square gives:
	\[
	r^{-1}\left(\top \tensor f^{-1}(V)\right) = r^{-1}\left(f^{-1}(\Downsub{Q}V)\tensor f^{-1}(V)\right).
	\]
	In turn applying the left adjoint~$s_!$ of the source map of~$R$ and using~\cref{lemma:reduced cone from s r} gives:
	\[
	\Downsub{R}f^{-1}(V) = f^{-1}(\Downsub{Q}V)\wedge \Downsub{R}f^{-1}(V),
	\]
	which is precisely the conic morphism condition of~$f^{-1}$.
\end{proof}

\begin{theorem}
	\label{theorem:cone rel adjunction}
	There is an adjunction with identity counit
		\[
	\begin{tikzcd}[cramped,column sep=1.2cm]
		\rosLoc & {\dcFrm^\op}.
		\arrow[""{name=0, anchor=center, inner sep=0}, "{\cone}", shift left=1.6, from=1-1, to=1-2]
		\arrow[""{name=1, anchor=center, inner sep=0}, "{\rel}", shift left=1.6, from=1-2, to=1-1]
		\arrow["\dashv"{anchor=center, rotate=-90}, draw=none, from=0, to=1]
	\end{tikzcd}
	\]
\end{theorem}
\begin{proof}
	Given~\cref{theorem:induced relation}, the counit~$\epsilon\colon \cone\circ\rel\Longrightarrow \id$ is taken as the identity natural transformation. The unit~$\eta\colon \id\Longrightarrow \rel\circ\cone$ is given by the identity locale maps~$\id_X\colon(X,R)\to (X,\RDown)$, which is monotone by~\cref{proposition:R in RDown}. Since the components of the unit and counit are both defined in terms of identity maps, the triangle identities are straightforwardly verified, as in~\cite[Theorem~8.3]{schaaf2026LocalicRelationsOpenCones}.
\end{proof}

\begin{corollary}
	\label{corollary:universality Rdc}
	We have~$R\subseteq R_{\dc}$ iff~$\Down\sqleq \dc$.
\end{corollary}
\begin{proof}
	The inclusion $R\subseteq R_{\dc}$ holds iff $\id_X\colon (X,R)\to (X,R_{\dc})=\rel(\Opens X,\dc)$ is monotone, and by $\cone\dashv \rel$ this holds in turn iff $\id_{\Opens X}\colon (\Opens X,\dc)\to (\Opens X,\Down)$ is a conic morphism, which is equivalent to~$\Down\sqleq \dc$.
\end{proof}

%~~~~~~~~~~~~~~~~~~~~~~~~~~~~~~~~~~~~
\subsection{Closed relations are fixed points}
\label{section:closed relations are fixed points}
Recall that the \emph{fixed points} of an adjunction are those objects in the respective categories whose component in the unit or counit are isomorphisms. Since for~$\cone\dashv\rel$ the counit is the identity, all conic frames are fixed points. On the other hand, an open source relation~$R$ defines a fixed point precisely when~$\id_X\colon (X,R)\to (X,\RDown)$ is an isomorphism, which holds just when~$R\cong \RDown$ as subobjects of~$X\times X$. A general criterion for being a fixed point is unknown, but in this section we show that it suffices to assume~$R$ is \emph{closed}. First, we also record the following important class of examples.

\begin{example}
	\label{example:diagonals are fixed points}
	All diagonal relations~$\Delta$ are fixed points of~$\cone\dashv\rel$, and their corresponding cone is the identity relation on the underlying frame. The proof is analogous to~\cite[Example~9.3]{schaaf2026LocalicRelationsOpenCones}.
\end{example}

\begin{remark}
\label{remark:closed sublocales}
Recall from~\cite[\S III.6.1]{picado2012FramesLocalesTopology} that the \emph{closed} sublocales of~$X$ are those of the form~${\up u = \{o\in\Opens X : u\sqleq o\}}$ for~$u\in \Opens X$, where the frame map of the sublocale inclusion is given by~$o\mapsto o\vee u$, which has corresponding congruence~$\cc{u}$ given by
\[
	x\cc{u} y
	\quad\iff\quad
	x\vee u = y\vee u.
\]
In fact, the congruence~$\cc{u}$ is precisely the principal one generated by the singleton relation~$\{(u,\bot)\}$. In particular, a localic relation~$R$ on~$X$ is \emph{closed} when there exists~$u\in \Opens X\tensor \Opens X$ such that the inclusion~$r$ is given (up to unique isomorphism) as~$r^{-1}(o)=o\vee u$. The element~$u$ is called the \emph{open complement} of~$R$, also generally denoted by~$\ocomplement{R}$.
\end{remark}

\begin{proposition}
	\label{proposition:closed relation fixed point}
	If~$R$ is a closed relation with open source, then~$R\cong \RDown$.
\end{proposition}
\begin{proof}
	Suppose that $R$ is the closed sublocale represented by~$u\in\Opens X\tensor\Opens X$, and hence corresponding to the frame congruence~$\langle(u,\bot)\rangle$ generated by~${\{(u,\bot)\}}$. Since~$R\subseteq \RDown$~(\cref{proposition:R in RDown}) it suffices to show $\langle(u,\bot)\rangle \subseteq\langle\sim\rangle$.
	
	The element~$u$ is the bottom in~$\Opens R=\up u$, so for any rectangle~${A\tensor B\sqleq u}$ we get~$r^{-1}(A\tensor B)= u=r^{-1}(\bot)$. Applying~$s_!$ and using~\cref{lemma:reduced cone from s r} this in turn reduces to~$\DownR(A,B)=\bot$. Hence by definition of the generating relation we get~$A\tensor B\sim \DownR(A,B)\tensor B = \bot\tensor B=\bot$. Since~$u$ is precisely the join over such rectangles~$A\tensor B\sqleq u$, under the generated frame congruence~$\langle\sim\rangle$ we will get that~$u\langle\sim\rangle \bot$, and the result follows.
\end{proof}

\begin{remark}
	It was shown in~\cite[\S 9.3]{schaaf2026LocalicRelationsOpenCones} that in the two-sided setting this statement actually generalises to \emph{weakly closed} relations. Weak closure (and \emph{strong density}) were introduced in~\cite{johnstone1989ConstructiveClosedSubgroup} to provide a constructively appropriate notion of closure for a localic closed subgroup theorem. While~\cref{proposition:closed relation fixed point} also generalises to weakly closed relations, the stronger version of closedness really seems essential for the present work, since open complements are explicitly used to define ordered Stone locales~(\cref{definition:ordered stone locale}).
\end{remark}

%~~~~~~~~~~~~~~~~~~~~~~~~~~~~~~
\subsection{Closed conic frames}
\label{section:closed conic frames}
We saw in~\cref{proposition:closed relation fixed point} that closed localic relations with open source maps are fixed points of the adjunction~$\cone\dashv\rel$. However, not all fixed points are necessarily closed (take for instance the diagonal relation on any non-Hausdorff locale). Thus we will need a further criterion that determines when the induced relation~$R_{\dc}$ is closed in terms of the conic frame~$(L,\dc)$. 

Spatially, note that a rectangle~$x\tensor y$ is disjoint from the relation~$R$ precisely when~${x\wedge \dc y =\bot}$, which are precisely the rectangles that generate the open~$O^\Dc_\bot$. Thus, if $R_{\dc}$ were to be closed, we would expect it to be the one with open complement~$O^\Dc_\bot$. Then we get (up to isomorphism) that~$r_{\dc}^{-1}\colon o\mapsto o\vee O^\Dc_\bot$, and this map equalises the generating relation~$\sim$, so we get in particular that~$x\tensor y \vee O^\Dc_\bot$ is equal to~$\Dc(x,y)\tensor y \vee O^\Dc_\bot$. This condition is adopted into a definition as follows.

\begin{definition}
	\label{definition:closed conic frame}
	A conic frame~$(L,\dc)$ is called \emph{closed} if for all~$x,y\in L$:
	\[
	x\tensor y \sqleq \Dc(x,y)\tensor y \vee O^\Dc_\bot. 
	\]
\end{definition}

\begin{proposition}
	\label{proposition:closed iff Rdc closed}
	A conic frame~$(L,\dc)$ is closed iff~$R_{\dc}$ is closed.
\end{proposition}
\begin{proof}
	First assume that~$R_{\dc}$ is closed, so it admits an open complement~$u\in L\tensor L$ so that~$R_{\dc}$ can be presented by the principal congruence~$\cc{u}=\langle\sim\rangle$. We thus get in particular that
	\[
	x\tensor y \vee u = \Dc(x,y)\tensor y \vee u
	\]	
	for all $x,y\in L$. Since~$\Dc(x,y)\sqleq x$, this is equivalent to~${x\tensor y\sqleq \Dc(x,y)\tensor y\vee u}$. We claim that indeed~$u=O^\Dc_\bot$. Set in particular~$x,y\in L$ such that~$x\wedge \dc y =\bot$, forming a rectangle that generates~$O^\Dc_\bot$. Then
	\[
	x\tensor y \sqleq \Dc(x,y)\tensor y \vee u = \bot \tensor y\vee u = \bot \vee u = u,
	\]
	so taking joins over such rectangles we get~$O^\Dc_\bot \sqleq u$. Conversely, since the congruence~$\cc{u}$ is generated by~$\{(u,\bot)\}$ we get~$u\langle\sim\rangle \bot$, and since~$r_{\dc}^{-1}=\mu_\sim$ equalises this relation we get in particular for all rectangles~$x\tensor y \sqleq u$ that
	\[
	r_{\dc}^{-1}(x\tensor y)\sqleq r_{\dc}^{-1}(u)= r_{\dc}^{-1}(\bot)=\bot.
	\]
	Applying the left adjoint~$s_!\dashv s_{\dc}^{-1}$, which preserves joins, hence gives using~\cref{corollary:left adjoint on rectangles} that
	\[
	x\wedge \dc y = s_!r_{\dc}^{-1}(x\tensor y) = s_!(\bot)=\bot,
	\]
	showing that~$x\tensor y \sqleq O^\Dc_\bot$. Taking joins over such rectangles hence gives~$u\sqleq O^\Dc_\bot$, and thus equality. This proves~$(L,\dc)$ is closed.
	
	Now assume that~$(L,\dc)$ satisfies~$x\tensor y \sqleq \Dc(x,y)\tensor y \vee O^\Dc_\bot$ for all~$x,y\in L$. Abbreviating~$u:=O^\Dc_\bot$, this clearly implies~$x\tensor y \vee u = \Dc(x,y)\tensor y \vee u$, so~${\sim}\subseteq \cc{u}$ and hence~$\langle \sim\rangle \subseteq \cc{u}$. For the converse, since~$\cc{u}$ is generated by~$\{(u,\bot)\}$, it suffices to show~$u\langle \sim\rangle \bot$. To see this, take~$x\tensor y\sqleq u$, which by~$\Sigma\dashv \sigma$ holds precisely when~$\Dc(x,y)\sqleq \bot$. Thus~$x\tensor y \sim \Dc(x,y)\tensor y =\bot$, and after passing to the generated congruence and taking joins gives~$u\langle \sim\rangle \bot$, as desired. This shows~$R_{\dc}$ is closed with open complement~$O^\Dc_\bot$.
\end{proof}

%~~~~~~~~~~~~~~~~~~~~~~~~~~~~~~~
\subsection{Preorders in terms of cones}
\label{section:preorders in terms of cones}
The axioms of a preorder can be expressed in terms of certain properties of the cone. Recall that a function~$c\colon P\to P$ on a partially ordered set is called:
	\begin{itemize}
		\item \emph{inflationary} if $x\leq c(x)$ for all $x\in P$;
		\item \emph{subidempotent} if $c(c(x))\leq c(x)$ for all $x\in P$;
		\item \emph{idempotent} if $c(c(x))=c(x)$ for all $x\in P$;
		\item a \emph{closure operator} if it is monotone, inflationary and subidempotent.
	\end{itemize}
Mirroring point-set intuition, the next sequence of results show that fixed points~$R_{\dc}$ of~$\cone\dashv \rel$ are preorders iff their cone~$\dc$ is a closure operator.

\begin{lemma}
	\label{lemma:relation inclusion implies cone inclusion}
	If $Q\subseteq R$ for open source relations, then~$\Downsub{Q}\sqleq \Downsub{R}$.
\end{lemma}
\begin{proof}
	The inclusion~$Q\subseteq R$ holds iff~$\id_X\colon (X,Q)\to (X,R)$ is monotone, so applying the functor $\cone$ we get a conic morphism $\id_{\Opens X}\colon (\Opens X,\Downsub{R})\to (\Opens X,\Downsub{Q})$, which by~\cref{definition:conic morphism} holds iff~$\Downsub{Q}\sqleq \Downsub{R}$.
\end{proof}

\begin{lemma}[{\cite[\S 10.1]{schaaf2026LocalicRelationsOpenCones}}]
	\label{lemma:cones of composition}
	If~$R,Q$ have open source, then so does~$R\circ Q$, and
	\[
	\Downsub{R\circ Q}=\Downsub{R}\circ\Downsub{Q}.
	\]
\end{lemma}
%\begin{proof}
%	The proof is the same as in~\cite[\S 10.1]{schaaf2026LocalicRelationsOpenCones}.
%\end{proof}

\begin{lemma}
	\label{lemma:preorder implies closure operator}
	For an open source preorder~$R$ the cone~$\Down$ is a closure~operator.
\end{lemma}
\begin{proof}
	This follows immediately by applying~\cref{lemma:relation inclusion implies cone inclusion,lemma:cones of composition} to the inclusions in~\cref{definition:localic relation properties}.
\end{proof}

\begin{proposition}
	\label{proposition:Rdc preorder iff closure operator}
	A fixed point~$R_{\dc}$ is a preorder iff~$\dc$ is a closure operator.
\end{proposition}
\begin{proof}
	By~\cref{lemma:preorder implies closure operator} we are left to show that if~$\dc$ is a closure operator then~$R_{\dc}$ is a preorder. That~$\dc$ is inflationary means~$\id\sqleq \dc$, so applying the functor~$\rel$ we get~$R_{\id}\subseteq R_{\dc}$, and we saw in~\cref{example:diagonal relation} that~$R_{\id}=\Delta$, so~$R_{\dc}$ is reflexive.
	
	On the other hand, that~$\dc$ is subidempotent means~$\dc^2\sqleq\dc$. By~\cref{lemma:cones of composition} the cone of the composition~$R_{\dc}\circ R_{\dc}$ is precisely~$\dc^2$, so by~\cref{corollary:universality Rdc} this implies~$R_{\dc}\circ R_{\dc}\subseteq R_{\dc}$, proving transitivity.
\end{proof}

%~~~~~~~~~~~~~~~~~~~~~~~~~~~~
\section{Internal p-morphisms}
\label{section:internal p morphisms}
As explained in~\cref{section:esakia duality}, morphisms between Esakia spaces are not always taken to be just the monotone continuous functions. Corresponding to Heyting algebra morphisms are instead the p-morphisms between Esakia spaces~(\cref{definition:p morphism pointwise}). In this section we define the internal analogue of this condition for localic relations, which will be used in~\cref{section:esakia locales and esakia frames} to define categories of Esakia locales.

We start by internalising the first-order~\cref{definition:p morphism pointwise} of p-morphisms. The main task of this section will be to translate the resulting internal notion to an external one involving cones, which follows the spatial result~\cite[Proposition 1.4.12(3)]{esakia2019HeytingAlgebrasDuality} that~$f$ is a p-morphism iff~$\down\circ f^{-1}=f^{-1}\circ \down\,$.

\begin{definition}
	\label{definition:epi pullback}
	In a category with pullbacks, an \emph{epi-pullback}~(also known as \emph{quasi-pullback}) is a commutative~square
	% https://q.uiver.app/#q=WzAsNCxbMCwwLCJBIl0sWzAsMSwiQiJdLFsxLDAsIkMiXSxbMSwxLCJEIl0sWzAsMV0sWzAsMl0sWzIsM10sWzEsM11d
	\[\begin{tikzcd}[cramped]
		A & C \\
		B & D
		\arrow[from=1-1, to=1-2]
		\arrow[from=1-1, to=2-1]
		\arrow[from=1-2, to=2-2]
		\arrow[from=2-1, to=2-2]
	\end{tikzcd}\]
	such that the canonical comparison map~$A\to B\times_D C$ is an epimorphism.
\end{definition}

\begin{definition}
	\label{definition:p morphism}
	A map~$f\colon (X,R)\to (Y,Q)$ in~$\rLoc$ is called a \mbox{\emph{p-morphism}} if its source square~(recall~\cref{definition:monotone internal}) is an epi-pullback:
	\[\begin{tikzcd}[cramped]
		R & Q \\
		X & {Y.}
		\arrow["{\bar f}", from=1-1, to=1-2]
		\arrow["{s_R}"', from=1-1, to=2-1]
		\arrow["{s_Q}", from=1-2, to=2-2]
		\arrow["f", from=2-1, to=2-2]
	\end{tikzcd}\]
\end{definition}

\begin{remark}
	We check that this internal condition unpacks to the pointwise definition of p-morphisms in~\cref{definition:p morphism pointwise}.
	That~$f\colon (X,R)\to (Y,Q)$ is internally monotone unpacks in this case to the existence of the function~$\bar f\colon R\to Q$ given by~$(x_1,x_2)\mapsto (f(x_1),f(x_2))$. On the other hand, we see that the pullback~$X\times_YQ$ of the source square may be identified with the set containing pairs~$(x_1,y_2)$ such that~$f(x_1)Qy_2$. In this form, the canonical comparison map becomes~${R\ni (x_1,x_2)\mapsto (x_1,f(x_2))}$. Requiring this to be an epimorphism says that for every~$f(x_1)Qy_2$ there exists~$x_2\in X$ with~$x_1 Rx_2$ and~$f(x_2)=y_2$, which together with monotonicity is precisely the p-morphism condition.
\end{remark}

We are particularly interested in the setting that the source maps are open, meaning that the vertical legs of the commutative source square are open. In this case we can say more, using the following well-known result that pullbacks of open locale maps are open. Using this result, we establish some lemmas that will help us translate the internal p-morphism condition into a statement about the involved induced cones.

\begin{lemma}[{\cite[Proposition V.4.1]{joyal1984ExtensionGaloisTheory}}]
	\label{lemma:JT pullback}
	Suppose $B\times_D^{b,c} C$ is the pullback of~$b,c$, with projections~$\pi_B,\pi_C$. If~$c$ is open then~$\pi_B$ is open, and
	\[
	(\pi_B)_!\circ \pi_C^{-1} = b^{-1}\circ c_!.
	\]
\end{lemma}

\begin{lemma}
	\label{lemma:epi pullback of opens}
	If a commutative square of locales of the form
	% https://q.uiver.app/#q=WzAsNCxbMCwwLCJBIl0sWzAsMSwiQiJdLFsxLDAsIkMiXSxbMSwxLCJEIl0sWzAsMSwiXFx0ZXh0e29wZW5+fWYiLDJdLFswLDIsImciXSxbMiwzLCJjXFx0ZXh0e35vcGVufSJdLFsxLDMsImIiXV0=
	\[\begin{tikzcd}[cramped]
		A & C \\
		B & D
		\arrow["g", from=1-1, to=1-2]
		\arrow["{\text{\normalfont open~}f}"', from=1-1, to=2-1]
		\arrow["{c\text{\normalfont~open}}", from=1-2, to=2-2]
		\arrow["b", from=2-1, to=2-2]
	\end{tikzcd}
	\qquad
	\begin{array}{c}
		\text{is an epi-pullback, then:}\\[1.5ex]
		f_!\circ g^{-1} = b^{-1}\circ c_!.
	\end{array}
	\]
\end{lemma}
\begin{proof}
	Abbreviate by~$P:=B\times_D C$ the actual pullback of~$b$ and~$c$, and denote the projections by~${\pi_B\colon P\to B}$ and~$\pi_C\colon P \to C$. Let~$e\colon A\tworightarrow P$ denote the canonical comparison map, uniquely defined by~$\pi_B\circ e = f$ and~$\pi_C \circ e= g$. It is assumed that~$e$ is an epimorphism, which we recall means that~$e^{-1}$ is an injective frame map~(\cite[Lemma~IX.4.2]{maclane1994SheavesGeometryLogic}), or equivalently that~$e^{-1}$ reflects the orders on the frames. Since~$c$ is open, by~\cref{lemma:JT pullback} the projection~$\pi_B$ is also open, and
	\[
	(\pi_B)_!\circ \pi_C^{-1}
	=
	b^{-1}\circ c_!.
	\]
	Hence it suffices to show~$f_!\circ g^{-1}$ equals the map on the left-hand side. To see this, pick opens~$V\in\Opens B$ and~$W\in\Opens C$ and observe:
	\begin{align*}
		f_!\circ g^{-1}(W)\sqleq V
		&\iff
		g^{-1}(W)\sqleq f^{-1}(V)
		&&(f_!\dashv f^{-1})
		\\&\iff
		g^{-1}(W)\sqleq e^{-1}\circ \pi_B^{-1}(V)
		&&(f= \pi_B\circ e)
		\\&\iff
		e^{-1}\circ \pi_C^{-1}(W)\sqleq e^{-1}\circ \pi_B^{-1}(V)
		&&(g=\pi_C\circ e)
		\\&\iff
		\pi_C^{-1}(W)\sqleq \pi_B^{-1}(V)
		&&(e\text{~epi})
		\\&\iff
		(\pi_B)_!\circ \pi_C^{-1}(W)\sqleq V,
		&&((\pi_B)_!\dashv \pi_B^{-1})
	\end{align*}
	from which the claim follows.
\end{proof}

\begin{lemma}
	\label{lemma:p morphism implies cone equality}
	If a map~$f\colon (X,R)\to (Y,Q)$ in~$\rosLoc$ is a p-morphism, then
	\[
	\Downsub{R}\circ f^{-1}= f^{-1}\circ \Downsub{Q}.
	\]
\end{lemma}
\begin{proof}
	That~$f$ is internally monotone means we get a map of locales~$\bar f\colon R\to Q$ such that~$s_Q\circ \bar f = f\circ s_R$ and~$t_Q\circ \bar f = f\circ t_R$. Next, since~$R,Q$ have open source, the vertical legs of the source square are open, so the p-morphism condition now gives through~\cref{lemma:epi pullback of opens} that
	\[
	\tag{$\star$}
	(s_R)_!\circ \bar f^{-1} = f^{-1}\circ (s_Q)_!.
	\]
	Using this, we obtain the desired equation:
	\begin{align*}
		\Downsub{R}\circ f^{-1}
		&=
		(s_R)_!\circ t_R^{-1}\circ f^{-1}
		&&(\text{by definition})
		\\&=
		(s_R)_! \circ \bar f^{-1}\circ t_Q^{-1}
		&&(t_Q\circ \bar f = f\circ t_R)
		\\&=
		f^{-1}\circ (s_Q)_! \circ t_Q^{-1}
		&&(\star)
		\\&=
		f^{-1}\circ \Downsub{Q}.
		&&(\text{by definition})\qedhere
	\end{align*}
\end{proof}

We will also need a converse of this in the setting of Esakia locales, where the equality~$\Downsub{R}\circ f^{-1}= f^{-1}\circ \Downsub{Q}$ of cones needs to imply that~$f$ is a p-morphism. To show this, we use the fact that a dense map from a compact locale into a regular locale is an epimorphism, and then show that the canonical comparison map is indeed dense. The rest of this section is dedicated to proving this.

Recall that a map of locales~$f$ is \emph{dense} if~$f^{-1}(V)=\bot$ implies~${V=\bot}$~(see~\cite[\S III.8.2]{picado2012FramesLocalesTopology}). Recall also the notion of a \emph{regular} locale from~\cite[\S V.5.3]{picado2012FramesLocalesTopology}, which is the point-free analogue of regular topological spaces. We will not need the explicit definition here, save to note the following fact.

\begin{lemma}[{\cite[\S III.1.1]{johnstone1982StoneSpaces}}]
	\label{lemma:zero dimensional is regular}
	Every zero-dimensional locale is regular.
\end{lemma}

Using this, the following result says that to prove a morphism between Stone locales is an epimorphism, it suffices to show it is dense.

\begin{lemma}[{\cite[Proposition~VII.2.2.2]{picado2012FramesLocalesTopology}}]
	\label{lemma:dense implies epimorphism}
	A dense map~$e\colon X\to Y$ from a compact locale~$X$ to a regular locale~$Y$ is an epimorphism.
\end{lemma}

To show that the source square of a map~$f\colon (X,R)\to (Y,Q)$ in~$\rLoc$ is an epi-pullback, we need to show that the comparison map~$R\to X\times_Y Q$ is an epimorphism. In order to use the previous lemma, we need to show that the pullback~$X\times_Y Q$ is a regular locale, in case that~$X,Y$ are Stone locales and~$Q$ is a closed relation. The following lemma will allow us to obtain that.

Recall that in any category, a map is called a \emph{regular monomorphism} if it is the equaliser of some parallel pair of maps. By~\cite[\S IV.1.3]{picado2012FramesLocalesTopology}, sublocale inclusions are precisely the regular monomorphisms in~$\Loc$. We need the following pullback stability behaviour.

\begin{lemma}
	\label{lemma:regular mono pullback}
	In a category with finite limits, let~$q=(s,t)\colon Q\rightarrowtail Z\times Y$ be a regular monomorphism, and~$f\colon X\to Z$ any morphism. If~$P:=X\times_Z Q$ is the pullback of~$s$ and~$f$ with projections~$\pi_X,\pi_Q$, then~$j:=(\pi_X,t\circ\pi_Q)$ is a regular monomorphism.
\end{lemma}
\begin{proof}
	We claim that the following square is a pullback:
	% https://q.uiver.app/#q=WzAsNSxbMywxLCJRIl0sWzMsMywiWlxcdGltZXMgWS4iXSxbMSwzLCJYXFx0aW1lcyBZIl0sWzEsMSwiUCJdLFswLDAsIkEiXSxbMiwxLCJmXFx0aW1lcyBcXGlkX1kiLDJdLFswLDEsInEiLDAseyJzdHlsZSI6eyJ0YWlsIjp7Im5hbWUiOiJtb25vIn19fV0sWzMsMCwiXFxwaV9RIl0sWzMsMiwiaiIsMl0sWzQsMiwiYSIsMix7ImN1cnZlIjozfV0sWzQsMCwiYiIsMCx7ImN1cnZlIjotM31dLFs0LDMsIlxcZXhpc3RzIT8iLDAseyJzdHlsZSI6eyJib2R5Ijp7Im5hbWUiOiJkYXNoZWQifX19XV0=
	\[\begin{tikzcd}[cramped,sep=small]
		A &&& \\
		& P && Q \\
		\\
		& {X\times Y} && {Z\times Y.}
		\arrow["{\exists!?}", dashed, from=1-1, to=2-2]
		\arrow["b", bend left = 20, from=1-1, to=2-4]
		\arrow["a"', bend right = 20, from=1-1, to=4-2]
		\arrow["{\pi_Q}", from=2-2, to=2-4]
		\arrow["j"', from=2-2, to=4-2]
		\arrow["q", tail, from=2-4, to=4-4]
		\arrow["{f\times \id_Y}"', from=4-2, to=4-4]
	\end{tikzcd}\]
	Observe first that the square commutes, since
	\begin{align*}
		(f\times \id_Y)\circ j
		&=
		(f\times \id_Y) \circ (\pi_X,t\circ \pi_Q)
		&&(\text{by definition})
		\\&=
		(f\circ \pi_X, t\circ \pi_Q)
		\\&=
		(s\circ \pi_Q,t\circ \pi_Q)
		&&(f\circ \pi_X= s\circ \pi_Q)
		\\&=
		(s,t)\circ \pi_Q.
	\end{align*}
	Now take the span of arrows~$a,b$, as depicted. Writing~$a=(a_1,a_2)$, that the outer square commutes implies~$f\circ a_1 = s\circ b$ and~$a_2 = t\circ b$. The first equation invokes the universal property of~$P$, so there exists a unique~$c\colon A\to P$ such that~$\pi_X\circ c = a_1$ and~$\pi_Q\circ c = b$. Together with the second equation, this gives
	\[
	j\circ c
	=
	(\pi_X\circ c,t\circ\pi_Q\circ c)
	=
	(a_1,t\circ b)
	=
	(a_1,a_2)
	=a.
	\]
	Therefore,~$c$ completes the diagram above, and that it is the unique such arrow follows since~$\pi_X,\pi_Q$ are jointly monic.
	Thus~$j$ is a pullback of~$q$, and the result follows since regular monomorphisms are preserved under pullback~(see for example~\cite[Proposition~4.3.8]{borceux1994Handbook1}).
\end{proof}

\begin{proposition}
	\label{proposition:p morphism on compact regular locales iff cones}
	Let~$f\colon (X,R)\to (Y,Q)$ be a morphism in~$\rosLoc$ where~$X$ is compact regular,~$Y$ is regular, and~$R,Q$ are closed localic relations. Then~$f$ is a p-morphism iff~${\Downsub{R}\circ f^{-1}= f^{-1}\circ \Downsub{Q}}$.
\end{proposition}
\begin{proof}
	That p-morphisms satisfy this equation of cones follows from~\cref{lemma:p morphism implies cone equality}, so we are left to prove the converse. Hence, assume that~$\Downsub{R}\circ f^{-1}= f^{-1}\circ \Downsub{Q}$ holds. As before, abbreviate by~$P:= X\times_Y Q$ the pullback of~$f$ and~$s_Q$, and let~${e\colon R\to P}$ be the resulting comparison map of the source square, the unique morphism satisfying~${\pi_X\circ e = s_R}$ and~$\pi_Q\circ e= \bar f$. 
	
	We claim that~$e$ is dense. Take~$W\in\Opens P$ such that~$e^{-1}(W)=\bot$, so we need to show~$W=\bot$. For this, we pick a join-basis of~$\Opens P$ as follows. By~\cref{lemma:regular mono pullback} the map~$j:=(\pi_X,t_Q\circ \pi_Q)\colon P\rightarrowtail X\times Y$ is a sublocale inclusion, and hence the elements of the form~$j^{-1}(U\tensor V)$ for~$U\in \Opens X$ and~$V\in\Opens Y$ form a join-basis of~$\Opens P$. Now, to show that~$W=\bot$ it suffices to show that~$j^{-1}(U\tensor V)\sqleq W$ implies~$j^{-1}(U\tensor V)=\bot$. To see this, we calculate how~$e^{-1}$ evaluates on these basis elements:
	\begin{align*}
		e^{-1}\circ j^{-1}(U\tensor V)
		&=
		e^{-1}\circ\left(\pi_X^{-1}(U)\wedge \pi_Q^{-1}\circ t_Q^{-1}(V)\right)
		&&(\text{\cref{lemma:coproduct structure}})
		\\&=
		e^{-1}\circ \pi_X^{-1}(U)\wedge e^{-1}\circ \pi_Q^{-1}\circ t_Q^{-1}(V)
		&&(e^{-1}\text{~preserves~}\wedge)
		\\&=
		s_R^{-1}(U)\wedge e^{-1}\circ \pi_Q^{-1}\circ t_Q^{-1}(V)
		&&(\pi_X\circ e = s_R)
		\\&=
		s_R^{-1}(U)\wedge \bar f^{-1}\circ t_Q^{-1}(V)
		&&(\pi_Q\circ e = \bar f)
		\\&=
		s_R^{-1}(U)\wedge t_R^{-1}\circ f^{-1}(V).
		&&(f\text{~monotone})
	\end{align*}
	Hence, if~$j^{-1}(U\tensor V)\sqleq W$, this evaluates to~$\bot$, and in turn applying the left adjoint~$(s_R)_!$ we get
	\begin{align*}
		\bot
		&=
		(s_R)_!(\bot)
		&&((s_R)_!\text{~preserves joins})
		\\&=
		(s_R)_!\left(s_R^{-1}(U)\wedge t_R^{-1}\circ f^{-1}(V)\right)
		&&(\text{previous equation})
		\\&=
		U\wedge (s_R)_!\circ t_R^{-1} \circ f^{-1} (V)
		&&(\text{Frobenius})
		\\&=
		U\wedge \Downsub{R}\circ f^{-1}(V)
		&&(\text{by definition})
		\\&=
		U\wedge f^{-1}\circ \Downsub{Q}V.
		&&(\text{by hypothesis})
	\end{align*}
	On the other hand,~$\pi_X$ is open as the pullback of~$s_Q$~(\cref{lemma:JT pullback}), and using this we calculate
	\begin{align*}
		(\pi_X)_!\circ j^{-1}(U\tensor V)
		&=
		(\pi_X)_!\circ\left(\pi_X^{-1}(U)\wedge \pi_Q^{-1}\circ t_Q^{-1}(V)\right)
		&&(\text{\cref{lemma:coproduct structure}})
		\\&=
		U\wedge (\pi_X)_!\circ \pi_Q^{-1}\circ t_Q^{-1}(V)
		&&(\text{Frobenius})
		\\&=
		U\wedge f^{-1}\circ (s_Q)_!\circ t_Q^{-1}(V)
		&&(\text{\cref{lemma:JT pullback}})
		\\&=
		U\wedge f^{-1}\circ\Downsub{Q}V
		&&(\text{by definition})
		\\&=
		\bot.
		&&(\text{previous equation})
	\end{align*}
	Lastly, by the adjunction~$(\pi_X)_!\dashv \pi_X^{-1}$ this implies~$j^{-1}(U\tensor V)\sqleq \pi_X^{-1}(\bot)=\bot$, so~$j^{-1}(U\tensor V)=\bot$, as was to be shown. Hence~$e$ is dense.
	
	To conclude that~$e$ is an epimorphism, by~\cref{lemma:dense implies epimorphism} it now suffices to show that~$R$ is compact and~$P$ is regular. For the domain, note that since~$X$ is compact, the product~$X\times X$ is compact~(by the localic Tychonoff theorem,~\cite[\S III.1.7]{johnstone1982StoneSpaces}), and so~$R$ is a closed sublocale of a compact locale, and must hence itself be compact~(\cite[\S III.1.2]{johnstone1982StoneSpaces}). For the codomain, since~$X$ and~$Y$ are regular, the product~${X\times Y}$ is regular~(\cite[\S III.1.6]{johnstone1982StoneSpaces}), and by~\cref{lemma:regular mono pullback} the pullback~$P$ is a sublocale of~$X\times Y$, and hence is regular by~\cite[\S III.1.2]{johnstone1982StoneSpaces}. 
\end{proof}

\begin{corollary}
	\label{corollary:p morphism on Esakia locales}
	Let~$f\colon (X,R)\to (Y,Q)$ be a locale map between Stone locales equipped with closed localic relations~$R,Q$ with open source. Then~$f$ is a p-morphism iff~$f^{-1}\circ \Downsub{Q}=\Downsub{R}\circ f^{-1}$.
\end{corollary}
\begin{proof}
	The `only if' direction is just~\cref{lemma:p morphism implies cone equality}. For the `if' direction, assume that the equation~${f^{-1}\circ \Downsub{Q}=\Downsub{R}\circ f^{-1}}$ holds. In particular, this means~$f^{-1}$ defines a morphism of conic frames~$(\Opens Y,\Downsub{Q})\to (\Opens X,\Downsub{R})$. Since~$R,Q$ are closed relations, they are fixed points~(\cref{proposition:closed relation fixed point}), and applying the functor~$\rel$ shows that~$f$ is internally monotone. The locales~$X,Y$ are Stone, so in particular compact regular by~\cref{lemma:zero dimensional is regular}, and hence~$f$ is a p-morphism by~\cref{proposition:p morphism on compact regular locales iff cones}.
\end{proof}

%~~~~~~~~~~~~~~~~~~~~~~~~~~~~~
\section{Esakia locales and frames}
\label{section:esakia locales and esakia frames}
In this section we define the localic and frame-theoretic analogues of Esakia spaces. This builds on the localic analogue of Priestley spaces introduced in~\cite{townsend1996preframeTechniquesConstructiveLocale,townsend1997LocalicPriestleyDuality}.
There, \emph{ordered Stone locales} are defined as Stone locales~$X$ equipped with certain localic partial orders~$R$ that satisfy a suitable point-free analogue of the Priestley separation axiom.
Adopting the spatial intuition of~\cref{definition:priestley space} that Esakia spaces are precisely the Priestley spaces with open source map, we then define~\emph{Esakia locales} as the ordered Stone locales with open source map~(\cref{section:esakia locales}). As such, we can apply the formalism of conic frames to provide a frame-theoretic dual to Esakia locales, which we call~\emph{Esakia frames}~(\cref{section:esakia frames}). This section establishes a dual equivalence between these two notions by restricting the adjunction~$\cone\dashv\rel$. The main work is showing that the external, frame-theoretic notions translate back to the corresponding internal localic notions.

%~~~~~~~~~~~~~~~~~~~~~~~~~~~~~~
\subsection{Ordered Stone locales}
To state the definitions, we need a notion of downwards closure~$\down$ for general localic relations. Let~$R$ be a relation on~$X$, and pick a sublocale~${A\in\Sl(X)}$. We can no longer assume that~$\down A$ defines an open sublocale, but following the same spatial intuition we used to define~$\Down$ in~\cref{definition:induced cones}, it is possible to give a more general notion of cone using internal images and preimages of localic maps~\cite[\S III.4]{picado2012FramesLocalesTopology}.

\begin{definition}
	For a localic relation~$R$ on a locale~$X$, its \emph{internal (down) cone} is defined as the function
	\[
	\down\colon \Sl(X)\longrightarrow \Sl(X);
	\qquad
	A\longmapsto s\left[t^{-1}[A]\right].
	\]
\end{definition}

\begin{remark}
	Since~$X\cong X\times 1$, we can equivalently view~$A$ as a localic relation from~$X$ to~$1$. Composing with~$R$ we thus get a relation~$R\circ A$ from~$X$ to~$1$, which in turn by universality can be identified with the sublocale~$\down A\in \Sl(X)$.
\end{remark}

In the case that~$R$ has open source map, we can see that the cone~$\Down$ is just the restriction of~$\down$ to open sublocales. Recall from~\cite[\S III.6.1.1]{picado2012FramesLocalesTopology} that the~\emph{open sublocale}~$\mathfrak{o}(u)$ associated to~$u\in \Opens X$ is defined by the frame congruence
\[
	x\Delta_u y
	\quad\iff\quad
	x\wedge u = y\wedge u.
\]

\begin{lemma}
	\label{lemma:Down = down}
	The cone~$\Down$ is the restriction of~$\down$ to open sublocales.
\end{lemma}
\begin{proof}
	This follows from~\cite[\S III.6.3]{picado2012FramesLocalesTopology} that~$t^{-1}\left[\mathfrak{o}(U)\right]=\mathfrak{o}\left(t^{-1}(U)\right)$, and from the proof of~\cite[Proposition~IX.7.3]{maclane1994SheavesGeometryLogic} that~$s[\mathfrak{o}(V)]=\mathfrak{o}(s_!(V))$. Combined:
	\[
	\down \mathfrak{o}(U)
	=
	s\left[t^{-1}[\mathfrak{o}(U)]\right]
	=
	\mathfrak{o}\left(s_!t^{-1}(U)\right)
	=
	\mathfrak{o}(\Down U).
	\qedhere
	\]
\end{proof}

For Priestley duality, a distributive lattice is constructed from the clopen upsets of the ordered space. A localic analogue can now be defined using the internal cone~$\down\mspace{3mu}$.
Spatially, a (clopen) subset~$C\subseteq X$ is upwards closed~$\up C = C$ iff its complement is downwards closed:
\[
\down(X\setminus C)=X\setminus C.
\]
For a Stone locale~$X$, a compact element~$c\in K(\Opens X)$ represents the clopen sublocale~$\mathfrak{o}(c)$, whose complement is~$\mathfrak{o}(\neg c)$. Directly translating the previous equation into localic language, we say~$c$ is an \emph{upper element} if:
\[
\down\mathfrak{o}(\neg c)=\mathfrak{o}(\neg c).
\]
This is also the convention followed in~\cite[p.326]{townsend1997LocalicPriestleyDuality}. For us, defining upwards closure in terms of downwards closure of the complement serves the independent purpose that the open source condition gives a downwards cone~$\Down$, so an analogous upwards closure condition~$\Up c= c$ is generally unavailable.

\begin{definition}
	\label{definition:K down}
	For any localic relation~$R$ on a locale~$X$, we define the set of \emph{compact upper elements} as
	\[
	K_{\down}(\Opens X):=\{c\in K(\Opens X): \down \mathfrak{o}(\neg c)=\mathfrak{o}(\neg c)\}.
	\]
\end{definition}

Townsend uses upper compact elements to give a localic analogue of the Priestley separation axiom. Recall from~\cref{definition:priestley space} that the spatial version of the axiom says for all $x\not\leq y$ there is a clopen~$C=\up C$ such that~$x\in C$ and~$y\notin C$. Note that this is a characterisation of the complement~$\not\leq$, and since Priestley relations are closed, the relation~$\not\leq$ should be represented by an open subspace. By~\cite[Lemma~2.1]{townsend1997LocalicPriestleyDuality} we see that this open subspace can be written as follows:
\[
{\not\leq}
=
\bigcup\{C\times (X\setminus C)
:
\up C = C\text{~compact open}
\}.
\]
This now has a direct point-free analogue. Recall from~\cref{remark:closed sublocales} that if~$R$ is a closed localic relation, we denote by~$\ocomplement{R}$ its open complement.

\begin{definition}[{\cite[\S 3]{townsend1997LocalicPriestleyDuality}}]
	\label{definition:ordered stone locale}
	An \emph{ordered Stone locale} $(X,R)$ is a Stone locale~$X$ equipped with a closed localic partial order~$R$ satisfying the \emph{localic Priestley separation axiom:}
	\[
	\ocomplement{R}
	=
	\bigvee\{c\tensor \neg c: c\in K_{\down}(\Opens X) \}.
	\]
\end{definition}

%~~~~~~~~~~~~~~~~~~~~~~~~~~~~~~
\subsection{Esakia locales}
\label{section:esakia locales}
Given this definition of ordered Stone locales as the localic analogue of Priestley spaces and directly mirroring the spatial intuition that Esakia spaces are the Priestley spaces with open source map, we immediately obtain the localic analogue of Esakia spaces.

\begin{definition}
	\label{definition:esakia locale}
	An \emph{Esakia locale} is an ordered Stone locale with open source map.
\end{definition}

In particular, an Esakia locale~$(X,R)$ induces a cone~$\Down$. We verify that the upper compact sets defined in terms of~$\Down$ and~$\down$ agree. For that, we first define a frame-theoretic analogue of upper compact elements.

\begin{definition}
	\label{definition:Kdc}
	For any frame~$L$ and function~$\dc\colon L\to L$ we define the set:
	\[
	\Kdc(L) = \left\{c\in K(L): \dc\neg c =\neg c\right\}.
	\]
\end{definition}

\begin{corollary}
	\label{corollary:KDown = Kdown}
	If~$R$ is an open source relation on~$X$ we have~${K_{\down}(\Opens X)=\KDown(\Opens X)}$.
\end{corollary}
\begin{proof}
	A compact element~$c\in K(\Opens X)$ is in~$K_{\down}(\Opens X)$ iff~$\down \mathfrak{o}(\neg c)=\mathfrak{o}(\neg c)$, which by~\cref{lemma:Down = down} holds in turn iff~$\mathfrak{o}(\Down\neg c)=\mathfrak{o}(\neg c)$, which is equivalent to the equality~$\Down\neg c=\neg c$ in~$\Opens X$ by~\cite[\S III.6.1.4]{picado2012FramesLocalesTopology}. But this just means~$c\in\KDown(\Opens X)$.
\end{proof}

%~~~~~~~~~~~~~~~~~~~~~~~~~~~~~
\subsection{Esakia frames}
\label{section:esakia frames}
The idea is to `externalise' the notion of Esakia locales, a localic notion, into frame-theoretic language. The partial order of an Esakia locale is a closed localic relation with open source, so it induces a conic frame~$(E,\Down)$ that defines a fixed point of~$\cone\dashv\rel$ by~\cref{proposition:closed relation fixed point}. Moreover, the frame~$E$ is Stone, the conic frame is closed by~\cref{proposition:closed iff Rdc closed}, and the cone~$\Down$ is a closure operator by~\cref{lemma:preorder implies closure operator}. Axiomatising these properties, together with an axiom that will correspond to anti-symmetry, we introduce the notion of an Esakia frame.

\begin{samepage} % force on same page
\begin{definition}
\label{definition:esakia frame}
	An \emph{Esakia frame} is a pair~$(E,\dc)$ such that:
	\begin{enumerate}[label=(E\arabic*),start=0]
		\item $(E,\dc)$ is a conic frame;
		\item $E$ is a Stone frame;
		\item $\dc$ is a closure operator;
		\item $(E,\dc)$ is closed;
		% \item $K(E)$ is generated as a Boolean algebra by $\Kdc(E)$.
		\item $\Kdc(E)$ Boolean generates~$K(E)$.
	\end{enumerate}
\end{definition}
\end{samepage}

Note the new axiom~(E4). Since~$E$ is a Stone frame, the compact elements~$K(E)$ form a Boolean algebra, of which~$\Kdc(E)$ forms a subset. To say that this subset generates~$K(E)$ as a Boolean algebra means that every element in the latter can be written as a finite Boolean combination of elements in the former. As mentioned, this corresponds to the anti-symmetry of the localic relation that is part of an Esakia locale. Spatially,~$\Kdc(E)$ represent the clopen upper sets and~$K(E)$ arbitrary clopen sets. If~$x\leq y$ and~$y\leq x$ holds in an ordered space, it is easy to see~$x$ and~$y$ are contained in precisely the same upper sets, and hence in particular in the same finite Boolean combinations of clopen upper sets. Under the assumption that finite Boolean combinations of clopens generate arbitrary clopens, which is what~(E4) models, it follows that~$x$ and~$y$ must be contained in the same clopens. But in a Stone space this implies~$x=y$, so~$\leq$ is anti-symmetric. See for example~\cite[Theorem~3.4.4]{esakia2019HeytingAlgebrasDuality}.

The frame-theoretic analogue of this story will be explained in~\cref{section:antisymmetry}. The corresponding localic statement was already observed for ordered Stone locales by Townsend, where the claim follows from the following description of compact elements in terms of clopen upper cones. This is also the last missing ingredient needed to prove Esakia locales induce Esakia frames under the functor~$\cone$.

\begin{theorem}[{\cite[Theorem~5.1]{townsend1997LocalicPriestleyDuality}}]
	\label{theorem:townsend normal form}
	If $(X,R)$ is an ordered Stone locale and $a\in K(\Opens X)$, then~$a=\bigwedge_{i\in I}(a_i\vee \neg b_i)$ for some finite~$I$ and~$a_i,b_i\in K_{\down}(\Opens X)$.
\end{theorem}

\begin{corollary}
	\label{corollary:Kdown generates K}
	In any ordered Stone locale,~$K_{\down}(\Opens X)$ Boolean generates~$K(\Opens X)$.
\end{corollary}
\begin{proof}
	Take the normal form of a compact element~$a\in K(\Opens X)$ from Townsend's~\cref{theorem:townsend normal form}, meaning~$a=\bigwedge_{i\in I}(a_i\vee \neg b_i)$ for~$a_i,b_i\in K_{\down}(\Opens X)$ and~$I$ finite. Clearly the expression on the right-hand side lands in the Boolean algebra generated by~$K_{\down}(\Opens X)$, so the result follows.
\end{proof}

\begin{proposition}
	\label{proposition:esakia locale gives esakia frame}
	If~$(X,R)$ is an Esakia locale then~$(\Opens X,\Down)$ is an Esakia frame.
\end{proposition}
%\begin{proof}
%	Since~$R$ has open source, it induces a conic frame~$(\Opens X,\Down)$ where~$\Opens X$ is Stone. Since~$R$ is closed, by~\cref{proposition:closed relation fixed point} it defines a fixed point of~$\cone\dashv\rel$, so~$R\cong\RDown$. From this, it follows by~\cref{proposition:closed iff Rdc closed,proposition:Rdc preorder iff closure operator} that~$(\Opens X,\Down)$ is closed as a conic frame, and~$\Down$ is a closure operator. Lastly, by~\cref{corollary:KDown = Kdown,corollary:Kdown generates K} $\KDown(\Opens X)$ generates~$K(\Opens X)$.
%	Hence~$(\Opens X,\Down)$ defines an Esakia frame.
%\end{proof}
\begin{proof}
	Observe:
	\begin{enumerate}[label=(E\arabic*),start=0]
		\item $(\Opens X,\Down)$ is a conic frame by~\cref{proposition:induced cone is join-preserving};
		\item $\Opens X$ is a Stone frame since~$X$ is Stone;
		\item $\Down$ is a closure operator by~\cref{lemma:preorder implies closure operator};
		\item $(\Opens X,\Down)$ is closed by~\cref{proposition:closed relation fixed point,proposition:closed iff Rdc closed};
		\item $\KDown(\Opens X)$ generates $K(\Opens X)$ by~\cref{corollary:KDown = Kdown,corollary:Kdown generates K}.\qedhere
	\end{enumerate}
\end{proof}

The rest of this section is dedicated to the converse construction, that an Esakia frame induces an Esakia locale under~$\rel$. While the axioms of an Esakia frame are clearly motivated by the axioms of an Esakia locale, it is not immediate that the locale~$\rel(E,\dc)$ obtained from an Esakia frame~$(E,\dc)$ will indeed form an Esakia locale. The main two missing ingredients are localic Priestley separation, and anti-symmetry. The next two subsections establish these properties, respectively.

%~~~~~~~~~~~~~~~~~~~~~~~~~~~~
\subsection{Priestley separation}
\label{section:priestley separation}
Note the absence of a direct frame-theoretic analogue of the Priestley separation axiom in the~\cref{definition:esakia frame} of an Esakia frame. In fact, it can be derived. This is analogous to the spatial fact that any Stone space equipped with a closed preorder with open source map automatically satisfies the Priestley separation axiom. See for example~\cite[Theorem~3.2.22(1)]{esakia2019HeytingAlgebrasDuality}. This section establishes the frame-theoretic analogue of this statement, and concludes with an explicit formula for the open complement of the resulting localic preorder.

\begin{definition}
\label{definition:priestley separation frames}
	A conic frame~$(L,\dc)$ satisfies the \emph{Priestley separation axiom} if for all compact~$a,b\in K(L)$:
	\[
	a\wedge \dc b = \bot
	\quad\implies\quad
	\exists c\in \Kdc(L): a\sqleq c\text{ and } b\sqleq \neg c.
	\]
\end{definition}

We now show that any closed conic Stone frame where the cone is a closure operator satisfies the Priestley separation axiom. First we need a general lemma about compact elements in coproduct frames.

\begin{lemma}
	\label{lemma:compact rectangles}
	If~$a\in K(L)$ and~$b\in K(M)$ then~${a\tensor b\in K(L\tensor M)}$.
\end{lemma}
\begin{proof}
	The open sublocale~$\mathfrak{o}(a)$ of~$a\in L$ can be represented by the subframe of~$L$ of all elements contained in~$a$~(\cite[\S III.6.1.1]{picado2012FramesLocalesTopology}), whose top element is~$a$ itself. Hence~$a\in K(L)$ iff~$\mathfrak{o}(a)$ is a compact frame. If also~$b\in K(M)$, by the localic Tychonoff theorem~(\cite[\S III.1.7]{johnstone1982StoneSpaces}) the frame coproduct~$\mathfrak{o}(a)\tensor\mathfrak{o}(b)\cong \mathfrak{o}(a\tensor b)$ is again compact, from which it follows that~$a\tensor b\in L\tensor M$ is compact.
\end{proof}

\begin{lemma}
	\label{lemma:dc preserves compact elements}
	If~$(L,\dc)$ is a closed conic Stone frame, then~$\dc$ preserves compacts.
\end{lemma}
\begin{proof}
	Start with a compact element~$c\in K(L)$. Since~$L$ is Stone and in particular coherent, the element~$\dc c$ may be written as the join over the compact elements it contains, so in particular
	\[
	\dc c\tensor c
	=
	\bigvee\{d\in K(L): d\sqleq \dc c\}\tensor c
	=
	\bigvee\left\{d\tensor c: d\in K(L): d\sqleq \dc c \right\}. 		
	\]
	For the same reason,~$O^\Dc_\bot$ is precisely the join over the compact rectangles~${a\tensor b}$, where~$a,b\in K(L)$, such that~$\Dc(a,b)=\bot$. Since~$c$ is compact, the rectangle~$\top\tensor c$ is compact in~$L\tensor L$~(\cref{lemma:compact rectangles}), and since by closedness we have the inclusion
	\[
	\top\tensor c\sqleq \dc c\tensor c \vee O^\Dc_\bot,
	\]
	it follows there are finitely many compact elements~$d_i\sqleq \dc c$,~$i\in I$, and finitely many compact rectangles~$a_j\tensor b_j$,~$j\in J$ satisfying~$\Dc(a_j,b_j)=\bot$ for all~$j\in J$, such that:
	\[
	\top\tensor c
	\sqleq
	\bigvee_{i\in I} d_i\tensor c \vee \bigvee_{j\in J	}a_j\tensor b_j.
	\]
	Applying the join-preserving map~$\Sigma\colon L\tensor L\to L$ from~\cref{section:Sigma} this implies
	\begin{align*}
		\dc c
		&=
		\Sigma(\top\tensor c)
		&&(\text{\cref{lemma:Sigma exists}})
		\\&\sqleq
		\Sigma\left(
		\bigvee_{i\in I} d_i\tensor c \vee \bigvee_{j\in J} a_j\tensor b_j
		\right)
		\\&=
		\bigvee_{i\in I} \Dc(d_i,c) \vee \bigvee_{j\in J} \Dc(a_j,b_j)
		&&(\text{$\Sigma$ preserves joins})
		\\&=
		\bigvee_{i\in I} d_i\wedge \dc c,
		&&(\forall j:\Dc(a_j,b_j)=\bot)
	\end{align*}
	from which it follows that~$\dc c=\bigvee_{i\in I} d_i$, so~$\dc c$ is compact by~\cref{lemma:finite joins of compact elements are compact}.
\end{proof}

\begin{lemma}
	\label{lemma:closed Stone implies Priestley separation}
	Any closed conic Stone frame~$(L,\dc)$ where~$\dc$ is a closure operator satisfies the Priestley separation axiom.
\end{lemma}
\begin{proof}
	Suppose~$a,b\in K(L)$ are compact elements satisfying~$a\wedge\dc b =\bot$. By~\cref{lemma:dc preserves compact elements},~$\dc b$ is compact, and by definition of a Stone frame~$c:=\neg\dc b$ is compact. By construction~$a\sqleq c$, and since~$\dc$ is inflationary we have~$b\sqleq \dc b= \neg c$. We are left to show~$c\in \Kdc(L)$. For that, since~$\dc$ is a closure operator we have in particular that~$\dc^2=\dc$, so $\dc \neg c = \dc \dc b =\dc b =\neg c$, as desired.
\end{proof}

Recall from~\cref{section:closed conic frames} that a closed conic frame~$(L,\dc)$ models a localic relation whose open complement is~$O^\Dc_\bot$. The next results establish that the frame-theoretic Priestley separation axiom of~\cref{definition:priestley separation frames} implies Townsend's localic Priestley separation axiom from~\cref{definition:ordered stone locale}.

\begin{lemma}
	\label{lemma:priestley separation implies ODc from Kdc}
	If~$(L,\dc)$ is a conic Stone frame satisfying Priestley separation, then
	\[
	O^\Dc_\bot = \bigvee\{c\tensor \neg c:c\in \Kdc(L)\}.
	\]
\end{lemma}
\begin{proof}
	Abbreviate~$u:=\bigvee\{c\tensor \neg c: c\in \Kdc(L)\}$, so it is to be shown that~${u=O^\Dc_\bot}$. Observe first that if~$c\in \Kdc(L)$ then~$\Dc(c,\neg c) = c\wedge \dc \neg c= c\wedge \neg c =\bot$, so it follows immediately that~$u\sqleq O^\Dc_\bot$.
	
	For the converse, pick a compact rectangle~$a\tensor b\sqleq O^\Dc_\bot$, so by~\cref{corollary:separating open rectangle adjunction} we get~$\Dc(a,b)=\bot$. Applying the Priestley separation axiom, there exists~$c\in \Kdc(L)$ such that~$a\sqleq c$ and~$b\sqleq \neg c$, and so~$a\tensor b\sqleq c\tensor \neg c\sqleq u$. Since~$L$ is Stone and hence coherent, $O^\Dc_\bot$ is generated by these~$a\tensor b$, so taking joins gives~$O^\Dc_\bot\sqleq u$.	
\end{proof}

\begin{corollary}
	\label{corollary:closed conic Stone frame ODc from Kdc}
	If~$(L,\dc)$ is a closed conic Stone frame where~$\dc$ is a closure operator, then~$O^\Dc_\bot = \bigvee\{c\tensor \neg c:c\in \Kdc(L)\}$.
\end{corollary}
\begin{proof}
	This follows from~\cref{lemma:closed Stone implies Priestley separation,lemma:priestley separation implies ODc from Kdc}.
\end{proof}

%~~~~~~~~~~~~~~~~~~~~~~~~~~~~~~
\subsection{Anti-symmetry}
\label{section:antisymmetry}
This section proves how the axiom~(E4) of an Esakia frame ensures the induced localic relation is anti-symmetric. This allows us to conclude in~\cref{proposition:esakia frame gives esakia locale} that~$\rel(E,\dc)$ is an Esakia locale for every Esakia frame~$(E,\dc)$.

The strategy is to translate the internal anti-symmetry condition~${R\cap R^\op \subseteq \Delta}$ from~\cref{definition:localic relation properties} into a condition involving the open complements~$\ocomplement{R},\ocomplement{R^\op}$ of~$R$ and its opposite, and the open complement~$\ocomplement{\Delta}$ of the diagonal relation. We first have the following general observation about the opposite of a closed relation.

\begin{remark}
	\label{remark:swap}
	\label{lemma:Rop is swap preimage}
	Denote by~$\swap\colon X\times X\to X\times X$ the swap map of~$X$ in~$\Loc$, defined generally as~$\swap=(\pr_2,\pr_1)$. Explicitly, by~\cref{lemma:coproduct structure} we get the corresponding frame map on rectangles as
	\[
	\swap^{-1}(x\tensor y) = y\tensor x.
	\]
	From general categorical arguments we also get that if~$R$ is a localic relation with inclusion~$r=(s,t)$, its opposite relation~$R^\op$ with inclusion~$(t,s)$ is just the internal preimage~$\swap^{-1}[R]$. In particular, if~$R$ is closed with open complement~$\ocomplement{R}$ it follows by~\cite[\S III.6.3]{picado2012FramesLocalesTopology} that the preimage~$\swap^{-1}[R]$ is again closed with open complement
	\[
	\ocomplement{R^\op}= \swap^{-1}(\ocomplement{R}).
	\]
\end{remark}

Using this, since complements reverse order, the expected condition for anti-symmetry~${R\cap R^\op \subseteq \Delta}$ of a closed localic relation should be
\[
	\ocomplement{\Delta}\sqleq \ocomplement{R}\vee \swap^{-1}(\ocomplement{R}).
\]
To prove this, first we show that in a Stone locale the diagonal relation~$\Delta$ is indeed closed, and that~$\ocomplement{\Delta}$ is generated by the rectangles~$c\tensor \neg c$ for~${c\in K(\Opens X)}$.

\begin{lemma}
	\label{lemma:diagonal on Stone is closed}
	The diagonal~$\Delta$ on a Stone locale~$X$ is closed with open~complement
	\[
	\ocomplement{\Delta} = \bigvee\{c\tensor \neg c: c\in K(\Opens X)\}.
	\]
\end{lemma}
%\begin{proof}
%	That the diagonal of a Stone locale is closed follows by~\cite[\S III.1]{johnstone1982StoneSpaces}. Its open complement is given by~(\cite[\S V.2]{picado2012FramesLocalesTopology}):
%	\[
%	v:= \bigvee \{x\tensor y: x,y\in \Opens X: x\wedge y =\bot\}.
%	\]
%	Clearly~$\ocomplement{\Delta}\sqleq v$, so we are left to show~$v\sqleq \ocomplement{\Delta}$. Since~$\Opens X$ is in particular coherent, elements can be written as the join over the compact elements they contain. Then, given~$x,y\in \Opens X$ such that~$x\wedge y =\bot$, for any compact~$c\in K(\Opens X)$ with~$c\sqleq x$ we have~$c\wedge y=\bot$, or equivalently~$y\sqleq \neg c$. Hence~$c\tensor y \sqleq c\tensor \neg c \sqleq \ocomplement{\Delta}$, so taking joins over~$c$ we get~$x\tensor y \sqleq \ocomplement{\Delta}$, and hence~$v\sqleq \ocomplement{\Delta}$.
%\end{proof}
\begin{proof}
	The diagonal of a Stone locale is closed by~\cite[\S III.1]{johnstone1982StoneSpaces}. On the other hand, from~\cref{example:diagonals are fixed points} we know that~$\Delta$ is a fixed point of~$\cone\dashv\rel$ generated by~${\dc=\id}$, so by~\cref{proposition:closed iff Rdc closed}~$(\Opens X,\id)$ is closed as a conic frame and~$\Delta$ has open complement~$O^\Dc_\bot$, which has the desired form by~\cref{corollary:closed conic Stone frame ODc from Kdc}.
\end{proof}

\begin{lemma}
	\label{lemma:anti-symmetric iff complements contain diagonal}
	On a Stone locale~$X$, a closed relation~$R$ is anti-symmetric~iff
	\[
	\ocomplement{\Delta}\sqleq \ocomplement{R}\vee\swap^{-1}(\ocomplement{R}).
	\]
\end{lemma}
\begin{proof}
	By definition,~$R$ is anti-symmetric precisely when~$R\cap R^\op\subseteq \Delta$. By~\cref{lemma:Rop is swap preimage}, we see~$R^\op$ is the closed sublocale with open complement~$\swap^{-1}(\ocomplement{R})$.
	Together with the complement of the diagonal from~\cref{lemma:diagonal on Stone is closed} and the fact that binary meets of closed sublocales equal the closed sublocale defined by the binary join of the underlying opens~\cite[\S III.6.1.5]{picado2012FramesLocalesTopology}, we get the desired~inclusion.
\end{proof}

Now that we have a characterisation of anti-symmetry in terms of the open complements, we show using axiom~(E4) that the relation~$R_{\dc}$ induced by an Esakia frame is anti-symmetric. Observe that on an ordered Stone locale~$(X,R)$, the open complements~$\ocomplement{\Delta}$ and~$\ocomplement{R}$ are both generated by rectangles~$c\tensor \neg c$, the first with~$c\in K(\Opens X)$ and the second with~$c\in K_{\down}(\Opens X)$. Since~$K(\Opens X)$ is Boolean generated by~$K_{\down}(\Opens X)$~(\cref{corollary:Kdown generates K}), this suggests the following more general pattern. For any subset~$D\subseteq L$ of a frame, define in~$L\tensor L$:
\[
	\ocomplement{D}:= \bigvee\{d\tensor \neg d: d\in D\}.
\]

\begin{lemma}
	\label{lemma:boolean generated complement is swap}
	Let~$L$ be a frame with Boolean subalgebra~$G$. If the subset~$D\subseteq G$ Boolean generates~$G$, then
	\[
	\ocomplement{G} = \ocomplement{D}\vee \swap^{-1}(\ocomplement{D}).
	\]
\end{lemma}
\begin{proof}
	Abbreviate~$v:=\ocomplement{D}\vee\swap^{-1}(\ocomplement{D})$. Since~$D\subseteq G$ we get~$\ocomplement{D}\sqleq \ocomplement{G}$, and since~$G$ is a Boolean algebra it is closed under~$\neg$, so~$\swap^{-1}(\ocomplement{G})=\ocomplement{G}$. Thus we immediately get the inclusion~$v\sqleq \ocomplement{G}$.
	
	For the converse inclusion, consider the following subset:
	\[
	S:=\{a\in G: a\tensor \neg a \sqleq v\}.
	\]
	We claim that~$S$ is a Boolean subalgebra of~$G$. For that, we need to show that~$S$ is closed under finite meets, finite joins, and negation. First, observe that trivially~$\top,\bot\in S$, since the corresponding rectangles are just the bottom. Next, observe that if~$a\in S$ then~$\neg a\in S$, since applying the involutive map~$\swap^{-1}$ to~$a\tensor \neg a \sqleq v$ we get
	\[
	\neg a\tensor \neg \neg a= \neg a\tensor a \sqleq \swap^{-1}(v) = \swap^{-1}(\ocomplement{D})\vee \swap^{-1}\swap^{-1}(\ocomplement{D})= v.
	\]
	We are left to show~$S$ is closed under binary meets and joins. For meets of~$a,b\in S$:
	\[
	(a\wedge b)\tensor \neg(a\wedge b)
	=
	(a\wedge b)\tensor (\neg a\vee \neg b)
	\sqleq
	(a\tensor \neg a)\vee (b\tensor \neg b)\sqleq v,
	\]
	and similarly for joins:
	\[
	(a\vee b)\tensor \neg(a\vee b)
	=
	(a\vee b) \tensor (\neg a\wedge \neg b)
	\sqleq
	(a\tensor \neg a)\vee (b\tensor \neg b)
	\sqleq v.
	\]
	Thus~$S$ is a Boolean subalgebra of~$G$, and note moreover that it contains~$D$, since for~$d\in D$ we have by construction~$d\tensor \neg d\sqleq \ocomplement{D}\sqleq v$. Since~$D$ Boolean generates~$G$, it follows~$S=G$, and therefore~$\ocomplement{G}\sqleq v$, as was left to show.
\end{proof}

\begin{lemma}
	\label{lemma:closed conic Stone implies Rdc anti-symmetric}
	Let~$(L,\dc)$ be a closed conic Stone frame such that~$\dc$ is a closure operator. If~$\Kdc(L)$ Boolean generates~$K(L)$, then~$R_{\dc}$ is anti-symmetric.
\end{lemma}
\begin{proof}
	Since~$(L,\dc)$ is closed,~\cref{proposition:closed iff Rdc closed} says~$R_{\dc}$ is closed with open complement~$\ocomplement{R_{\dc}}=O^\Dc_\bot$, which in turn by~\cref{corollary:closed conic Stone frame ODc from Kdc} is seen to equal
	\[
	\ocomplement{R_{\dc}}=\bigvee\{c\tensor \neg c:c\in \Kdc(L)\}.
	\]
	Since~$\Kdc(L)$ generates~$K(L)$ as a Boolean algebra we get from~\cref{lemma:boolean generated complement is swap} that
	\[
	\bigvee\{c\tensor \neg c: c\in K(L)\}=\ocomplement{R_{\dc}}\vee \swap^{-1}(\ocomplement{R_{\dc}}),
	\]
	where by~\cref{lemma:diagonal on Stone is closed} the left-hand side is precisely the open complement~$\ocomplement{\Delta}$ of the diagonal relation, so anti-symmetry follows from~\cref{lemma:anti-symmetric iff complements contain diagonal}.
\end{proof}

Combining this with results from the previous sections, we can prove that Esakia frames induce Esakia locales under~$\rel$.

\begin{proposition}
	\label{proposition:esakia frame gives esakia locale}
	If~$(E,\dc)$ is an Esakia frame then~$\rel(E,\dc)$ is an Esakia locale.
\end{proposition}
\begin{proof}
	Let~$X$ be the locale defined by~$\Opens X:=E$. It is a Stone locale since~$E$ is a Stone frame. Since~$(E,\dc)$ is closed, by~\cref{proposition:closed iff Rdc closed} the relation~$R_{\dc}$ is closed with open complement~$\ocomplement{R_{\dc}}=O^\Dc_\bot$. In turn, by~\cref{corollary:KDown = Kdown,corollary:closed conic Stone frame ODc from Kdc} we get
	\[
	\ocomplement{R_{\dc}}=\bigvee\{c\tensor \neg c:c\in K_{\down}(E)\},
	\]
	so~$R_{\dc}$ satisfies the localic Priestley separation axiom. Since~$\dc$ is a closure operator,~\cref{proposition:Rdc preorder iff closure operator} shows~$R_{\dc}$ is a preorder, and since~$\Kdc(E)$ Boolean generates~$K(E)$, \cref{lemma:closed conic Stone implies Rdc anti-symmetric} shows~$R_{\dc}$ is a partial order. By~\cref{theorem:induced relation}~$R_{\dc}$ has open source, so~$(X,R_{\dc})$ is an Esakia locale.
\end{proof}

%~~~~~~~~~~~~~~~~~~~~~~~~~~~~~~~~~~~~~~~~~
\subsection{The categories}
\label{section:ELoc and EFrm}
As in~\cite{bezhanishvili2010BitopologicalDualityDistributive,bezhanishvili2023FrametheoreticPerspectiveEsakia} and as alluded to in~\cref{section:esakia duality}, we distinguish between different types of morphisms between Esakia spaces.
First, we have monotone functions, just preserving the order. Algebraically, these are just the distributive lattice morphisms between Heyting algebras. Second, we have p-morphisms, which correspond to Heyting morphisms.
Analogously, we distinguish between two different categories of Esakia locales, the first using just internally monotone locale maps, and the second using the internal notion of p-morphism from~\cref{section:internal p morphisms}.

Recall that a subcategory is called~\emph{full} when it contains the same morphisms (but a subclass of objects), and~\emph{wide} when it contains the same objects (but a subclass of morphisms).

\begin{definition}
	\label{definition:ELoc}
	We denote by:
	\begin{enumerate}
		\item $\ELoc_\weak$ the full subcategory of~$\rosLoc$ containing Esakia locales;
		\item $\ELoc_\strong$ the wide subcategory of~$\ELoc_\weak$ containing the p-morphisms.
	\end{enumerate}
\end{definition}

Mirroring to the frame-theoretic side, using the characterisation of p-morphisms in terms of cones in~\cref{corollary:p morphism on Esakia locales}, we make the following definition.

\begin{definition}
	\label{definition:EFrm}
	We denote by:
	\begin{enumerate}
		\item $\EFrm_\weak$ the full subcategory of~$\dcFrm$ whose objects are Esakia frames;
		\item $\EFrm_\strong$ the wide subcategory of~$\EFrm_\weak$ whose morphisms are the \emph{strong conic morphisms}~$h\colon (E,\dc)\to (F,\dcx)$ satisfying~$h\circ \dc = \dcx\circ h$.
	\end{enumerate}
\end{definition}

Now that the categories are defined, we show that the object-level assignments in~\cref{proposition:esakia frame gives esakia locale,proposition:esakia locale gives esakia frame} extend to functors. Subscripts~${\mode\in\{\weak,\strong\}}$ will indicate the category with either `weak' or `strong' morphisms, so for example~$\EFrm_\mode$ denotes either~$\EFrm_\weak$ or~$\EFrm_\strong$, and similarly for the other categories and functors defined below. Note that these categories have the same objects, but the strong category has a restricted class of morphisms.

%~~~~~~~~~~~~~~~~~~~~~~~~~~~~~~
\subsection{The equivalence}
\label{section:equivalence EFrm ELoc}
We show that the functors~$\cone,\rel$ from~\cref{section:conic frames} between conic frames and open source localic relations restrict to functors between Esakia frames and Esakia locales, and define a dual equivalence.

\begin{proposition}
	\label{proposition:ELoc to EFrm}
	The functor~$\cone$ restricts to functors~$\cone_\mode \colon \ELoc_\mode\to \EFrm_\mode^\op$.
\end{proposition}
\begin{proof}
	That~$\cone_\mode$ are well defined on objects is precisely~\cref{proposition:esakia locale gives esakia frame}, and that~$\cone_\weak$ is well defined on morphisms is just~\cref{proposition:functor rosLoc to ucdcFrm}. That~$\cone_\strong$ sends p-morphisms to strong conic morphisms follows from~\cref{lemma:p morphism implies cone equality}.
\end{proof}

\begin{proposition}
	\label{proposition:EFrm to ELoc}
	The functor~$\rel$ restricts to functors~$\rel_\mode\colon \EFrm_\mode^\op\to \ELoc_\mode$.
\end{proposition}
\begin{proof}
	That~$\rel_\mode$ are well defined on objects is precisely~\cref{proposition:esakia frame gives esakia locale}, and that~$\rel_\weak$ is well defined on morphisms is just~\cref{proposition:functor ucdcFrm to rosLoc}. Explicitly, if we have a conic morphism~$f^{-1}\colon (\Opens Y,\dcx)\to (\Opens X,\dc)$, then under~$\rel$ this induces an internally monotone map of locales~$f\colon (X,R_{\dc})\to (Y,R_{\dcx})$. If in addition~$f^{-1}$ is a strong conic morphism, it follows that~$f$ is a p-morphism by~\cref{corollary:p morphism on Esakia locales}. Hence~$\rel_\strong$ is well-defined on morphisms.
\end{proof}

We are now ready to establish the dual equivalence between Esakia locales and Esakia frames.
Recall that an \emph{equivalence} between categories~$\cat{C}$ and~$\cat{D}$ consists of functors~$F\colon \cat{C}\to \cat{D}$ and~$G\colon \cat{D}\to \cat{C}$, together with natural isomorphisms~${\eta\colon \id \Longrightarrow G\circ F}$ and~${\epsilon\colon F\circ G\Longrightarrow\id}$, referred to as the \emph{unit} and \emph{counit} of the equivalence, respectively (though we do not require them to satisfy the triangle identities). A \emph{dual} equivalence between~$\cat{C}$ and~$\cat{D}$ is an equivalence between~$\cat{C}^\op$ and~$\cat{D}$. See for example~\cite[p.18]{maclane1998CategoriesWorkingMathematician} for more details.

\begin{remark}
	\label{remark:strategy}
	In this informal remark we discuss the general proof strategy for obtaining the equivalences, which also applies in subsequent sections. First note that we need to establish equivalences~$F_\mode\colon \cat{C}_\mode\simeq\cat{D}_\mode\cocolon\mspace{1mu}G_\mode$ in both the weak and strong settings.

	Suppose for the moment that the weak equivalence~$F_\weak\colon \cat{C}_\weak\simeq \cat{D}_\weak\cocolon\mspace{1mu}G_\weak$ is already established. Since the functors~$F_\strong,G_\strong$ are just the restrictions of the functors~$F_\weak,G_\weak$ to wide subcategories, to obtain the equivalence in the strong setting it suffices to show that the unit and counit of~$\cat{C}_\weak\simeq\cat{D}_\weak$ restrict to natural isomorphisms for the strong functors. For that, in turn, it suffices that the components of the unit and counit not only define isomorphisms in the weak category, but also in the strong category. For this paper, this can be proved case-by-case.

	Moreover, it happens that the functors~$F_\weak,G_\weak$ are restrictions of yet a third pair of functors~$F,G$ that are already part of an equivalence~$\cat{C}\simeq\cat{D}$~(in this section the fixed points of~$\cone\dashv\rel$). In that case, since~$\cat{C}_\weak,\cat{D}_\weak$ are \emph{full} subcategories of~$\cat{C},\cat{D}$, the unit and counit for~$F,G$ can be used directly to define a unit and counit for~$F_\weak,G_\weak$, and the equivalence~$\cat{C}_\weak\simeq\cat{D}_\weak$ follows immediately.
\end{remark}

\begin{theorem}
	\label{theorem:ELoc EFrm equivalence}
	The adjunction~$\cone\dashv\rel$ restricts to equivalences
	\[
	\cone_\mode\colon\ELoc_\mode
	\simeq
	\EFrm_\mode^\op\cocolon \rel_\mode.
	\]
\end{theorem}
\begin{proof}
	The adjunction~$\cone\dashv\rel$ restricts to an equivalence of categories between its fixed points. Both~$\ELoc_\mode$ and~$\EFrm^\op_\mode$ are subcategories of the respective fixed points. Hence, it follows immediately that the functors~$\cone_\weak$ and~$\rel_\weak$ define an equivalence~$\ELoc_\weak\simeq \EFrm_\weak^\op$ between the full subcategories.
	
	For the strong morphisms, it suffices to show that the components of the unit and counit of~$\cone\dashv\rel$ define isomorphisms in~$\ELoc_\strong$ and~$\EFrm_\strong$, respectively. For the counit this is immediate, since its components are just identities. On the other hand, if~$(X,R)$ is an Esakia locale, its unit component is given by
	\[
	\eta_{(X,R)}= \id_X\colon (X,R)\longrightarrow (X,\RDown),
	\]
	whose underlying frame map defines a strong conic morphism, since the relations have the same cones~(\cref{theorem:induced relation}). Hence by~\cref{corollary:p morphism on Esakia locales}~$\eta_{(X,R)}$ and its inverse are p-morphisms. Thus we conclude that~$\cone_\strong$ and~$\rel_\strong$ define an equivalence~$\ELoc_\strong\simeq\EFrm_\strong^\op$.
\end{proof}

%~~~~~~~~~~~~~~~~~~~~~~~~~~~~~
\section{Heyting and Esakia frames}
\label{section:heyting frames and esakia frames}
Heyting frames were introduced in~\cite{bezhanishvili2023FrametheoreticPerspectiveEsakia} to provide~\citetitle{bezhanishvili2023FrametheoreticPerspectiveEsakia}. However, the dualities established there are ultimately in relation to Esakia \emph{spaces}. In the present work, they will play the point-free analogue of Heyting algebras, just as coherent frames play the point-free analogue of distributive lattices in~\cite{townsend1997LocalicPriestleyDuality}. To obtain a localic Esakia duality, the goal of this section will thus be to establish an equivalence between Heyting frames and Esakia frames.

We start by recalling the definition of Heyting frames from~\cite[Definition~3.1]{bezhanishvili2023FrametheoreticPerspectiveEsakia}, and then defining the relevant categories of Heyting frames and Heyting algebras. Instead of constructing the equivalence between Esakia frames and Heyting frames directly, we construct an equivalence between Esakia frames and Heyting algebras first.
Starting from an Esakia frame~$(E,\dc)$, following the spatial intuition outlined in~\cref{section:esakia duality}, the resulting Heyting algebra is the set of compact upper elements~$\Kdc(E)$ with Heyting implication given by 
\[
a\to b := \neg \dc (a\wedge \neg b).
\]
Conversely, starting with a Heyting algebra~$A$ we follow the strategy of~\cite{townsend1997LocalicPriestleyDuality} and define the Stone frame of ideals~$\idl(\bool{A})$ of the free Boolean algebra~$\bool{A}$ generated by~$A$, on which we construct a cone~$\dc{A}$ that turns it into an Esakia frame. \cref{theorem:EFrm Heyt equivalence} proves that these assignments constitute an equivalence, and with the help of the equivalence between Heyting frames and Heyting algebras from Theorems~3.11 and~3.14 in~\cite{bezhanishvili2023FrametheoreticPerspectiveEsakia} we then get the desired equivalence:
\[
\EFrm\simeq \Heyt\simeq \HFrm.
\]

\subsection{The categories}
\label{section:HFrm and Heyt}
To set the scene, we recall the notion of a Heyting frame from~\cite[Definition~3.1]{bezhanishvili2023FrametheoreticPerspectiveEsakia}. Following the same remarks at the start of~\cref{section:ELoc and EFrm}, we distinguish between categories with `weak' morphisms and `strong' morphisms of both Heyting frames and Heyting algebras.

\begin{definition}
	\label{definition:heyting frame}
	A \emph{Heyting frame} is a coherent frame~$H$ such that~$K(H)$ is a Heyting subalgebra of~$H$.
\end{definition}

Recall from~\cref{section:ideals} that~$\CohFrm$ is the category of coherent frames with coherent morphisms, and~$\DLat$ is the category of distributive lattices with distributive lattice morphisms.

\begin{definition}
	\label{definition:HFrm}
	We denote by:
	\begin{enumerate}
		\item $\HFrm_\weak$ the full subcategory of~$\CohFrm$ containing Heyting frames;
		\item $\HFrm_\strong$ the wide subcategory of~$\HFrm_\weak$ containing coherent morphisms that restrict to Heyting algebra morphisms on the compact elements.
	\end{enumerate}
\end{definition}

\begin{definition}
	\label{definition:Heyt}
	We denote by:
	\begin{enumerate}
		\item $\Heyt_\weak$ the full subcategory of~$\DLat$ containing Heyting algebras;
		\item $\Heyt_\strong$ the wide subcategory of~$\Heyt_\weak$ containing Heyting algebra maps.
	\end{enumerate}
\end{definition}

\begin{remark}
	In~\cite[Definitions~3.5]{bezhanishvili2023FrametheoreticPerspectiveEsakia} these categories are denoted by~$\HFrm^-$ and~$\HFrm$, respectively, with the same convention for categories of Heyting algebras. In general, they use a superscript~$-$ to denote categories with monotone or `weak' morphisms, and the absence of a superscript denotes the categories with p-morphisms and analogous `strong' morphisms. They treat other classes of stronger morphisms, which we will not consider here.
\end{remark}

It was shown in~\cite[Theorem~3.2]{bezhanishvili2023FrametheoreticPerspectiveEsakia} that a coherent frame~$H$ is Heyting iff~$K(H)$ is a Heyting algebra, so the Heyting structure is uniquely determined by the ambient one of~$H$. 
% Recall that for any distributive lattice~$A$ the frame of ideals~$\idl(A)$ is coherent~\cite[\S II.3.2]{johnstone1982StoneSpaces}.
Thus, a distributive lattice~$A$ is a Heyting algebra iff~$\idl(A)$ is a Heyting frame~\cite[Corollary~3.3]{bezhanishvili2023FrametheoreticPerspectiveEsakia}. To obtain a Heyting frame from an Esakia frame~$(E,\dc)$, it thus suffices to construct a Heyting algebra. For this reason, we shall first construct an equivalence~$\EFrm_\mode\simeq\Heyt_\mode$. To get the desired equivalence between Esakia frames and Heyting frames, we use the equivalences~$\HFrm_\mode\simeq\Heyt_\mode$ from~\cite{bezhanishvili2023FrametheoreticPerspectiveEsakia}, which are in turn restrictions of the equivalence~$\CohFrm\simeq\DLat$.

\begin{remark}
	As mentioned in the~\nameref{section:introduction}, this strategy mirrors the construction of a coherent frame out of an ordered Stone locale in~\cite{townsend1997LocalicPriestleyDuality}. Although Townsend does not state this explicitly, it can be seen that first a distributive lattice~$K_{\down}(\Opens X)$ is constructed, the ideals of which form the desired coherent frame. Hence the localic Priestley duality factorises similarly through~$\DLat$. Here this also has the benefit of simplifying some computations by not having to keep up with implicit identifications between principal ideals and lattice elements.
\end{remark}

\subsection{Esakia frame to Heyting algebra}
\label{section:esakia frame to heyting algebra}
Following the spatial intuition in~\cref{section:esakia locales and esakia frames}, the Heyting algebra induced by an Esakia frame~$(E,\dc)$ should consist of the clopen upper sets, which we saw in the frame-theoretic language corresponds to~$\Kdc(E)$ from~\cref{definition:Kdc}. The next sequence of results shows that this is indeed a Heyting algebra.

\begin{lemma}
	\label{lemma:Kdc disjointness}
	Let~$(L,\dc)$ be a conic Stone frame with~$\dc$ inflationary. If~$c\in \Kdc(L)$ and~$d\in K(L)$ then
	\[
	\Dc(c,d)=\bot\quad\iff\quad c\wedge d=\bot.
	\]
\end{lemma}
\begin{proof}
	If~$\Dc(c,d)=\bot$ then~$c\wedge d\sqleq c\wedge \dc d=\bot$ by inflationarity. Conversely, if~$c\wedge d=\bot$ then~$d\sqleq \neg c$, so~$\dc d\sqleq \dc\neg c=\neg c$, or equivalently~${\Dc(c,d)=\bot}$.
\end{proof}

\begin{lemma}
	\label{lemma:Kdc is lattice}
	If~$(L,\dc)$ is a conic Stone frame with~$\dc$ inflationary, then~$\Kdc(L)$ is a bounded sublattice of~$K(L)$.
\end{lemma}
\begin{proof}
	Since~$L$ is Stone,~$K(L)$ is a Boolean algebra, and we use the De Morgan rules freely. Observe first that~$\dc\bot=\bot$ since~$\dc$ preserves joins, and~$\dc\top=\top$ since~$\dc$ is inflationary. Hence
	\[
	\dc\neg\bot=\dc\top=\top =\neg\bot
	\quad\text{and}\quad
	\dc\neg\top=\dc\bot=\bot=\neg\top,
	\]
	so~$\bot,\top\in \Kdc(L)$. For meets, take~$a,b\in \Kdc(L)$ and calculate:
	\[
	\dc\neg (a\wedge b)
	=
	\dc (\neg a \vee \neg b)
	=
	\dc\neg a \vee \dc \neg b
	=
	\neg a \vee \neg b
	=
	\neg (a\wedge b),
	\]
	so~$a\wedge b\in \Kdc(L)$. For joins, calculate:
	\[
	\dc\neg (a\vee b)
	=
	\dc (\neg a \wedge \neg b)
	\sqleq
	\dc \neg a
	=
	\neg a,
	\]
	and similarly~$\dc\neg(a\vee b)\sqleq \neg b$. Hence~$\dc\neg(a\vee b) \sqleq \neg a \wedge \neg b = \neg (a\vee b)$, which with inflationarity gives equality, and hence~$a\vee b\in \Kdc(L)$.
\end{proof}

For any Esakia frame~$(E,\dc)$ we thus get a distributive lattice~$\Kdc(E)$. To motivate its Heyting structure, recall from~\cref{section:esakia duality} that the Heyting implication of two clopen upsets~$A,B$ in an Esakia space~$X$ is defined as:
\[
A\to B
:=
X\setminus \down (A\setminus B).
\]

\begin{proposition}
	\label{proposition:Kdc Heyting algebra}
	If~$(E,\dc)$ is an Esakia frame, then~$\Kdc(E)$ is a Heyting algebra~with
	\[
	a\to b:= \neg\dc(a\wedge \neg b).
	\]
\end{proposition}
\begin{proof}
	By~\cref{lemma:Kdc is lattice} we know~$\Kdc(E)$ is a distributive lattice, so it suffices to show the displayed formula makes it into a Heyting algebra.
	
	Fix~$a,b\in \Kdc(E)$. First we show~$a\to b\in \Kdc(E)$. Since~$E$ is Stone,~$a\wedge \neg b$ is compact, and so by~\cref{lemma:dc preserves compact elements}~$\dc(a\wedge \neg b)$ is compact, and hence~$a\to b$ is compact. Then we calculate:
	\begin{align*}
		\dc \neg (a\to b)
		&=
		\dc\neg \neg\dc (a\wedge \neg b)
		&&(\text{by definition})
		\\&=
		\dc\dc (a\wedge \neg b)
		&&(\text{$E$ Stone})
		\\&=
		\dc(a\wedge \neg b)
		&&(\text{$\dc$ idempotent})
		\\&=
		\neg (a\to b),
		&&(\text{by definition})
	\end{align*}
	so~$a\to b\in \Kdc(E)$, as claimed. To show this defines a Heyting implication, take an element~${c\in \Kdc(E)}$ and calculate
	\begin{align*}
		c\sqleq a\to b
		&\iff
		c\sqleq \neg \dc(a\wedge \neg b)
		&&(\text{by definition})
		\\&\iff
		\Dc(c,a\wedge \neg b)=\bot
		\\&\iff
		c\wedge a\wedge \neg b=\bot
		&&(\text{\cref{lemma:Kdc disjointness}})
		\\&\iff
		c\wedge a\sqleq \neg\neg b=b.
		&&(\text{$E$ Stone})\qedhere
	\end{align*}
\end{proof}

\begin{remark}
	\label{remark:closure algebras}
	For an Esakia frame~$(E,\dc)$, we get a Boolean algebra~$K(E)$, and by~\cref{lemma:dc preserves compact elements} the cone~$\dc$ restricts to a (finite) join-preserving closure operator. Thus the pair~$(K(E),\dc)$ forms a \emph{closure algebra} in the sense of~McKinsey and~Tarski~\cite{mckinsey1944AlgebraTopology,mckinsey1946ClosedElementsClosure}. See also~\cite[\S 2.2]{esakia2019HeytingAlgebrasDuality}. The dual \emph{interior operation}~$\square:=\neg\dc\neg$ further gives an \emph{interior algebra}~$(K(E),\square)$. Note then that the set of compact upper elements~$\Kdc(E)$ just equals the set of fixed points of~$\square$ in~$K(E)$, which is well-known to be a Heyting algebra with implication:
	\[
	\neg \dc (a\wedge \neg b)
	=
	a\to b
	=
	\square(\neg a\vee b).
	\]
	See for example~\cite[Proposition~2.2.4]{esakia2019HeytingAlgebrasDuality}, the proof of which is essentially that of~\cref{proposition:Kdc Heyting algebra}.

	Note also that axiom~(E4) is not needed in the construction of the Heyting algebra~$\Kdc(E)$ out of~$(E,\dc)$. However, it will be crucial in~\cref{section:counit patch Kdc} to reconstruct an Esakia frame from its induced Heyting algebra.
\end{remark}

%\begin{corollary}
%	\label{corollary:esakia frame to heyting frame}
%	If~$(E,\dc)$ is an Esakia frame then~$\idl \Kdc(E)$ is a Heyting frame.
%\end{corollary}

%~~~~~~~~~~~~~~~~~~~~~~~~~~~~~~~~~~
\subsection{Free Boolean algebras}
\label{section:free boolean algebras}
\label{remark:bool}
Starting with a Heyting algebra, to obtain an Esakia frame, we first need to construct a Stone frame, whose compact elements form a \emph{Boolean} algebra. To do this, we follow the strategy of~\cite{townsend1997LocalicPriestleyDuality} and consider the free Boolean algebra generated by the Heyting algebra, the ideals of which then generate the desired Stone frame. In this intermezzo, we recall and establish notation around free Boolean algebras.

% \begin{remark}
	A constructive proof of the existence of free Boolean algebras on distributive lattices was first obtained in~\cite{peremans1957EmbeddingDistributiveLattice}. See~\cite{maietti2012InductionPrincipleConsequence} for an alternative, more recent constructive proof.
	This constitutes a functor~${\bool\colon\DLat\to\Bool}$ into the category of Boolean algebras, sending a distributive lattice~$A$ to the Boolean algebra~$\bool{A}$. Denoting by~$U\colon \Bool\to\DLat$ the forgetful functor, the unit of the adjunction~$\bool\dashv U$ is denoted by~$\beta\colon \id\Longrightarrow U\circ \bool$, where for each distributive lattice~$A$ we get an order embedding via a distributive lattice morphism~${\beta_A\colon A\to \bool{A}}$~\cite[Theorem~32]{maietti2012InductionPrincipleConsequence}. Recall that an \emph{order embedding} $f\colon (P,\leq)\to (Q,\preccurlyeq)$ is a function between posets such that~$x\leq y$ holds iff~$f(x)\preccurlyeq f(y)$ holds. In what follows, we shall mostly suppress the forgetful functor~$U$.
	
	For a distributive lattice~$A$ and elements~$a,b\in A$, we denote the corresponding generators in the free Boolean algebra~by
	\[
	u_a:=\beta_A(a)
	\qquad\text{and}\qquad
	\ell_b:= \neg u_b.
	\]
	For a map~$g\colon A\to B$ of distributive lattices, the induced map~$\bool{g}\colon \bool{A}\to \bool{B}$ is the unique morphism of Boolean algebras characterised by~${\beta_B\circ g = \bool{g}\circ \beta_A}$, or explicitly:
	\[
	\bool{g}(u_a)= u_{g(a)}
	\qquad\text{and}\qquad
	\bool{g}(\ell_b)= \ell_{g(b)}.
	\]
% \end{remark}

The spatial intuition is that~$u_a$ represents the clopen upper set generated by~$a$, and~$\ell_b$ the clopen lower set generated by~$b$. Spatially~(\cref{section:esakia duality}), the clopen upper and lower sets generate a topology via `patches'. Here we have the lattice-theoretic analogue.

\begin{definition}
	\label{definition:patch}
	A \emph{patch} is an element in~$\bool{A}$ of the form~$u_a\wedge \ell_b$, for~$a,b\in A$.
\end{definition}

\begin{lemma}
	\label{lemma:patches generate bool}
	Every element in~$\bool{A}$ is a finite join of patches.
\end{lemma}
\begin{proof}
	By the disjunctive normal-form theorem of Boolean algebras~\cite[\S 4.24]{davey2002IntroductionLatticesOrder}, every element is a join of elements of the form
	\[
	\bigwedge_{i\in I} u_{a_i}\wedge \bigwedge_{j\in J} \ell_{b_j},
	\]
	for finite index sets~$I,J$. Since~$\beta_A$ is a morphism of distributive lattices we see that
	\[
	\bigwedge_{i\in I} u_{a_i} = u_{\bigwedge_{i\in I} a_i}
	\quad\text{and}\quad
	\bigwedge_{j\in J} \ell_{b_j}
	=
	\neg\bigvee_{j\in J} u_{b_j}
	=
	\neg u_{\bigvee_{j\in J} b_j}
	=
	\ell_{\bigvee_{j\in J}b_j},
	\]
	from which the claim follows.
\end{proof}

In our setting~$A$ will not just be a distributive lattice, but a Heyting algebra. We collect the following elementary lemma that shows how the Heyting structure interacts with the Boolean generators.

\begin{lemma}
	\label{lemma:patch inclusion implies Heyting inclusion}
	Consider a family~$a_i,b_i,c,d\in A$ indexed by finite~$I$. Then:
	\[
	u_c\wedge \ell_d\sqleq \bigvee_{i\in I} (u_{a_i}\wedge \ell_{b_i})
	\quad\implies\quad
	\ell_{c\to d} \sqleq \bigvee_{i\in I} \ell_{a_i\to b_i}.
	\]
\end{lemma}
\begin{proof}
	First observe that the term~$\bigvee_{i\in I}\ell_{a_i\to b_i}$ on the right-hand side can be rewritten according to
	\[
	\bigvee_{i\in I}\ell_{a_i\to b_i}
	=
	\bigvee_{i\in I} \neg u_{a_i\to b_i}
	=
	\neg\bigwedge_{i\in I} u_{a_i\to b_i}
	=
	\neg u_{\bigwedge_{i\in I}{a_i\to b_i}}
	=
	\ell_{\bigwedge_{i\in I}{a_i\to b_i}}.
	\]
	Abbreviating~$e:=\bigwedge_{i\in I} (a_i\to b_i)$, we thus need to prove~$e\sqleq c\to d$. This in turn holds iff~$e\wedge c\sqleq d$, for which it suffices to prove that~$u_{e\wedge c}=u_e\wedge u_c\sqleq u_d$. To see this, note that by construction for each~$i\in I$ we have~$e\sqleq a_i\to b_i$, and so using that~$a\wedge (a\to b) = a\wedge b$ in any Heyting algebra~(\cite[\S III.3.1.1]{picado2012FramesLocalesTopology}) we get
	\begin{align*}
		u_e\wedge u_{a_i}\wedge \ell_{b_i}
		&\sqleq
		u_{a_i\to b_i}\wedge u_{a_i}\wedge \ell_{b_i}
		\\&=
		u_{a_i}\wedge u_{b_i}\wedge \ell_{b_i}
		&&
		\\&=
		\bot.
		&&(\ell_{b_i}:=\neg u_{b_i})
	\end{align*}
	Taking joins, this implies~$u_e\wedge \bigvee_{i\in I}(u_{a_i}\wedge \ell_{b_i})=\bot$. With the stated hypothesis this gives~$u_e\wedge u_c\wedge \ell_d=\bot$, which is seen to be equivalent to the desired inclusion.
\end{proof}

%~~~~~~~~~~~~~~~~~~~~~~~~~~~
\subsection{Heyting algebra to Esakia frame}
\label{section:heyting algebra to esakia frame}
As outlined, starting with a Heyting algebra~$A$, which we fix for this section, we construct a Stone frame~$E_A$ as the frame of ideals of the free Boolean algebra~$\bool{A}$ generated by~$A$. This follows~\cite{townsend1997LocalicPriestleyDuality}.
On this frame, we construct a cone~$\dc{A}$. Then follows a string of lemmas that computes the cone on~(principal ideals of) patches, proves it preserves all joins, proves it forms a closure operator, identifies its compact upper elements, allowing us to finally conclude that~$(E_A,\dc{A})$ is an Esakia frame.

\begin{definition}
	If~$A$ is a Heyting algebra, we define the following Stone frame:
	\[
	E_A:= \idl(\bool{A}).
	\]
\end{definition}

To obtain an Esakia frame, we need to construct a join-preserving cone~$\dc{A}$ defined on~$E_A$. Its definition can be motivated as follows. First, patches~${u_a\wedge \ell_b}$ are analogous to set-theoretic complements~$A\setminus B$ for clopen upsets~$A,B$ in an Esakia space~$X$. Since~(principal ideals of) patches generate the frame~$E_A$, and since~$\dc{A}$ is to be join-preserving, it suffices to specify~$\dc{A}$ on patches. For this, consider how the down cone~$\down$ of an Esakia space acts on a spatial patch~$A\setminus B$, which is obtained from the Heyting implication as follows:
\[
A\to B := X\setminus \down (A\setminus B)
\quad\implies\quad
\down (A\setminus B) = X\setminus (A\to B).
\]
Translating this to frame-theoretic language, we expect the following equation to hold, forming the intuition of the full definition of~$\dc{A}$:
\[
\dc{A}\pidl{u_a\wedge \ell_b} = \pidl{\ell_{a\to b}}.
\]

\begin{definition}
	\label{definition:dcA}
	For a Heyting algebra~$A$ we define the following function:
	\[
	\dc{A}\colon E_A\longrightarrow E_A;
	\qquad
	\dc{A} (I):= \bigvee\{\pidl{\ell_{a\to b}}
	:
	a,b\in A
	:
	\pidl{u_a\wedge \ell_b}\sqleq I
	\}.
	\]
\end{definition}

The next sequence of lemmas establishes that~$(E_A,\dc{A})$ indeed defines an Esakia frame. The proof starts by calculating~$\dc{A}$ on finite joins of patches, which in the unary case confirms that~$\dc{A}$ adheres to the spatial intuition outlined above. Then, since~$E_A$ is generated by patches, proving the desired properties of~$\dc{A}$ can be carried out at that level. The chronology follows the list of axioms of an Esakia frame in~\cref{definition:esakia frame}.

\begin{lemma}
	\label{lemma:dcA on patches}
	For any finite family~$a_i,b_i\in A$, with~$i\in I$:
	\[
	\dc{A}\bigvee_{i\in I}\pidl{u_{a_i}\wedge \ell_{b_i}}
	=
	\bigvee_{i\in I}\pidl{\ell_{a_i\to b_i}}.
	\]
\end{lemma}
\begin{proof}
	By construction of~$\dc{A}$ we get the~$\sqgeq$ inclusion, so it suffices to prove the converse. For that, take~$c,d\in A$ such that~$u_c\wedge \ell_d\sqleq \bigvee_{i\in I}(u_{a_i}\wedge \ell_{b_i})$ in~$\bool{A}$. \cref{lemma:patch inclusion implies Heyting inclusion} directly implies that~$\ell_{c\to d}\sqleq \bigvee_{i\in I}\ell_{a_i\to b_i}$, and the result follows by taking principal ideals.
\end{proof}

\begin{corollary}
	\label{corollary:dcA on patch}
	For all~$a,b\in A$ we have
	\[
	\dc{A}\pidl{u_a\wedge \ell_b}= \pidl{\ell_{a\to b}},
	\quad\text{and in particular}\quad
	\dc{A}\pidl{\ell_b}=\pidl{\ell_b}.
	\]
\end{corollary}
\begin{proof}
	The first equation is just a special case of~\cref{lemma:dcA on patches}. For the second, under~$\beta_A$ the top element~$\top\in A$ corresponds to~${u_\top = \top}$, so
	\[
	\dc{A}\pidl{\ell_b}=\dc{A}\pidl{u_\top\wedge \ell_b}=\pidl{\ell_{\top\to b}}=\pidl{\ell_b}.
	\qedhere
	\]
\end{proof}

\begin{lemma}
	\label{lemma:dcA preserves compacts}
	$\dc{A}$ preserves compact elements.
\end{lemma}
\begin{proof}
	Take a compact element~$p\in K(E_A)\cong \bool{A}$. By~\cref{lemma:patches generate bool} this can be written as a finite join of the form~$p=\bigvee_{i\in I}\pidl{u_{a_i}\wedge \ell_{b_i}}$ for~$a_i,b_i\in A$, so applying~\cref{lemma:dcA on patches} directly gives
	\[
	\dc{A} (p) = \bigvee_{i\in I}\pidl{\ell_{a_i\to b_i}},
	\]
	which as a finite join of compact elements is again compact~(\cref{lemma:finite joins of compact elements are compact}).
\end{proof}

\begin{lemma}
	\label{lemma:dcA preserves finite joins of compacts}
	$\dc{A}$ preserves finite joins of compact elements.
\end{lemma}
\begin{proof}
	That~$\dc{A}$ preserves~$\bot$ follows by applying~\cref{lemma:dcA on patches} to the empty indexing family. For binary joins, take~$p,q\in K(E_A)\cong\bool{A}$, and use~\cref{lemma:patches generate bool} to write them as finite joins of patches:
	\[
	p=\bigvee_{i\in I}\pidl{u_{a_i}\wedge \ell_{b_i}}
	\quad\text{and}\quad
	q=\bigvee_{j\in J}\pidl{u_{c_j}\wedge \ell_{d_j}}.
	\]
	Again applying~\cref{lemma:dcA on patches} this gives:
	\[
	\dc{A}(p\vee q)
	=
	\bigvee_{i\in I}\pidl{\ell_{a_i\to b_i}} \vee\bigvee_{j\in J}\pidl{\ell_{c_j\to d_j}}
	=
	\dc{A}(p)\vee\dc{A}(q).
	\qedhere
	\]
\end{proof}

For~$p\in \bool{A}$, write~$\pidl{p}$ for the principal ideal it generates, regarded as a compact element of~$E_A$.

\begin{remark}
	Recall that in any frame of ideals~$\idl(A)$ the join of a family~$(I_\lambda)_{\lambda\in\Lambda}$ is the ideal generated by their union. Explicitly, this means~$p\in \bigvee_{\lambda\in\Lambda} I_\lambda$ iff there exists finite~$F\subseteq \bigcup_{\lambda\in\Lambda}I_\lambda$ so that~$p\sqleq \bigvee F$. See for instance the proof of~\cite[Proposition~II.3.2]{johnstone1982StoneSpaces}.
\end{remark}

\begin{lemma}
	\label{lemma:dcA decomposes as join of compacts}
	For every~$I\in E_A$ we have:
	\[
	\dc{A}(I)=\bigvee_{p\in I}\dc{A}\pidl{p}.
	\]
\end{lemma}
\begin{proof}
	Note first that~$\dc{A}$ is monotone, so for any~$p\in I$ we have the inclusion~${\pidl{p}\sqleq I}$ in~$E_A$, and we immediately get~$\dc{A}\pidl{p}\sqleq \dc{A}(I)$. Hence the~$\sqgeq$ direction holds, and we are left to prove the converse.
	
	For that, recall~$\dc{A}(I)$ is defined as the join over~$\pidl{\ell_{a\to b}}$ for those~$a,b\in A$ such that~${\pidl{u_a\wedge \ell_b}\sqleq I}$. Now, this inclusion of ideals holds iff~$u_a\wedge\ell_b\in I$, so the join~$\bigvee_{p\in I}\dc{A}\pidl{p}$ on the right-hand side already contains all the terms~$\dc{A}\pidl{u_a\wedge \ell_b}$, which by~\cref{corollary:dcA on patch} are just equal to~$\pidl{\ell_{a\to b}}$, giving the desired inclusion.
\end{proof}

\begin{lemma}
	\label{lemma:dcA preserves joins}
	$\dc{A}$ preserves joins.
\end{lemma}
\begin{proof}
	Take an arbitrary family~$(I_\lambda)_{\lambda\in \Lambda}$ of ideals in~$\bool{A}$, and put~$I:=\bigvee_{\lambda\in\Lambda}I_\lambda$. Since~$\dc{A}$ is monotone, based on the decomposition of~$\dc{A}(I)$ in~\cref{lemma:dcA decomposes as join of compacts} it suffices to prove that~$\dc{A}\pidl{p}\sqleq \bigvee_{\lambda\in\Lambda}\dc{A}(I_\lambda)$ for every~$p\in I$.
	
	To see this, since~$I$ is the ideal generated by the union of~$(I_\lambda)_{\lambda\in\Lambda}$, for every~${p\in I}$ there exists a finite subset~$F\subseteq\bigcup_{\lambda\in\Lambda}I_\lambda$ such that~$p\sqleq\bigvee F$. Since~$\dc{A}$ is monotone and preserves finite joins of compact elements~(\cref{lemma:dcA preserves finite joins of compacts}), we get an inclusion:
	\[
	\dc{A}\pidl{p}
	\sqleq
	\dc{A}\bigvee_{q\in F}\pidl{q}
	=
	\bigvee_{q\in F}\dc{A}\pidl{q}.
	\]
	Lastly, now using~\cref{lemma:dcA decomposes as join of compacts} for each term~$\dc{A}(I_\lambda)$, we see that this is in turn included in the desired term:
	\[
	\bigvee_{q\in F}\dc{A}\pidl{q}
	\sqleq
	\bigvee_{q\in\bigcup_{\lambda\in\Lambda} I_\lambda} \dc{A}\pidl{q}
	=
	\bigvee_{\lambda\in\Lambda}\bigvee_{q\in I_\lambda}\dc{A}\pidl{q}
	=
	\bigvee_{\lambda\in\Lambda}\dc{A}(I_\lambda).
	\qedhere
	\]
\end{proof}

\begin{lemma}
	\label{lemma:dcA inflationary}
	$\dc{A}$ is inflationary.
\end{lemma}
\begin{proof}
	Given~\cref{lemma:dcA decomposes as join of compacts}, it suffices to show~$\dc{A}$ is inflationary on compact elements. In turn, by~\cref{lemma:patches generate bool,lemma:dcA preserves finite joins of compacts,corollary:dcA on patch} it suffices to prove~$\dc{A}$ is inflationary on patches~$\pidl{u_a\wedge \ell_b}$, for~$a,b\in A$, which by~\cref{corollary:dcA on patch} even further simplifies to the inclusion~${\pidl{u_a\wedge\ell_b}\sqleq \dc{A}\pidl{u_a\wedge\ell_b}=\pidl{\ell_{a\to b}}}$. To see this, note that~$a\wedge (a\to b)\sqleq b$ holds in any Heyting algebra, which after the embedding~$\beta_A$ and taking principal ideals rewrites to the desired inclusion.
\end{proof}

\begin{lemma}
	\label{lemma:dcA subidempotent}
	$\dc{A}$ is subidempotent.
\end{lemma}
\begin{proof}
	We know from~\cref{corollary:dcA on patch} that~$\dc{A}$ fixes all elements of the form~$\pidl{\ell_c}$.
	However, for arbitrary~$I\in E_A$, the element~$\dc{A}(I)$ is defined just as a join over such elements, so since~$\dc{A}$ preserves joins~(\cref{lemma:dcA preserves joins}) it follows that~${\dc{A}^2=\dc{A}}$. In particular, it is subidempotent.
\end{proof}

\begin{lemma}
	\label{lemma:dcA closed}
	$(E_A,\dc{A})$ is a closed conic frame.
\end{lemma}
\begin{proof}
	Denote the reduced cone of~$\dc{A}$ by~$\Dc$. We need to show that for all~${I,J\in E_A}$:
	\[
	I\tensor J\sqleq \Dc(I,J)\tensor J \vee O^\Dc_\bot.
	\]
	Combining coherence with~\cref{lemma:patches generate bool}, the ideals~$I,J$ can be written as the join over the principal ideals generated by the patches they contain. It thus suffices to show that the defining inclusion of a closed frame above holds in the case that~$I,J$ are principal patches. That is, for arbitrary~$a,b,c,d\in A$ we need with~${p:=\pidl{u_a\wedge \ell_b}}$ and~$q:= \pidl{u_c\wedge \ell_d}$ that:
	\[
	p\tensor q\sqleq \Dc(p,q)\tensor q\vee O^\Dc_\bot.
	\]
	We simplify this expression. First, using~\cref{corollary:dcA on patch} we get
	\[
	\Dc(p,q)= p\wedge \dc{A}(q)=p \wedge \pidl{\ell_{c\to d}}.
	\]
	Using the Boolean structure, we can thus decompose~$p=\Dc(p,q)\vee (p\wedge \pidl{u_{c\to d}})$. On the other hand, since~$\dc{A}$ is inflationary~(\cref{lemma:dcA inflationary}) we get~${q\sqleq \pidl{\ell_{c\to d}}}$. Combining these two, we get:
	\[
	p\tensor q 
	=
	\Dc(p,q)\tensor q \vee (p\wedge \pidl{u_{c\to d}})\tensor q
	\sqleq
	\Dc(p,q)\tensor q \vee \pidl{u_{c\to d}}\tensor \pidl{\ell_{c\to d}}.
	\]
	We are thus left to show that this latter term~$\pidl{u_{c\to d}}\tensor \pidl{\ell_{c\to d}}$ is contained in~$O^\Dc_\bot$.
	To see this, recall from~\cref{corollary:dcA on patch} that~$\dc{A}\pidl{\ell_e}=\pidl{\ell_e}$ for all~${e\in A}$, so it follows
	\[
	\Dc(\pidl{u_e},\pidl{\ell_e})= \pidl{u_e}\wedge \dc{A}\pidl{\ell_e}=\pidl{u_e}\wedge\pidl{\ell_e} = \pidl{u_e}\wedge\neg \pidl{u_e}=\bot,
	\]
	and hence each rectangle of the form~$\pidl{u_e}\tensor \pidl{\ell_e}$ is contained in~$O^\Dc_\bot$.
\end{proof}

\begin{lemma}
	\label{lemma:KdcA}
	$\Kdc{A}(E_A)=\{\pidl{u_a}:a\in A\}$.
\end{lemma}
\begin{proof}
	For each~$a\in A$ we get using~\cref{corollary:dcA on patch} that:
	\[
	\dc{A}\neg \pidl{u_a} = \dc{A}\pidl{\ell_a} = \pidl{\ell_a} = \neg \pidl{u_a},
	\]
	so every~$\pidl{u_a}$ lies in~$\Kdc{A}(E_A)$. Conversely, take~$p\in \Kdc{A}(E_A)$. Using~\cref{lemma:patches generate bool} we can find~$a_i,b_i\in A$, indexed by finite~$I$, such that~${\neg p = \bigvee_{i\in I}\pidl{u_{a_i}\wedge \ell_{b_i}}}$. Applying~$\dc{A}$ and using~\cref{lemma:dcA on patches} we get~$\neg p =\dc{A}(\neg p)= \bigvee_{i\in I}\pidl{\ell_{a_i\to b_i}}$, where the latter term again simplifies to~$\pidl{\ell_e}$ with~$e:=\bigwedge_{i\in I}(a_i\to b_i)$. Hence~${p=\pidl{u_e}}$.
\end{proof}

\begin{proposition}
	\label{proposition:heyting algebra to esakia frame}
	$(E_A,\dc{A})$ is an Esakia frame.
\end{proposition}
\begin{proof}
	This is the culmination of the previous sequence of lemmas:
	\begin{enumerate}[label=(E\arabic*),start=0]
		\item $(E_A,\dc{A})$ is a conic frame by~\cref{lemma:dcA preserves joins};
		\item $E_A$ is Stone by construction;
		\item $\dc{A}$ is a closure operator by~\cref{lemma:dcA inflationary,lemma:dcA subidempotent};
		\item $(E_A,\dc{A})$ is closed by~\cref{lemma:dcA closed};
		\item $\Kdc{A}(E_A)$ generates~$K(E_A)$ by~\cref{lemma:patches generate bool,lemma:KdcA}.\qedhere
	\end{enumerate}
\end{proof}

\begin{remark}
	\label{remark:boolean esakia frame}
	Consider the special case that~$A$ is a Boolean algebra. In that case the free Boolean embedding identifies~$\beta_A\colon A\cong \bool{A}$, and for~$a,b\in A$ the Heyting implication~$a\to b= \neg a \vee b$ translates to~$u_a\wedge \ell_b = \ell_{a\to b}$. Thus we see with~\cref{corollary:dcA on patch} that~$\dc{A}\pidl{u_a\wedge \ell_b}= \pidl{u_a\wedge \ell_b}$, so in fact~$\dc{A}=\id_{E_A}$. The induced localic relation~$R_{\dc{A}}$ is hence just the diagonal relation~$\Delta$, which conforms with the spatial situation~\cite[Corollary~3.4.9]{esakia2019HeytingAlgebrasDuality} that an Esakia space~$(X,\leq)$ corresponds to a Boolean algebra iff its partial order~$\leq$ is the identity relation.
\end{remark}

%~~~~~~~~~~~~~~~~~~~~~~~~~~~~~~~~~~~~~~~
\subsection{The functors}
Next, we construct functors between Esakia frames and Heyting algebras. The functors are the obvious extensions of the object-level assignments in~\cref{section:esakia frame to heyting algebra,section:heyting algebra to esakia frame}. For morphisms, we prove that an Esakia frame map restricts to a morphism between the sets of compact upper elements, while a map between Heyting algebras first extends via the functor~$\bool$ to a map between the free Boolean algebras, and then via~$\idl$ to a map between the corresponding Stone~frames.

\begin{lemma}
	\label{lemma:Kdc preserved by conic morphism}
	If~$k\colon (E,\dc)\to (F,\dcx)$ is a map in~$\EFrm_\weak$ then~${k[\Kdc(E)]\subseteq K_{\dcx}(F)}$.
\end{lemma}
\begin{proof}
	Take~$c\in \Kdc(E)$, so~$c\in K(E)$ and $\dc\neg c = \neg c$. Since~$E$ is Stone, the compact elements are precisely the complemented elements, which are preserved by frame maps. Hence~$k(c)\in K(F)$ with complement~$\neg k(c)=k(\neg c)$. Using that~$k$ is a conic morphism we get an inclusion:
	\[
	\dcx \neg k(c)
	=
	\dcx k (\neg c)
	\sqleq
	k(\dc \neg c)
	=
	k(\neg c)
	= \neg k(c),
	\]
	and since~$\dcx$ is inflationary this is actually an equality, so~$k(c)\in K_{\dcx}(F)$.
\end{proof}

\begin{proposition}
	\label{proposition:EFrm to Heyt}
	There are functors
	\begin{align*}
		\Kdc^\mode
		\colon
		\EFrm_\mode
		&\longrightarrow
		\Heyt_\mode;
		\\
		(E,\dc)&\longmapsto \Kdc(E)
		\\
		k &\longmapsto k|_{\Kdc(E)}.
	\end{align*}
\end{proposition}
\begin{proof}
	That~$\Kdc^\mode$ are well defined on objects follows by~\cref{proposition:Kdc Heyting algebra}. For morphisms, first let~$k\colon (E,\dc)\to (F,\dcx)$ be a conic morphism between Esakia frames. By~\cref{lemma:Kdc preserved by conic morphism},~$k$ restricts to a map~${\Kdc(E)\to K_{\dcx}(F)}$ of distributive lattices, and so defines a morphism in~$\Heyt_\weak$. Thus~$\Kdc^\weak$ is well defined on morphisms.
	
	Now suppose further that~$k$ is a strong conic morphism. Recalling the Heyting structure of~$\Kdc(E)$ from~\cref{proposition:Kdc Heyting algebra}, for any~$a,b\in \Kdc(E)$ we calculate
	\begin{align*}
		k(a\to b)
		&=
		k\left(\neg \dc (a\wedge \neg b)\right)
		&&(\text{by definition})
		\\&=
		\neg k \dc(a\wedge \neg b)
		&&(k\text{~preserves~}\neg)
		\\&=
		\neg \dcx k(a\wedge \neg b)
		&&(k\text{~strong conic})
		\\&=
		\neg\dcx \left(k(a)\wedge \neg k(b)\right)
		&&(k\text{~preserves~}\wedge,\neg)
		\\&=
		k(a)\to k(b),
		&&(\text{by definition})
	\end{align*}
	so~$k$ in fact restricts to a map~$\Kdc(E)\to K_{\dcx}(F)$ of Heyting algebras, meaning it defines a morphism in~$\Heyt_\strong$. Thus~$\Kdc^\strong$ is well defined on morphisms.
\end{proof}

To define the functor in the converse direction, we recall the notation of the free Boolean functor~$\bool$ from~\cref{remark:bool}. A map~$g\colon A\to B$ in~$\Heyt_\mode$ is always a map of distributive lattices, and so induces a map between the free Boolean algebras~$\bool{g}\colon \bool{A}\to\bool{B}$ characterised by the formula~$\bool{g}(u_a) = u_{g(a)}$. Applying the functor~$\idl$ then gives a frame map~$\idl(\bool{g})\colon E_A\to E_B$ between the Esakia frames.

\begin{proposition}
	\label{proposition:Heyt to EFrm}
	There are functors
	\begin{align*}
		\patch_\mode
		\colon
		\Heyt_\mode
		&\longrightarrow
		\EFrm_\mode;
		\\
		A&\longmapsto (E_A,\dc{A});
		\\
		g &\longmapsto \idl(\bool{g}).
	\end{align*}
\end{proposition}
\begin{proof}
	That~$\patch_\mode$ are well defined on objects is precisely~\cref{proposition:heyting algebra to esakia frame}.
	Next, take a map of distributive lattices~$g\colon A\to B$, defining a morphism in~$\Heyt_\weak$. By construction, we get a coherent frame morphism~$\idl(\bool{g})\colon E_A\to E_B$, and we need to show that this is a conic morphism. Explicitly, we need for all~$I\in E_A$ that there is an inclusion
	\[
	\dc{B}\left(\down \bool{g}[I]\right)
	=
	\dc{B} \circ \idl(\bool{g})(I)
	\sqleq
	\idl(\bool{g})\circ \dc{A}(I)
	=
	\down \bool{g}\left[\dc{A}(I)\right].
	\]
	However, by~\cref{lemma:patches generate bool} any ideal~$I$ can be written as a join over patches~$\pidl{u_a\wedge \ell_b}$, for~$a,b\in A$, and since~$\dc{A},\dc{B}$ preserve joins~(\cref{lemma:dcA preserves joins}) it suffices to prove this inclusion for~$I=\pidl{u_a\wedge\ell_b}$. For that, first calculate the left-hand side:
	\begin{align*}
		\dc{B}\idl(\bool{g})(\pidl{u_a\wedge \ell_b})
		&=
		\dc{B}\pidl{\bool{g}(u_a\wedge \ell_b)}
		&&(\idl\text{~on principal ideals})
		\\&=
		\dc{B}\pidl{\bool{g}(u_a)\wedge\bool{g}(\ell_b)}
		&&(\bool{g}\text{~preserves~}\wedge)
		\\&=
		\dc{B}\pidl{u_{g(a)}\wedge \ell_{g(b)}}
		&&(\text{characterisation of~}\bool{g})
		\\&=
		\pidl{\ell_{g(a)\to g(b)}},
		&&(\text{\cref{corollary:dcA on patch}})
	\end{align*}
	and second, calculate the right-hand side:
	\begin{align*}
		\idl(\bool{g})\left(\dc{A}\pidl{u_a\wedge \ell_b}\right)
		&=
		\idl(\bool{g})(\pidl{\ell_{a\to b}})
		&&(\text{\cref{corollary:dcA on patch}})
		\\&=
		\pidl{\bool{g}(\ell_{a\to b})}
		&&(\idl\text{~on principal ideals})
		\\&=
		\pidl{\ell_{g(a\to b)}}.
		&&(\text{characterisation of~}\bool{g})
	\end{align*}
	In any Heyting algebra we have~$(a\to b)\wedge a=a\wedge b\sqleq b$~(see~\cite[\S III.3.1.1]{picado2012FramesLocalesTopology}), hence~$g(a\to b)\wedge g(a)\sqleq g(b)$, and so~${g(a\to b)\sqleq g(a)\to g(b)}$. This implies~$\pidl{\ell_{g(a)\to g(b)}}\sqleq \pidl{\ell_{g(a\to b)}}$, as was to be shown. Hence~$\patch_\weak$ is well defined on morphisms.

	In the case that~$g$ is a morphism in~$\Heyt_\strong$, we get by definition the equalities~${g(a\to b)=g(a)\to g(b)}$, and the same argument as above shows that we get the strong conic equality~${\dc{B}\circ \idl(\bool{g})=\idl(\bool{g})\circ \dc{A}}$. Thus~$\idl(\bool{g})$ is a strong conic morphism, and hence~$\patch_\strong$ is well defined on morphisms.
\end{proof}

%~~~~~~~~~~~~~~~~~~~~
\subsection{The unit}
In the next two sections we establish that these two functors define an equivalence. We construct the unit and counit explicitly:
\[
\pidl{\beta} \colon \id\Longrightarrow \Kdc^\mode\circ \patch_\mode
\quad\text{and}\quad
\epsilon \colon 
\patch_\mode\circ \Kdc^\mode \Longrightarrow \id.
\]

For the unit, the component on a Heyting algebra~$A$ is to be an isomorphism of Heyting algebras of the form
\[
\pidl{\beta}_A\colon A
\longrightarrow
K_{\dc{A}}(E_A)
=
\{
I\in K(E_A): \dc{A}\neg I= \neg I
\}.
\]
To define it, first we use the natural isomorphism~$\pidl{-}\colon \id\Longrightarrow K\circ \idl$ from~\cref{section:ideals} that identifies elements in a distributive lattice~$A$ with their principal ideals. Note that this map behaves well with respect to Heyting structure.

\begin{lemma}
	\label{lemma:pidl preserves heyting implication}
	If~$A$ is Heyting, then~$\pidl{-}_A$ is a Heyting isomorphism.
	% For any Heyting algebra~$A$, the map~$\pidl{-}_A\colon A\to K(\idl(A))$ is a Heyting algebra isomorphism.
\end{lemma}
\begin{proof}
	It suffices to prove that~$\pidl{-}_A$ preserves Heyting implication, which follows immediately from the proof of~\cite[Lemma~3.12]{bezhanishvili2023FrametheoreticPerspectiveEsakia}, giving the formula:
	\[
	\pidl{a}\to \pidl{b}
	=
	\pidl{a\to b}.
	\qedhere
	\]
\end{proof}

Second, we need to identify the free Boolean algebra~$\bool{A}$ with its principal ideals~$K(E_A)$. This is described by the whiskering of the natural isomorphism~$\pidl{-}$ with the functor~$\bool$ to give a natural isomorphism:
\[
\pidl{-}_{\bool} \colon \bool \Longrightarrow K\circ \idl\circ \bool.
\]

Lastly, we use the natural map~$\beta\colon \id\Longrightarrow U\circ \bool$ that embeds a distributive lattice into its free Boolean algebra. Combined, this gives the desired natural transformation. Recall~\cref{remark:strategy} for the proof strategy.

\begin{proposition}
	There are natural isomorphisms
	\[
	\pidl{\beta}:= \pidl{-}_{\bool}\circ \beta
	\colon
	\id
	\Longrightarrow
	\Kdc^\mode \circ\patch_\mode.
	\]
\end{proposition}
\begin{proof}
	Clearly~$\pidl{\beta}$ is natural as the composition of two natural transformations. We are left to show that it does not just land in~$K\circ\idl\circ\bool$, but in~$\Kdc^\mode\circ\patch_\mode$, and that its components define Heyting isomorphisms.

	To see this, note for a Heyting algebra~$A$ the component of~$\pidl{\beta}$ is given by
	\[
	\pidl{\beta}_A\colon A\longrightarrow K(E_A);
	\qquad
	a\longmapsto \pidl{u_a}.
	\]
	However, we already saw in~\cref{lemma:KdcA} that this function defines an order embedding into its image~$\Kdc{A}(E_A)$. Thus~$\pidl{\beta}_A\colon A\to \Kdc{A}(E_A)$ defines an isomorphism of distributive lattices. To show that it defines a Heyting isomorphism, we are left to show the Heyting implication is preserved, which can be seen as follows:
	\begin{align*}
		\pidl{\beta}_{\!A}(a\to b)
		&=
		\pidl{u_{a\to b}}
		&&(\text{by definition})
		\\&=
		\neg\pidl{\ell_{a\to b}}
		&&(\text{\cref{lemma:pidl preserves heyting implication}})
		\\&=
		\neg \dc{A}\pidl{u_a\wedge \ell_b}
		&&(\text{\cref{corollary:dcA on patch}})
		\\&=
		\neg\dc{A}(\pidl{u_a}\wedge \neg\pidl{u_b})
		&&(\text{\cref{lemma:pidl preserves heyting implication}})
		\\&=
		\pidl{u_a}\to \pidl{u_b}.
		&&(\text{\cref{proposition:Kdc Heyting algebra}})
		\qedhere
	\end{align*}
\end{proof}

%~~~~~~~~~~~~~~~~~~~
\subsection{The counit}
\label{section:counit patch Kdc}
For the counit~$\epsilon\colon \patch_\mode\circ \Kdc^\mode\Longrightarrow \id$ we need to define strong conic isomorphisms
\[
\epsilon_E
\colon
\left(
	E_{\Kdc(E)},\dc{\Kdc(E)}
\right)
\longrightarrow (E,\dc)
\]
for Esakia frames~$(E,\dc)$. This proceeds in three stages.
The first step is to identify the free Boolean algebra generated by the Heyting algebra~$\Kdc(E)$ with the algebra of compact elements~$K(E)$. This Boolean isomorphism~${\boolrec_E\colon \bool{\Kdc(E)}\to K(E)}$ is the canonical map induced by the distributive lattice inclusion~$\Kdc(E)\hookrightarrow K(E)$. Subsequently, the map~$\boolrec_E$ lifts to an isomorphism~${E_{\Kdc(E)}\to \idl K(E)}$ of frames, and finally we use the natural isomorphism~${\jidl_E\colon \idl K(E)\to E}$ to obtain the isomorphism with the original frame~$E$.

\begin{lemma}
	\label{lemma:generating iff free Boolean algebra}
	Let~$i\colon A\to B$ be an order reflecting morphism of distributive lattices into a Boolean algebra. If~$i[A]$ generates~$B$, then the canonical comparison map~${\boolrec\colon \bool{A}\to B}$ is an isomorphism of Boolean algebras.
\end{lemma}
\begin{proof}
	Denote by~$\beta_A\colon A\to\bool{A}$ the order embedding into the free Boolean algebra of~$A$, with~$\boolrec$ the unique morphism of Boolean algebras such that $\boolrec\circ \beta_A= i$. It suffices to show~$\boolrec$ is bijective. First, since the image of~$i$ generates~$B$ as a Boolean algebra,~$\boolrec$ is surjective.
	
	On the other hand, to show injectivity of a Boolean morphism it suffices to show it has trivial kernel, meaning~$\boolrec(x)=\bot$ implies~${x=\bot}$. Namely, if~$\boolrec(x)=\boolrec(y)$ then~$\boolrec(x\wedge\neg y)=\bot$, and hence~$x\wedge \neg y = \bot$, so~$x\sqleq y$. A symmetric argument gives~$y\sqleq x$, so~$x=y$. To prove the comparison map~$\boolrec$ has trivial kernel, take an element~${x\in\bool{A}}$ in normal form
	\[
	x=\bigvee_{j\in J}(\beta_A(a_j)\wedge\neg\beta_A(b_j)),
	\]
	for some families~$a_j,b_j\in A$ indexed by finite~$J$. Then,~$\boolrec(x)=\bot$ implies that for all $j\in J$:
	\begin{align*}
		i(a_j)\wedge\neg i(b_j)
		&=
		\boolrec\circ\beta_A(a_j)\wedge\neg\boolrec\circ\beta_A(b_j)
		&&(\boolrec\circ\beta_A=i)
		\\&=
		\boolrec\left(\beta_A(a_j)\wedge\neg\beta_A(b_j)\right)
		&&(\boolrec\text{~preserves~}\wedge,\neg)
		\\&\sqleq
		\boolrec(x)
		=
		\bot,
	\end{align*}
	or equivalently:~$i(a_j)\sqleq i(b_j)$ in~$B$. Since~$i$ is an order embedding this gives an inclusion~${a_j\sqleq b_j}$ in~$A$, and now applying~$\beta_A$ we get~$\beta_A(a_j)\wedge\neg\beta_A(b_j)=\bot$, so~$x=\bot$. This completes the proof.
\end{proof}

\begin{lemma}
	\label{lemma:bool Kdc = K}
	For every Esakia frame~$(E,\dc)$ there is a Boolean isomorphism:
	\[
	\boolrec_E\colon \bool{\Kdc(E)}\longrightarrow K(E)
	\quad\text{such that}\quad
	\boolrec_E(u_a)= a\in \Kdc(E).
	\]
\end{lemma}
\begin{proof}
	By~\cref{lemma:Kdc is lattice} we get an order embedding~$i\colon \Kdc(E)\hookrightarrow K(E)$ of the distributive sublattice~$\Kdc(E)$ into the Boolean algebra~$K(E)$. Axiom~(E4) of Esakia frames then guarantees via the preceding~\cref{lemma:generating iff free Boolean algebra} that the canonical map~$\boolrec_E\colon \bool{\Kdc(E)}\to K(E)$ is an isomorphism of Boolean algebras. In particular, by universality it satisfies~$\boolrec_E\circ \beta_{\Kdc(E)}=i$, which under the conventions set out at the start of~\cref{remark:bool} evaluates on~$a\in \Kdc(E)$ to~$\boolrec_E(u_a)= a$, as desired.
\end{proof}

Applying~$\idl$, this gives a family of isomorphisms~${\idl(\boolrec_E)\colon E_{\Kdc(E)}\to \idl K(E)}$.
To complete the construction, recall that~$E$ can be recovered as the frame of ideals over~$K(E)$ via the natural isomorphism~$\jidl\colon \idl\circ K\Longrightarrow \id$ from~\cref{section:ideals}. It remains to show that the resulting family is natural and respects the conic structure.

\begin{proposition}
	There are natural isomorphisms
	\[
	\epsilon
	\colon
	\patch_\mode\circ\Kdc^\mode
	\Longrightarrow
	\id.
	\]
\end{proposition}
\begin{proof}
	Fixing an Esakia frame~$(E,\dc)$, we take the Boolean algebra isomorphism~$\boolrec_E\colon \bool{\Kdc(E)}\to K(E)$ from~\cref{lemma:bool Kdc = K}, which in turn induces a frame isomorphism
	\[
	\idl(\boolrec_E)
	\colon
	\idl(\bool{\Kdc(E)})
	\longrightarrow
	\idl(K(E)),
	\]
	where the domain is the underlying frame of the Esakia frame~$(E_{\Kdc(E)},\dc{\Kdc(E)})$.
	In turn, we compose with the natural isomorphism~$\jidl\colon \idl\circ K\Longrightarrow \id$ to obtain an isomorphism of frames
	\[
		\epsilon_E
		:=
		\jidl_E\circ \idl(\boolrec_E)
		\colon
		\idl(\bool{\Kdc(E)})
		\longrightarrow
		\idl (K(E))
		\longrightarrow E.
	\]
	Observe that since~$\jidl_E$ takes joins over ideals, it identifies principal ideals with underlying elements, and so together with~\cref{lemma:bool Kdc = K} we get for all~${a,b\in \Kdc(E)}$:
	\[
	\tag{$\star$}
	\label{eq:epsilon}
	\epsilon_E(\pidl{u_a})= a
	\qquad\text{and}\qquad
	\epsilon_E(\pidl{\ell_b})= \neg b.
	\]

	To show~$\epsilon$ is natural, take a conic map~$k\colon (E,\dc)\to (F,\dcx)$ in~$\EFrm_\weak$. We need to show that the outside rectangle commutes:
	% https://q.uiver.app/#q=WzAsNixbMCwwLCJcXGlkbChcXGJvb2x7XFxLZGMoRSl9KSJdLFswLDEsIlxcaWRsKFxcYm9vbHtcXEtkYyhGKX0pIl0sWzEsMCwiXFxpZGwoSyhFKSkiXSxbMSwxLCJcXGlkbChLKEYpKSJdLFsyLDAsIkUiXSxbMiwxLCJGLCJdLFs0LDUsImsiXSxbMCwxLCJcXGlkbChcXGJvb2x7a30pIiwyXSxbMCwyLCJcXGlkbChcXGJhclxcaW1hdGhfRSkiXSxbMSwzLCJcXGlkbChcXGJhclxcaW1hdGhfRikiLDJdLFsyLDQsIlxcamlkbF9FIl0sWzMsNSwiXFxqaWRsX0YiLDJdLFsyLDMsIlxcaWRsKGt8X3tLKEUpfSkiLDJdXQ==
	\[\begin{tikzcd}[cramped]
		{\idl(\bool{\Kdc(E)})} & {\idl(K(E))} & E \\
		{\idl(\bool{\Kdc(F)})} & {\idl(K(F))} & {F,}
		\arrow["{\idl(\boolrec_E)}", from=1-1, to=1-2]
		\arrow["{\idl(\bool{k})}"', from=1-1, to=2-1]
		\arrow["{\jidl_E}", from=1-2, to=1-3]
		\arrow["{\idl(k|_{K(E)})}"', from=1-2, to=2-2]
		\arrow["k", from=1-3, to=2-3]
		\arrow["{\idl(\boolrec_F)}"', from=2-1, to=2-2]
		\arrow["{\jidl_F}"', from=2-2, to=2-3]
	\end{tikzcd}\]
	where for brevity we denote by~$\bool{k}$ the Boolean morphism induced by~$k|_{\Kdc(E)}$.
	The square on the right commutes by naturality of~$\jidl$. On the other hand, the square on the left is just the functor~$\idl$ applied to the square of Boolean algebras:
	% https://q.uiver.app/#q=WzAsNCxbMCwwLCJcXGJvb2x7XFxLZGMoRSl9Il0sWzAsMSwiXFxib29se1xcS2RjKEYpfSJdLFsxLDAsIksoRSkiXSxbMSwxLCJLKEYpLiJdLFsyLDMsImt8X3tLKEUpfSJdLFswLDIsIlxcYmFyXFxpbWF0aF9FIl0sWzEsMywiXFxiYXJcXGltYXRoX0YiLDJdLFswLDEsIlxcYm9vbHtrfSIsMl1d
	\[\begin{tikzcd}[cramped]
		{\bool{\Kdc(E)}} & {K(E)} \\
		{\bool{\Kdc(F)}} & {K(F).}
		\arrow["{\boolrec_E}", from=1-1, to=1-2]
		\arrow["{\bool{k}}"', from=1-1, to=2-1]
		\arrow["{k|_{K(E)}}", from=1-2, to=2-2]
		\arrow["{\boolrec_F}"', from=2-1, to=2-2]
	\end{tikzcd}\]
	Note that~$k|_{K(E)}$ is indeed a morphism of Boolean algebras, since~$k$ is a frame map and hence preserves complemented elements and their complements.
	Any element in~$\bool{\Kdc(E)}$ is generated via finite joins, meets, and negations of elements of the form~$u_a$, for~$a\in \Kdc(E)$, and so to show this square commutes it suffices to check on these elements:
	\begin{align*}
		\boolrec_F \circ \bool{k}(u_a)
		&=
		\boolrec_F (u_{k(a)})
		&&(\text{\cref{remark:bool}})
		\\&=
		k(a)
		&&(\text{\cref{lemma:bool Kdc = K}})
		\\&=
		k\circ \boolrec_E(u_a).
		&&(\text{\cref{lemma:bool Kdc = K}})
	\end{align*}
	
	Hence we see that~$\epsilon$ defines a natural family of frame isomorphisms. We are left to show that~$\epsilon_E$ lies in~$\EFrm_\strong$, meaning:
	\[
	\epsilon_E\circ \dc{\Kdc(E)}
	=
	\dc \circ \epsilon_E.
	\]
	Since~$\idl(\bool{\Kdc(E)})$ is generated by principal ideals of patches~(\cref{lemma:patches generate bool}) and all maps are join-preserving, it suffices to check on elements of the form~${\pidl{u_a\wedge \ell_b}}$, for~${a,b\in \Kdc(E)}$. For this, calculate:
	\begin{align*}
		\epsilon_E\circ \dc{\Kdc(E)}\pidl{u_a\wedge \ell_b}
		&=
		\epsilon_E(\pidl{\ell_{a\to b}})
		&&(\text{\cref{corollary:dcA on patch}})
		\\&=
		\neg(a\to b)
		&&(\text{\ref{eq:epsilon}})
		\\&=
		\dc (a\wedge \neg b)
		&&(\text{\cref{proposition:Kdc Heyting algebra}})
		\\&=
		\dc \circ \epsilon_E(\pidl{u_a\wedge \ell_b}),
		&&(\text{\ref{eq:epsilon}})
	\end{align*}
	proving that~$\epsilon_E$ is a strong conic morphism. Since it is also a frame isomorphism, it follows that it must be an isomorphism in~$\EFrm_\strong$. This completes the proof.
\end{proof}

\begin{theorem}
	\label{theorem:EFrm Heyt equivalence}
	The functors~$\Kdc^\mode$ and~$\patch_\mode$ define equivalences
	\[
	\Kdc^\mode
	\colon
	\EFrm_\mode
	\simeq
	\Heyt_\mode
	\cocolon
	\patch_\mode.
	\]
\end{theorem}

The desired relation between Esakia frames and Heyting frames is now obtained directly by composing the equivalences from~\cref{theorem:EFrm Heyt equivalence} with those between Heyting algebras and Heyting frames established in~\cite{bezhanishvili2023FrametheoreticPerspectiveEsakia}. These equivalences are in turn restrictions of the equivalence between distributive lattices and coherent frames, so we get well-defined functors
\[
\idl\circ \Kdc^\mode
\colon
\EFrm_\mode
\longrightarrow
\HFrm_\mode
\quad\text{and}\quad
\patch_\mode\circ K
\colon
\HFrm_\mode
\longrightarrow
\EFrm_\mode.
\]

\begin{theorem}
	\label{theorem:EFrm HFrm equivalence}
	The functors~$\idl\circ \Kdc^\mode$ and~$\patch_\mode\circ K$ define equivalences
	\[
	\idl\circ \Kdc^\mode
	\colon
	\EFrm_\mode
	\simeq
	\HFrm_\mode
	\cocolon
	\patch_\mode\circ K.
	\]
\end{theorem}
\begin{proof}
	By~Theorems~3.11 and~3.14 in~\cite{bezhanishvili2023FrametheoreticPerspectiveEsakia} we know that the functors~$K,\idl$ restrict to equivalences~$\Heyt_\mode\simeq\HFrm_\mode$. Since equivalences of categories compose~(see for example~\cite[Lemma~1.5.5]{riehl2016CategoryTheoryContext}), combining these with the equivalences defined by~$\Kdc^\mode$ and~$\patch_\mode$ from~\cref{theorem:EFrm Heyt equivalence} shows the functors~$\idl\circ \Kdc^\mode$ and~$\patch_\mode\circ K$ define an equivalence, as desired.
\end{proof}

The unit and the counit of this equivalence are obtained by combining the data of the two composite equivalences. While not needed in the rest of this paper, we include a brief explicit description of their underlying functions for completeness.
For a Heyting frame~$H$ the unit defines an isomorphism
\[
	H
	\xrightarrow{~~\jidl_H^{-1}~~}
	\idl K(H)
	\xrightarrow{~~\idl(\pidl{\beta}_{K(H)})~~}
	% \idl\Kdc{K(H)}(E_{K(H)})=
	\idl \Kdc{K(H)}\!\idl(\bool{K(H)}),
\]
which for~$x\in H$ maps as follows:
\[
	x
	\longmapsto
	\{a\in K(H): a\sqleq x\}
	\longmapsto
	\{\pidl{u_a}: a\in K(H):a\sqleq x\}.
\]
For an Esakia frame~$(E,\dc)$ the counit is defined by an isomorphism
\[
\idl\bool{K\idl\Kdc(E)}
\xrightarrow{~~
	\idl
		\bool{\pidl{-}_{\Kdc(E)}^{-1}}
~~}
\idl\bool{\Kdc(E)}
\xrightarrow{~~\idl(\boolrec_E)~~}
\idl K(E)
\xrightarrow{~~\jidl_E~~}
E,
\]
which for~$a,b\in\Kdc(E)$ maps on generators as
\[
\pidl{u_{\pidl{a}}}
\longmapsto
\pidl{u_a}
\longmapsto
\pidl{a}
\longmapsto
a,
\qquad
\pidl{\ell_{\pidl{b}}}
\longmapsto
\pidl{\ell_b}
\longmapsto
\pidl{\neg b}
\longmapsto
\neg b.
\]

%~~~~~~~~~~~~~~~~~~~~~~~~~~~~~
\section{Restricting Townsend's equivalence}
\label{section:townsends functors}
The preceding two~\cref{section:esakia locales and esakia frames,section:heyting frames and esakia frames} proved equivalences between Esakia locales, Esakia frames, and Heyting frames. Combined, this can be interpreted as a localic Esakia duality~$\HFrm_\mode^\op\simeq \ELoc_\mode$.

In this section we relate this back to Townsend's original localic Priestley duality. Namely, since Heyting frames define subcategories of coherent frames, and Esakia locales define subcategories of ordered Stone locales, we can ask if Townsend's equivalence restricts in turn to a localic Esakia duality. In this section we prove that this is indeed the case, and moreover that this agrees with the duality obtained via Esakia frames in~\cref{section:esakia locales and esakia frames,section:heyting frames and esakia frames}.

To start, we recall the functor~\cite[\S 3]{townsend1997LocalicPriestleyDuality}:
\begin{align*}
	\townsendC
	\colon
	\OStoneLoc
	&\longrightarrow
	\CohFrm^\op;
	\\
	(X,R)&\longmapsto \idl K_{\down}(\Opens X);
	\\
	f &\longmapsto \idl(f^{-1}).
\end{align*}

\begin{proposition}
\label{proposition:ELoc to HFrm}
	The functor~$\townsendC$ restricts to functors~$\townsendC_\mode\colon \ELoc_\mode\to\HFrm^\op_\mode$.
%	\begin{align*}
%		\townsendC_\mode
%		\colon
%		\ELoc_\mode
%		&\longrightarrow
%		\HFrm^\op_\mode;
%		\\
%		(X,R)&\longmapsto \idl K_{\down}(\Opens X);
%		\\
%		f &\longmapsto \idl(f^{-1}).
%	\end{align*}
\end{proposition}
\begin{proof}
	Let~$(X,R)$ be an Esakia locale. Seen as an ordered Stone locale, Townsend's functor~$\townsendC$ assigns to this the coherent frame~$\idl K_{\down}(\Opens X)$. In this case,~$(X,R)$ induces a cone~$\Down$, and by~\cref{corollary:KDown = Kdown} we have that~${K_{\down}(\Opens X)=\KDown(\Opens X)}$. By~\cref{proposition:Kdc Heyting algebra} this is a Heyting algebra, so~$\idl K_{\down}(\Opens X)$ is a Heyting frame~(\cite[Corollary~3.3]{bezhanishvili2023FrametheoreticPerspectiveEsakia}). Hence~$\townsendC_\mode$ are well defined on objects.
	
	Now take a morphism~$f\colon (X,R)\to (Y,Q)$ in~$\ELoc_\weak$. As observed before, any frame morphism preserves complemented elements, so since~$X,Y$ are Stone, the map~$f^{-1}\colon \Opens Y\to \Opens X$ preserves compact elements. Then, following the same proof as~\cref{lemma:Kdc preserved by conic morphism} and again using~\cref{corollary:KDown = Kdown}, we see that~$f^{-1}$ restricts to a map~$K_{\down}(\Opens Y)\to K_{\down}(\Opens X)$ of distributive lattices. Hence we get the desired coherent morphism~$\idl(f^{-1})$, showing~$\townsendC_\weak$ is well defined on morphisms.
	
	Now assume additionally that~$f$ is a p-morphism, so~$f$ is in~$\ELoc_\strong$. We need to show that the restriction~$K_{\down}(\Opens Y)\to K_{\down}(\Opens X)$ of~$f^{-1}$ is a map of Heyting algebras.
	By~\cref{corollary:p morphism on Esakia locales} we have~$f^{-1}\circ \Downsub{Q}= \Downsub{R}\circ f^{-1}$, and using the definition of Heyting implication from~\cref{proposition:Kdc Heyting algebra} we calculate for~$a,b\in K_{\down}(\Opens Y)$:
	\begin{align*}
		f^{-1}(a\to b)
		&=
		f^{-1}\left(\neg\Downsub{Q}(a\wedge \neg b)\right)
		&&(\text{by definition})
		\\&=
		\neg f^{-1}\left(\Downsub{Q}(a\wedge \neg b)\right)
		&&(f^{-1}\text{~preserves~}\neg)
		\\&=
		\neg\Downsub{R}f^{-1}(a\wedge \neg b)
		&&(f\text{~is p-morphism})
		\\&=
		\neg\Downsub{R}\left(f^{-1}(a)\wedge \neg f^{-1}(b)\right)
		&&(f^{-1}\text{~preserves~}\neg,\wedge)
		\\&=
		f^{-1}(a)\to f^{-1}(b).
		&&(\text{by definition})
	\end{align*}
	Hence~$\idl(f^{-1})$ is a morphism in~$\HFrm_\strong$, and~$\townsendC_\strong$ is well defined.
\end{proof}

The converse functor is based on Townsend's assignment~\cite[\S 5]{townsend1997LocalicPriestleyDuality}, which in our notation reads as follows:
\begin{align*}
	\townsendB
	\colon
	\CohFrm^\op
	&\longrightarrow
	\OStoneLoc;
	\\
	L &\longmapsto (\idl (\bool{K(L)}), R_L);
	\\
	h &\longmapsto \idl(\bool{h}).
\end{align*}
We abuse notation somewhat in identifying the frame~$\idl (\bool{K(L)})$ with the corresponding locale. Townsend's localic partial order~$R_L$ is the closed localic relation whose open complement is defined by the directed join
\[
u_{R_L}:=\bigvee\nolimits^{\uparrow}
\left\{
	\bigwedge_{i\in I}\pidl{u_{a_i}}\pftensor \pidl{\ell_{b_i}}
:
\begin{array}{c}
	\bigwedge_{i\in I}(u_{a_i}\vee\ell_{b_i})=\bot,\\
	a_i,b_i\in K(L),\  I\text{~finite}
\end{array}
\right\}.
\]
Here we denote by~$\pftensor$ the \emph{preframe tensor}, which can be translated to the frame-theoretic tensor product as follows:
\[
x\pftensor y = (x\tensor \top) \vee (\top \tensor y)
\quad\text{and}\quad
x\tensor y = (x\pftensor \bot) \wedge (\bot\pftensor y).
\]
Our first task is to translate Townsend's preframe formula for the open complement of~$R_L$ to the language of closed conic frames. This is already alluded to in~\cite[\S 4]{townsend1997LocalicPriestleyDuality}, but we include the proof here for completeness.

\begin{lemma}
	\label{lemma:townsend R open complement}
	If~$L$ is a coherent frame, the open complement of Townsend's~$R_L$~is
	\[
	u_{R_L}
	=
	\bigvee\{\pidl{u_a}\tensor\pidl{\ell_a} : a\in K(L)\}.
	\]
\end{lemma}
\begin{proof}
	Abbreviate the right-hand side by~$U$. Observe first that for all~$a\in K(L)$ we have
	\[
	(u_a\vee \ell_\top)\wedge (u_\bot \vee \ell_a)
	=
	u_a\wedge \ell_a
	=
	\bot,
	\]
	and so~$(\pidl{u_a}\pftensor \bot) \wedge (\bot\pftensor \pidl{\ell_a}) = \pidl{u_a}\tensor \pidl{\ell_a}$ appears in the directed join defining~$u_{R_L}$. Therefore~$U\sqleq u_{R_L}$.
	
	For the converse inclusion, take a term~$\bigwedge_{i\in I}\pidl{u_{a_i}}\pftensor \pidl{\ell_{b_i}}$ for a family~${a_i,b_i\in K(L)}$ indexed by finite~$I$, and satisfying~$\bigwedge_{i\in I}(u_{a_i}\vee \ell_{b_i})=\bot$. Observe then that for all~$a\in K(L)$ we have
	\[
	\pidl{u_a}\tensor \top
	=
	\pidl{u_a}\tensor (\pidl{u_a}\vee \pidl{\ell_a})
	\sqleq
	\top\tensor \pidl{u_a} \vee U,
	\]
	where the last inclusion follows by absorbing the term~$\pidl{u_a}\tensor \pidl{\ell_a}$ into~$U$. Using this, we get for each~$i\in I$:
	\begin{align*}
		\pidl{u_{a_i}}\pftensor \pidl{\ell_{b_i}}
		&=
		(\pidl{u_{a_i}}\tensor \top)\vee (\top\tensor\pidl{\ell_{b_i}})
		&&(\pftensor\text{~to~}\tensor)
		\\&\sqleq
		(\top\tensor \pidl{u_{a_i}})\vee U \vee (\top\tensor \pidl{\ell_{b_i}})
		&&(\text{previous equation})
		\\&=
		\top \tensor \pidl{u_{a_i}\vee\ell_{b_i}}\vee U.
	\end{align*}
	In turn, taking the meet over~$I$ this gives
	\[
		\bigwedge_{i\in I}\pidl{u_{a_i}}\pftensor \pidl{\ell_{b_i}}
		\leq
		\top\tensor\pidl{\bigwedge_{i\in I}(u_{a_i}\vee \ell_{b_i})} \vee U
		=
		\top\tensor\bot \vee U
		=
		U,
	\]
	from which we conclude~$u_{R_L}\sqleq U$, as desired.
\end{proof}

\begin{lemma}
	\label{lemma:townsend R is RdcH}
	If~$H$ is a Heyting frame, then Townsend's~$R_H$ is equal to~$R_{\dc{K(H)}}$.
\end{lemma}
\begin{proof}
	By~\cref{lemma:townsend R open complement}, the relation~$R_H$ has open complement
	\[
	u_{R_H}=\bigvee\{\pidl{u_a}\tensor \pidl{\ell_a}: a\in K(H)\}.
	\]
	On the other hand,~$H$ induces an Esakia frame~$(E_{K(H)},\dc{K(H)})$ by~\cref{proposition:heyting algebra to esakia frame}. For brevity, denote its cone by just~$\dc$. Thus~$R_{\dc}$ is a closed localic relation~(\cref{proposition:closed iff Rdc closed}), which by~\cref{corollary:closed conic Stone frame ODc from Kdc} has open complement
	\[
	O^{\Dc}_\bot =
	\bigvee\{c\tensor \neg c : c\in \Kdc (E_{K(H)})\}.
	\]
	By~\cref{lemma:KdcA}, the set~$\Kdc(E_{K(H)})$ is precisely the image of~$K(H)$ under the free Boolean algebra embedding~$\beta_{K(H)}$, and from this it follows immediately that
	\begin{align*}
	O^\Dc_\bot
	&=
	\bigvee\{\pidl{\beta_{K(H)}(a)}\tensor \neg \pidl{\beta_{K(H)}(a)}:a\in K(H)\}
	\\&=
	\bigvee\{\pidl{u_a}\tensor \pidl{\ell_a}: a\in K(H)\}
	=
	u_{R_H}.
	\end{align*}
	Their open complements being equal, we get~$R_H= R_{\dc}$ by~\cite[\S III.6.1.4]{picado2012FramesLocalesTopology}.
\end{proof}

\begin{proposition}
\label{proposition:HFrm to ELoc}
The functor~$\townsendB$ restricts to functors~${\townsendB_\mode\colon\HFrm^\op_\mode \to \ELoc_\mode}$.
\end{proposition}
\begin{proof}
	If~$H$ is a Heyting frame then Townsend's functor~$\townsendB$ produces the ordered Stone locale~$(\idl(\bool{K(H)}), R_H)$. By~\cref{lemma:townsend R is RdcH} the relation~$R_H$ comes from a conic frame, which has open source by~\cref{theorem:induced relation}. Thus~$(\idl(\bool{K(H)}), R_H)$ defines an Esakia locale, and~$\townsendB_\mode$ is well defined on objects.
	
	That~$\townsendB_\weak$ is well defined on morphisms just follows since~$\HFrm_\weak$ has coherent frame morphisms as arrows, which are sent to internally monotone maps by~$\townsendB$. It thus remains to check~$\townsendB_\strong$ is well defined on morphisms. For that, take an arrow~$h\colon H\to G$ in~$\HFrm_\strong$, inducing a frame map~$\idl(\bool{h})$ of the frames underlying the Esakia locales. To show this defines a p-morphism, it suffices via~\cref{corollary:p morphism on Esakia locales} to show that
	\[
	\Downsub{R_G} \circ \idl(\bool{h})
	=
	\idl(\bool{h})\circ \Downsub{R_H}.
	\]
	But, again by~\cref{lemma:townsend R is RdcH}, these cones~$\Downsub{R_H},\Downsub{R_G}$ are just~$\dc{K(H)},\dc{K(G)}$, respectively, and that this equation holds for those cones follows by the proof of~\cref{proposition:Heyt to EFrm}. Hence~$\idl(\bool{h})$ is a frame map underlying a p-morphism, so~$\townsendB_\strong$ is well defined on morphisms.
\end{proof}

\begin{theorem}
	\label{theorem:HFrm ELoc equivalence}
	Townsend's functors~$\townsendB,\townsendC$ restrict to equivalences
	\[
	\townsendB_\mode
	\colon
	\HFrm^\op_\mode
	\simeq
	\ELoc_\mode
	\cocolon
	\townsendC_\mode.
	\]
\end{theorem}
\begin{proof}
	That these functors are well defined follows by~\cref{proposition:HFrm to ELoc,proposition:ELoc to HFrm}. Recalling the observations in~\cref{remark:strategy}, it follows directly that Townsend's~equivalence~${\CohFrm^\op\simeq \OStoneLoc}$ restricts to the equivalence~${\HFrm_\weak^\op\simeq\ELoc_\weak}$. It remains to show that the components of the corresponding natural isomorphisms live in the strong subcategories~$\HFrm_\strong$ and~$\ELoc_\strong$.

	It suffices to show more generally that isomorphisms in the weak categories define isomorphisms in the strong categories. First, let~$h\colon H\to G$ be an isomorphism in~$\HFrm_\weak$. Its restriction~$h|_{K(H)}\colon K(H)\to K(G)$ defines an order isomorphism between Heyting algebras. Since Heyting implication is a right adjoint, it is preserved under order isomorphism, so~$h|_{K(H)}$ must be a Heyting isomorphism. In particular,~$h|_{K(H)}$ and its inverse define Heyting algebra morphisms, and so~$h$ is an isomorphism in~$\HFrm_\strong$.

	Next, let~$f\colon (X,R)\to (Y,Q)$ be an isomorphism in~$\ELoc_\weak$, and denote its inverse by~$g$. Applying the functor~$\cone_\weak$ gives conic morphisms~$f^{-1}$ and~$g^{-1}$, so we get the inclusions of cones
	\[
	\Downsub{R}\circ f^{-1}
	\sqleq
	f^{-1}\circ\Downsub{Q}
	\qquad\text{and}\qquad
	\Downsub{Q}\circ g^{-1}
	\sqleq
	g^{-1}\circ\Downsub{R}.
	\]
	Using the second inclusion we then find straightforwardly that
	\[
	f^{-1}\circ\Downsub{Q}
	=
	f^{-1}\circ\Downsub{Q}\circ g^{-1}\circ f^{-1}
	\sqleq
	f^{-1}\circ g^{-1}\circ\Downsub{R}\circ f^{-1}
	=
	\Downsub{R}\circ f^{-1},
	\]
	which together with the first inclusion gives~$f^{-1}\circ\Downsub{Q} = \Downsub{R}\circ f^{-1}$. A~symmetric argument shows that the analogous condition holds for~$g^{-1}$, and hence~\cref{corollary:p morphism on Esakia locales} shows~$f$ and~$g$ are p-morphisms. Thus~$f$ is an isomorphism in~$\ELoc_\strong$.

	Consequently, the components of the unit and counit of the equivalence between the weak categories in fact define isomorphisms in the strong categories, and we obtain the desired equivalence~$\HFrm^\op_\strong\simeq \ELoc_\strong$.
\end{proof}

Lastly, we show that this equivalence in fact agrees with the one factoring through Esakia frames. The three equivalences here form the following triangle:
\[\begin{tikzcd}[
	column sep=2em,
	row sep=5em,
	labels={/tikz/font=\small, /tikz/text=DiagramFunctorColor}
]
	{\HFrm_\mode^\op} && {\ELoc_\mode} \\
	& {\EFrm^\op_\mode}
	\arrow["{\townsendB_\mode}",sloped, shift left, from=1-1, to=1-3]
	\arrow["{(\patch_\mode\circ K)^\op}"', sloped, shift right, from=1-1, to=2-2]
	\arrow["{\townsendC_\mode}"', sloped, shift left, from=1-3, to=1-1]
	\arrow["{\cone_\mode}", sloped, shift right, from=1-3, to=2-2]
	\arrow["{(\idl\circ \Kdc^\mode)^\op}", sloped, shift right, from=2-2, to=1-1]
	\arrow["{\rel_\mode}"', sloped, shift right, from=2-2, to=1-3]
\end{tikzcd}\]

\begin{proposition}
\label{proposition:townsend compatible}
	The following two equations hold:
	\[
	\townsendB_\mode = \rel_\mode \circ (\patch_\mode \circ K)^\op
	\qquad\text{and}\qquad
	\townsendC_\mode = (\idl\circ \Kdc^\mode)^\op\circ \cone_\mode.
	\]
\end{proposition}
\begin{proof}
	For the first equation, let~$H$ be a Heyting frame. Observe that both functors produce the same locale defined by the frame~$\idl(\bool{K(H)})$, while~$\townsendB_\mode$ equips it with Townsend's localic relation~$R_H$, and the functor on the right-hand side equips it with the relation~$R_{\dc{K(H)}}$. But these relations are equal by~\cref{lemma:townsend R is RdcH}. Both functors also send a morphism~$h$ of Heyting frames to the map~$\idl(\bool{h})$, so the equation follows.

	For the second equation, let~$(X,R)$ be an Esakia locale with induced cone~$\Down$. Townsend's functor~$\townsendC_\mode$ produces the Heyting frame~$\idl K_{\down}(\Opens X)$, while the functor on the right-hand side produces~$\idl \KDown(\Opens X)$. However, these agree by~\cref{corollary:KDown = Kdown}. Both functors send a morphism~$f$ to~$\idl(f^{-1})$, so equality follows.
\end{proof}

%~~~~~~~~~~~~~~~~~~~~~~~~~~~
\section{Conclusion}
\label{section:conclusion}
\cref{section:esakia locales and esakia frames,section:heyting frames and esakia frames,section:townsends functors} establish equivalences between Esakia locales, Esakia frames, Heyting algebras, and Heyting frames.
Taken together, these results yield a constructive, localic Esakia duality. The point-free analogues of Esakia spaces are Esakia locales and Esakia frames. On the algebraic side we use the Heyting frames introduced in~\cite{bezhanishvili2023FrametheoreticPerspectiveEsakia} as the point-free stand-in for Heyting algebras. The equivalence between Heyting frames and Esakia locales can be obtained directly by restricting Townsend's localic Priestley duality~\cite{townsend1997LocalicPriestleyDuality}. Using the notion of conic frame, we showed that this equivalence factors through the category of Esakia frames via the adjunction~$\cone\dashv\rel$. In this framework, the localic partial order is encoded by a frame-theoretic closure operator, the use of which makes the construction of a Heyting algebra more transparent.

All in all, the constructions can be summed up by the following diagram.

\[\begin{tikzcd}[
	column sep=0pt,
	row sep=0pt,
	cells={
		/tikz/outer xsep=.7ex,
		/tikz/outer ysep=.7ex
	},
	labels={/tikz/font=\small},
	marking/.append style={/tikz/font=\small},
	execute before arrows={
		\node[anchor=center, outer xsep=.6ex, outer ysep=.6ex] (\tikzcdmatrixname-3-1) at
			([xshift={-\summarydiagrambottomrowshift*(\summarydiagramlongcolumnsep/4)}]\tikzcdmatrixname-3-1.center)
			{$\Heyt^\op$};
		\node[anchor=center, outer xsep=.6ex, outer ysep=.6ex] (\tikzcdmatrixname-3-2) at
			([xshift={-\summarydiagrambottomrowshift*(\summarydiagramlongcolumnsep/4)}]\tikzcdmatrixname-3-2.center)
			{$\EFrm^\op$};
		\node[anchor=center, outer xsep=.6ex, outer ysep=.6ex] (\tikzcdmatrixname-3-3) at
			([xshift={-\summarydiagrambottomrowshift*(\summarydiagramlongcolumnsep/4)}]\tikzcdmatrixname-3-3.center)
			{$\dcFrm^\op$};
	}
]
	{\CohFrm^\op} &[\summarydiagramlongcolumnsep] \OStoneLoc &[\summarydiagramtightcolumnsep] \rLoc \\[\summarydiagramshortrowsep]
	{\HFrm^\op} & \ELoc & \rosLoc \\[\summarydiagramlongrowsep]
	{} & {} & {}
	\arrow[shift left=2, from=1-1, to=1-2]
	\arrow["{\citedsimeq{townsend1997LocalicPriestleyDuality}}"{marking, allow upside down}, draw=none, from=1-1, to=1-2]
	\arrow["{\color{\summarydiagramfunctorcolor}\townsendC}"{marking, font=\normalsize, allow upside down}, shift right=5, draw=none, from=1-1, to=1-2]
	\arrow["{\color{\summarydiagramfunctorcolor}\townsendB}"{marking, font=\normalsize, allow upside down}, shift left=5, draw=none, from=1-1, to=1-2]
	\arrow[shift left=2, from=1-2, to=1-1]
	\arrow["\subseteq"{anchor=center}, draw=none, from=1-2, to=1-3]
	\arrow["\subseteq"{marking, allow upside down}, draw=none, from=2-1, to=1-1]
	\arrow[shift left=2, from=2-1, to=2-2]
	\arrow["{\refsimeq{theorem:HFrm ELoc equivalence}}"{marking, allow upside down}, draw=none, from=2-1, to=2-2]
	\arrow[shift right=2, from=2-1, to=3-1]
	\arrow["{\color{\summarydiagrammapcolor}\idl\Kdc(E)\longmapsfrom(E,\dc)}"{marking, allow upside down}, shift left=5, draw=none, from=2-1, to=3-2]
	\arrow["{\color{\summarydiagrammapcolor} H\longmapsto (E_{K(H)},\dc{K(H)})}"{marking, allow upside down, pos=.45}, shift right=5, draw=none, from=2-1, to=3-2]
	\arrow[shift right=2, from=2-1, to=3-2]
	\arrow["{\unlinkedrefsimeq{theorem:EFrm HFrm equivalence}}"{marking, allow upside down}, draw=none, from=2-1, to=3-2]
	\arrow["{\hyperref[theorem:EFrm HFrm equivalence]{\phantom{\rule[-.78em]{.95em}{1.56em}}}}"{anchor=center, xshift=6pt, yshift=-7.2pt}, draw=none, from=2-1, to=3-2]
	\arrow["\subseteq"{marking, allow upside down}, draw=none, from=2-2, to=1-2]
	\arrow[shift left=2, from=2-2, to=2-1]
	\arrow["\subseteq"{anchor=center}, draw=none, from=2-2, to=2-3]
	\arrow[shift left=2, from=2-2, to=3-2]
	\arrow["\subseteq"{marking, allow upside down}, draw=none, from=2-3, to=1-3]
	\arrow[""{name=0, anchor=center, inner sep=0}, shift right=2, from=2-3, to=3-3]
	\arrow[shift right=2, from=3-1, to=2-1]
	\arrow["{\color{\summarydiagrammapcolor}A\longmapsto \idl(A)}"{marking, allow upside down}, shift right=5, draw=none, from=3-1, to=2-1]
	\arrow["{\color{\summarydiagrammapcolor}K(H)\longmapsfrom H}"{marking, allow upside down}, shift left=5, draw=none, from=3-1, to=2-1]
	\arrow["{\unlinkedcitedsimeq{BCM23}}"{marking, allow upside down}, draw=none, from=3-1, to=2-1]
	\arrow["{\hyperlink{cite.0@bezhanishvili2023FrametheoreticPerspectiveEsakia}{\phantom{\rule[-2.1em]{1.66em}{4.2em}}}}"{anchor=center, xshift=2.8pt, yshift=11pt}, draw=none, from=3-1, to=2-1]
	\arrow["{\color{\summarydiagrammapcolor}\Kdc(E)\longmapsfrom(E,\dc)}"{marking, allow upside down, pos={.5-.08*\summarydiagrambottomrowshift}}, shift left=5, draw=none, from=3-1, to=3-2]
	\arrow[shift right=2, from=3-1, to=3-2]
	\arrow["{\color{\summarydiagrammapcolor}A\longmapsto(E_A,\dc{A})}"{marking, allow upside down, pos={.5-.08*\summarydiagrambottomrowshift}}, shift right=5, draw=none, from=3-1, to=3-2]
	\arrow["{\refsimeq{theorem:EFrm Heyt equivalence}}"{marking, allow upside down}, draw=none, from=3-1, to=3-2]
	\arrow[shift right=2, from=3-2, to=2-1]
	\arrow["{\color{\summarydiagrammapcolor} (E,\dc)\longmapsto (X,R_{\dc})}"{marking, allow upside down}, shift left=5, draw=none, from=3-2, to=2-2]
	\arrow["{\color{\summarydiagrammapcolor} (\Opens X,\Down)\longmapsfrom (X,R)}"{marking, allow upside down}, shift right=5, draw=none, from=3-2, to=2-2]
	\arrow[shift left=2, from=3-2, to=2-2]
\arrow["\simeq"{marking, allow upside down}, draw=none, from=3-2, to=2-2]
\arrow["{\text{\scriptsize(\ref{theorem:ELoc EFrm equivalence})}}"{anchor=center, xshift={-\summarydiagrambottomrowshift*.12em}, yshift=-2.1ex, fill=white, inner sep=.5pt}, draw=none, from=3-2, to=2-2]
	\arrow[shift right=2, from=3-2, to=3-1]
	\arrow["\subseteq"{anchor=center}, draw=none, from=3-2, to=3-3]
	\arrow[""{name=1, anchor=center, inner sep=0}, shift right=2, from=3-3, to=2-3]
	\arrow["{\color{\summarydiagramfunctorcolor}\cone}"{marking, font=\normalsize, allow upside down}, shift left=5, draw=none, from=3-3, to=2-3]
	\arrow["{\color{\summarydiagramfunctorcolor}\rel}"{marking, font=\normalsize, allow upside down}, shift right=5, draw=none, from=3-3, to=2-3]
	\arrow["\dashv"{anchor=center, sloped, allow upside down}, draw=none, from=0, to=1]
	\arrow["{\text{\scriptsize(\ref{theorem:cone rel adjunction})}}"{anchor=center, xshift={-\summarydiagrambottomrowshift*.12em}, yshift=-2.1ex, fill=white, inner sep=.5pt}, draw=none, from=0, to=1]
\end{tikzcd}\]

Note that the bottom-left triangle commutes by construction, since the equivalence between Esakia frames and Heyting frames was defined to be the composition of the equivalences to Heyting algebras. The top-right triangle commutes by~\cref{proposition:townsend compatible}.

%~~~~~~~~~~~~~~~~~~~~~~~~~~
\subsection{Discussion}
\label{section:discussion}
We close the paper with some directions for future research and connections to the literature.
\begin{enumerate}
	\item As suggested in~\cite[\S 12.1]{schaaf2026LocalicRelationsOpenCones}, a natural direction would be to attempt to develop a more general theory of cones on locales that also encapsulates localic relations where neither the source nor target map is open. Using such a framework, Townsend's localic Priestley duality~\cite{townsend1997LocalicPriestleyDuality} could then be reinterpreted analogously to the constructions in the present work.
	
	\item Similarly, an alternative description of Townsend's duality~\cite{townsend1997LocalicPriestleyDuality} was developed in~\cite{korostenski2007LaxProperMaps} using lax proper maps of locales. The idea there is to model upwards closure of clopens by the coinserter~$q\colon X\tworightarrow C$ of the target~$t$ and source map~$s$ of the localic relation~$R$, which is the universal map so that~$q\circ t\sqleq q\circ s$. The subframe inclusion~$q^{-1}\colon \Opens C\hookrightarrow \Opens X$ picks out exactly those~$U\in \Opens X$ such that~$s^{-1}(U)\sqleq t^{-1}(U)$. Assuming that~$R$ is reflexive, we see that this holds on the level of sublocales iff~${\up U = t\left[s^{-1}[U]\right] = U}$, capturing upwards closure as desired, without relying on complements as in the definition of~$\Kdc(L)$. Relations between~\cite{korostenski2007LaxProperMaps} and the present work could be investigated.
	
	\item What is the relation between the present work and the choice-free Stone duality in~\cite{bezhanishvili2020ChoiceFreeStone} using upper Vietoris spaces? Studying a localic Vietoris construction using localic relations was also suggested to us independently by Tomáš Jakl. This could also relate to the approach~\cite{hofmann2014NotesEsakiaSpaces} using distributors.
	
	\item The work~\cite{bezhanishvili2023RemarksHyperspacesPriestley} uses \emph{Egli-Milner orders} on hyperspaces of Priestley spaces to build connections to coalgebraic modal logic. (We thank Nick Bezhanishvili for making us aware of this work.) It was shown in~\cite[\S 11]{schaaf2026LocalicRelationsOpenCones} that conic frames can be viewed as frames equipped with generalised Egli-Milner orders, in the sense of~\cite{heunen2024OrderedLocales}. Possible connections between these frameworks could be investigated.
	
	\item The present work should dualise to a localic account of co-Heyting algebras and co-Esakia spaces, similar to the treatment in~\cite{bezhanishvili2010BitopologicalDualityDistributive}. Namely, a \emph{co-Esakia space} is a Priestley space with open \emph{target} map. Thus, we could define \emph{co-Esakia locales} as ordered Stone locales with open target map, which would then induce an upwards closure cone~$\Up=t_!\circ s^{-1}$. This in turn suggests the notion of a \emph{co-Esakia frame}~$(E,\uc)$, for which the set of compact upper elements
	\[
	K_{\!\uc}(E)
	:=
	\{
		a\in K(E): \uc a = a
	\}
	\]
	should then form a co-Heyting algebra with co-implication for~$a,b\in K_{\uc}(E)$ defined by
	\[
		a\setminus b := \uc(a\wedge \neg b).
	\]
	Since conic frames~$(L,\uc)$ and~$(L,\dc)$ are formally the same thing, the difference lies in how the localic relation~$R_{\uc}$ is recovered, which in this case would be the one generated by the relation~$\sim$ defined as:
	\[
	\forall x,y\in L: x\tensor y \sim x\tensor \Uc(x,y).
	\]

	Lastly, bi-Heyting algebras could be modeled using the two-sided notion of conic frame~$(L,\uc,\dc)$ in~\cite{schaaf2026LocalicRelationsOpenCones}. There, the notion of \emph{parallelness} generalises the notion of conjugate pairs on Boolean algebras from~\cite{jonsson1951BooleanAlgebrasOperators}.
	Analogous to the discussion in~\cite[\S 4.2.4]{schaaf2024TowardsPointFreeSpacetimes}, for a two-sided conic frame~$(L,\uc,\dc)$ where~$\uc,\dc$ define closure operators, the images~$\im(\uc)$ and~$\im(\dc)$ form subframes of~$L$, and the triple~$(L,\im(\uc),\im(\dc))$ defines a \emph{biframe}~\cite{banaschewski1983BiframesBispaces} precisely when the `diamonds'~$\uc x\wedge \dc y$ form a join-basis for~$L$. Connections between the bitopological treatment of Priestley and Esakia duality in~\cite{bezhanishvili2010BitopologicalDualityDistributive} and the point-free conic frame approach in the present work could be investigated.

	\item In~\cref{remark:closure algebras} we saw that for any Esakia frame~$(E,\dc)$ we get a closure operator~$(K(E),\dc)$ in the sense of McKinsey and Tarski~\cite{mckinsey1944AlgebraTopology,mckinsey1946ClosedElementsClosure}, relating also to the framework of Boolean algebras with operators~\cite{jonsson1951BooleanAlgebrasOperators}.
	This suggests connections to (intuitionistic) modal logic~\cite{fischerservi1977ModalLogicIntuitionistic}, including duality-theoretic treatments of Heyting algebras with operators~\cite{orlowska2007DiscreteDualitiesHeyting} and the point-free setting of~\cite{hilken2000TopologicalDualityIntuitionistic}.
	Relations between Esakia frames and the point-free setting of McKinsey-Tarski algebras of~\cite{bezhanishvili2023McKinseyTarskiAlgebras} could also be investigated.
	
	\item Esakia duality can be used to give a spatial proof that Heyting algebras satisfy \emph{amalgamation} and \emph{interpolation} properties~\cite{maksimova1977CraigsTheoremSuperintuitionistic}. See for instance~\cite[\S 5]{maksimova1999InterrelationAlgebraicSemantical}, or the more recent Theorems~4.11 and~4.12 in~\cite{almeida2026ColimitsHeytingAlgebras}. On the other hand, Pitts gave a constructive localic proof of the amalgamation and interpolation properties of Heyting algebras in~\cite{pitts1983AmalgamationInterpolation}, but without Esakia duality. Can Pitts's results be recovered with the spatial intuition of Esakia duality, but using our point-free and constructive setting of Esakia locales or frames?
\end{enumerate}

\newpage
%~~~~~~~~~~~~~~~~~~~~~~~~~~
\subsection*{Acknowledgments} This work has been partially funded by the French National Research Agency (ANR) within the framework of ``Plan France 2030'', under the research projects EPIQ ANR-22-PETQ-0007, HQI-Acquisition ANR-22-PNCQ-0001 and HQI-R\&D ANR-22-PNCQ-0002.

%~~~~~~~~~~~~~~~~~~~~~~~~~~
\subsection*{AI tool disclaimer}
AI tools were used in this work for conceptual development, proof sketching and checking, literature search, suggestions for improved exposition, proofreading, and typesetting (diagrams). All final text and proofs were written manually by the author, who has verified the mathematical claims and references, and takes full responsibility for the contents of the paper.

\printbibliography

@article{bezhanishvili2023RemarksHyperspacesPriestley,
	title = {Remarks on Hyperspaces for {{Priestley}} Spaces},
	author = {Bezhanishvili, G. and Harding, J. and Morandi, P. J.},
	date = {2023-01-17},
	journaltitle = {Theoretical Computer Science},
	shortjournal = {Theor. Comput. Sci.},
	volume = {943},
	pages = {187--202},
	issn = {0304-3975},
	doi = {10.1016/j.tcs.2022.12.001}
}

@article{maksimova1977CraigsTheoremSuperintuitionistic,
	title = {Craig's Theorem in Superintuitionistic Logics and Amalgamable Varieties of Pseudo-Boolean Algebras},
	author = {Maksimova, L. L.},
	date = {1977-11-01},
	journaltitle = {Algebra and Logic},
	shortjournal = {Algebra Log.},
	volume = {16},
	number = {6},
	pages = {427--455},
	issn = {1573-8302},
	doi = {10.1007/BF01670006}
}

@article{maksimova1999InterrelationAlgebraicSemantical,
	title = {Interrelation of Algebraic, Semantical and Logical Properties for Superintuitionistic and Modal Logics},
	author = {Maksimova, Larisa},
	date = {1999},
	journaltitle = {Banach Center Publications},
	shortjournal = {Banach Cent. Publ.},
	volume = {46},
	number = {1},
	pages = {159--168},
	issn = {0137-6934}
}

@online{almeida2026ColimitsHeytingAlgebras,
	title = {Colimits of {{Heyting Algebras}} through {{Esakia Duality}}},
	author = {Almeida, Rodrigo Nicolau},
	date = {2026-04-02},
	eprint = {2402.08058},
	eprinttype = {arXiv},
	eprintclass = {math.LO},
	doi = {10.48550/arXiv.2402.08058}
}

@article{stone1938TopologicalRepresentationsDistributive,
	title = {Topological Representations of Distributive Lattices and {{Brouwerian}} Logics},
	author = {Stone, Marshall Harvey},
	date = {1938},
	journaltitle = {Časopis pro pěstování matematiky a fysiky},
	shortjournal = {Časopis Pěst. Mat. Fys.},
	volume = {067},
	number = {1},
	pages = {1--25},
	issn = {1802-114X}
}

@article{priestley1970RepresentationDistributiveLattices,
	title = {Representation of {{Distributive Lattices}} by Means of Ordered {{Stone Spaces}}},
	author = {Priestley, H. A.},
	date = {1970},
	journaltitle = {Bulletin of the London Mathematical Society},
	shortjournal = {Bull. Lond. Math. Soc.},
	volume = {2},
	number = {2},
	pages = {186--190},
	issn = {1469-2120},
	doi = {10.1112/blms/2.2.186}
}

@article{stone1936TheoryRepresentationBoolean,
	title = {The {{Theory}} of {{Representation}} for {{Boolean Algebras}}},
	author = {Stone, M. H.},
	date = {1936},
	journaltitle = {Transactions of the American Mathematical Society},
	shortjournal = {Trans. Am. Math. Soc.},
	volume = {40},
	number = {1},
	eprint = {1989664},
	eprinttype = {jstor},
	pages = {37--111},
	issn = {0002-9947},
	doi = {10.2307/1989664}
}

@article{banaschewski1983BiframesBispaces,
	title = {Biframes and {{Bispaces}}},
	author = {Banaschewski, B. and Brümmer, G. C.L. and Hardie, K. A.},
	date = {1983-01-01},
	journaltitle = {Quaestiones Mathematicae},
	shortjournal = {Quaest. Math.},
	volume = {6},
	number = {1--3},
	pages = {13--25},
	publisher = {Taylor \& Francis},
	doi = {10.1080/16073606.1983.9632289}
}

@article{hilken2000TopologicalDualityIntuitionistic,
	title = {Topological Duality for Intuitionistic Modal Algebras},
	author = {Hilken, Barnaby P.},
	date = {2000-04-28},
	journaltitle = {Journal of Pure and Applied Algebra},
	shortjournal = {J. Pure Appl. Algebra},
	volume = {148},
	number = {2},
	pages = {171--189},
	doi = {10.1016/S0022-4049(98)00170-4}
}

@article{fischerservi1977ModalLogicIntuitionistic,
	title = {On Modal Logic with an Intuitionistic Base},
	author = {Fischer Servi, Gisèle},
	date = {1977-09-01},
	journaltitle = {Studia Logica},
	shortjournal = {Stud. Logica},
	volume = {36},
	number = {3},
	pages = {141--149},
	doi = {10.1007/BF02121259}
}

@article{orlowska2007DiscreteDualitiesHeyting,
	title = {Discrete {{Dualities}} for {{Heyting Algebras}} with {{Operators}}},
	author = {Orłowska, Ewa and Rewitzky, Ingrid},
	date = {2007-01-01},
	journaltitle = {Fundamenta Informaticae},
	shortjournal = {Fundam. Informaticae},
	volume = {81},
	number = {1--3},
	pages = {275--295},
	issn = {0169-2968}
}

@online{hofmann2014NotesEsakiaSpaces,
	title = {Some Notes on {{Esakia}} Spaces},
	author = {Hofmann, Dirk and Nora, Pedro},
	date = {2014-08-05},
	eprint = {1408.1072},
	eprinttype = {arXiv},
	eprintclass = {math.CT},
	doi = {10.48550/arXiv.1408.1072}
}

@article{mckinsey1944AlgebraTopology,
	title = {The {{Algebra}} of {{Topology}}},
	author = {McKinsey, J. C. C. and Tarski, Alfred},
	date = {1944},
	journaltitle = {Annals of Mathematics},
	shortjournal = {Ann. Math.},
	volume = {45},
	number = {1},
	pages = {141--191},
	doi = {10.2307/1969080}
}

@article{hartonas2023ChoicefreeTopologicalDuality,
	title = {Choice-Free Topological Duality for Implicative Lattices and {{Heyting}} Algebras},
	author = {Hartonas, Chrysafis},
	date = {2023-11-14},
	journaltitle = {Algebra universalis},
	volume = {85},
	number = {1},
	pages = {3},
	issn = {1420-8911},
	doi = {10.1007/s00012-023-00830-8}
}

@article{mckinsey1946ClosedElementsClosure,
	title = {On {{Closed Elements}} in {{Closure Algebras}}},
	author = {McKinsey, J. C. C. and Tarski, Alfred},
	date = {1946},
	journaltitle = {Annals of Mathematics},
	shortjournal = {Ann. Math.},
	volume = {47},
	number = {1},
	% eprint = {1969038},
	% eprinttype = {jstor},
	pages = {122--162},
	doi = {10.2307/1969038}
}

@incollection{banaschewski1989UniversalZero,
	author    = {Banaschewski, Bernhard},
	title     = {Universal zero-dimensional compactifications},
	booktitle = {Categorical Topology and its Relation to Analysis, Algebra and Combinatorics},
	editor    = {Ad{\'a}mek, Ji{\v{r}}{\'i} and Mac Lane, Saunders},
	pages     = {257--269},
	publisher = {World Scientific},
	address   = {Teaneck, NJ},
	year      = {1989},
}

@book{gehrke2024TopologicalDualityDistributive,
	title = {Topological {{Duality}} for {{Distributive Lattices}}: {{Theory}} and {{Applications}}},
	shorttitle = {Topological {{Duality}} for {{Distributive Lattices}}},
	author = {Gehrke, Mai and van Gool, Sam},
	date = {2024-02-29},
	edition = {1},
	publisher = {Cambridge University Press},
	doi = {10.1017/9781009349680}
}

@book{riehl2016CategoryTheoryContext,
	title = {Category Theory in Context},
	author = {Riehl, Emily},
	date = {2016},
	series = {Aurora {{Dover}} Modern Math Originals},
	publisher = {Dover Publications, Inc},
	isbn = {978-0-486-80903-8},
	pagetotal = {240}
}

@book{davey2002IntroductionLatticesOrder,
	title = {Introduction to {{Lattices}} and {{Order}}},
	author = {Davey, B. A. and Priestley, H. A.},
	date = {2002},
	edition = {2},
	publisher = {Cambridge University Press},
	doi = {10.1017/CBO9780511809088},
	isbn = {978-0-521-78451-1}
}

@article{bezhanishvili2023McKinseyTarskiAlgebras,
	title = {{{McKinsey-Tarski}} Algebras: {{An}} Alternative Pointfree Approach to Topology},
	shorttitle = {{{McKinsey-Tarski}} Algebras},
	author = {Bezhanishvili, Guram and Raviprakash, Ranjitha},
	date = {2023-11-01},
	journaltitle = {Topology and its Applications},
	shortjournal = {Topol. Its Appl.},
	volume = {339},
	pages = {108689},
	issn = {0166-8641},
	doi = {10.1016/j.topol.2023.108689}
}

@book{borceux1994Handbook1,
	title = {Handbook of {{Categorical Algebra}}: {{Volume}} 1: {{Basic Category Theory}}},
	shorttitle = {Handbook of {{Categorical Algebra}}},
	author = {Borceux, Francis},
	date = {1994},
	series = {Encyclopedia of {{Mathematics}} and Its {{Applications}}},
	volume = {1},
	publisher = {Cambridge University Press},
	doi = {10.1017/CBO9780511525858},
	isbn = {978-0-521-44178-0}
}

@article{bezhanishvili2020ChoiceFreeStone,
	title = {Choice-{{Free Stone Duality}}},
	author = {Bezhanishvili, Nick and Holliday, Wesley H.},
	date = {2020-03},
	journaltitle = {The Journal of Symbolic Logic},
	volume = {85},
	number = {1},
	pages = {109--148},
	issn = {0022-4812, 1943-5886},
	doi = {10.1017/jsl.2019.11}
}

@article{maietti2012InductionPrincipleConsequence,
	title = {An Induction Principle for Consequence in Arithmetic Universes},
	author = {Maietti, Maria Emilia and Vickers, Steven},
	date = {2012-08-01},
	journaltitle = {Journal of Pure and Applied Algebra},
	shortjournal = {Journal of Pure and Applied Algebra},
	series = {Special {{Issue}} Devoted to the {{International Conference}} in {{Category Theory}} `{{CT2010}}'},
	volume = {216},
	number = {8},
	pages = {2049--2067},
	issn = {0022-4049},
	doi = {10.1016/j.jpaa.2012.02.040}
}

@article{peremans1957EmbeddingDistributiveLattice,
	title = {Embedding of a Distributive Lattice into a Boolean Algebra},
	author = {Peremans, Wouter},
	date = {1957-01-01},
	journaltitle = {Proceedings of the KNAW - Series A, Mathematical Sciences},
	volume = {60},
	number = {1},
	pages = {73--81}
}

@book{esakia2019HeytingAlgebrasDuality,
	title = {Heyting {{Algebras}}: {{Duality Theory}}},
	shorttitle = {Heyting {{Algebras}}},
	author = {Esakia, Leo},
	editor = {Bezhanishvili, Guram and Holliday, Wesley H.},
	date = {2019},
	series = {Trends in {{Logic}}},
	volume = {50},
	publisher = {Springer International Publishing},
	doi = {10.1007/978-3-030-12096-2},
	isbn = {978-3-030-12095-5 978-3-030-12096-2}
}

@online{schaaf2026LocalicRelationsOpenCones,
	title = {Localic {{Relations}} with {{Open Cones}}},
	author = {family=Schaaf, given=Nesta, prefix=van der, useprefix=true},
	date = {2026-05-05},
	eprint = {2605.04038},
	eprinttype = {arXiv},
	doi = {10.48550/arXiv.2605.04038},
	label = {vdS}
}

@article{korostenski2007LaxProperMaps,
  title = {Lax Proper Maps of Locales},
  author = {Korostenski, M. and Labuschagne, C. C. A.},
  date = {2007-02-01},
  journaltitle = {Journal of Pure and Applied Algebra},
  volume = {208},
  number = {2},
  pages = {655--664},
  issn = {0022-4049},
  doi = {10.1016/j.jpaa.2006.03.003}
}

@article{esakia1974topologicalKripkemodels,
	title = {Topological {K}ripke Models},
	author = {Esakia, L. L.},
	year = {1974},
	journal = {Doklady Akademii Nauk SSSR},
	volume = {214},
	number = {2},
	pages = {298--301}
}

@article{jonsson1951BooleanAlgebrasOperators,
  title = {Boolean {{Algebras}} with {{Operators}}. {{Part I}}},
  author = {Jónsson, Bjarni and Tarski, Alfred},
  date = {1951},
  journaltitle = {American Journal of Mathematics},
  volume = {73},
  number = {4},
  pages = {891--939},
  publisher = {The Johns Hopkins University Press},
  issn = {0002-9327},
  doi = {10.2307/2372123}
}

@article{johnstone1989ConstructiveClosedSubgroup,
  title = {A Constructive Closed Subgroup Theorem for Localic Groups and Groupoids},
  author = {Johnstone, Peter T.},
  date = {1989},
  journaltitle = {Cahiers de Topologie et Geometrie Differentielle Categoriques},
  volume = {30},
  number = {1},
  pages = {3--23}
}

@article{kock1989GodementTheoremLocales,
  title = {A {{Godement}} Theorem for Locales},
  author = {Kock, Anders},
  date = {1989-05},
  journaltitle = {Mathematical Proceedings of the Cambridge Philosophical Society},
  shortjournal = {Math. Proc. Camb. Phil. Soc.},
  volume = {105},
  number = {3},
  pages = {463--471},
  issn = {0305-0041, 1469-8064},
  doi = {10.1017/S0305004100077835}
}

@article{priestley1972OrderedTopologicalSpaces,
  title = {Ordered {{Topological Spaces}} and the {{Representation}} of {{Distributive Lattices}}},
  author = {Priestley, H. A.},
  date = {1972},
  journaltitle = {Proceedings of the London Mathematical Society},
  volume = {s3-24},
  number = {3},
  pages = {507--530},
  issn = {1460-244X},
  doi = {10.1112/plms/s3-24.3.507}
}

@article{bezhanishvili2023FrametheoreticPerspectiveEsakia,
  title = {A Frame-Theoretic Perspective on {{Esakia}} Duality},
  author = {Bezhanishvili, G. and Carai, L. and Morandi, P. J.},
  date = {2023-09-30},
  journaltitle = {Algebra universalis},
  volume = {84},
  number = {4},
  pages = {30},
  issn = {1420-8911},
  doi = {10.1007/s00012-023-00827-3}
}

@book{johnstone2002Elephant2,
  title = {Sketches of an {{Elephant}}: {{A Topos Theory Compendium}}},
  shorttitle = {Sketches of an {{Elephant}}},
  author = {Johnstone, Peter T.},
  date = {2002-11-21},
  series = {Oxford {{Logic Guides}}},
  volume = {2},
  publisher = {Oxford University Press},
  isbn = {978-0-19-851598-2}
}

@article{townsend1997LocalicPriestleyDuality,
	title = {Localic {{Priestley}} Duality},
	author = {Townsend, Chris},
	date = {1997-03-28},
	journaltitle = {Journal of Pure and Applied Algebra},
	shortjournal = {Journal of Pure and Applied Algebra},
	volume = {116},
	number = {1},
	pages = {323--335},
	issn = {0022-4049},
	doi = {10.1016/S0022-4049(96)00169-7}
}

@book{maclane1998CategoriesWorkingMathematician,
	title = {Categories for the Working Mathematician},
	author = {MacLane, Saunders},
	date = {1998},
	series = {Graduate Texts in Mathematics},
	edition = {2nd ed},
	number = {5},
	publisher = {Springer},
	location = {New York},
	isbn = {978-0-387-98403-2},
	pagetotal = {314}
}

@thesis{townsend1996preframeTechniquesConstructiveLocale,
	author = {Townsend, Christopher Francis},
	year = {1996},
	title = {Preframe techniques in constructive locale theory},
	institution = {University of London},
	type = {PhD thesis}
}

@article{bezhanishvili2010BitopologicalDualityDistributive,
	title = {Bitopological Duality for Distributive Lattices and {{Heyting}} Algebras},
	author = {Bezhanishvili, Guram and Bezhanishvili, Nick and Gabelaia, David and Kurz, Alexander},
	date = {2010-06},
	journaltitle = {Mathematical Structures in Computer Science},
	volume = {20},
	number = {3},
	pages = {359--393},
	publisher = {Cambridge University Press},
	issn = {1469-8072, 0960-1295},
	doi = {10.1017/S0960129509990302}
}

@article{moshier2017GeneratingSublocalesSubsets,
	title = {Generating Sublocales by Subsets and Relations: A Tangle of Adjunctions},
	shorttitle = {Generating Sublocales by Subsets and Relations},
	author = {Moshier, M. Andrew and Picado, Jorge and Pultr, Aleš},
	date = {2017-09-01},
	journaltitle = {Algebra universalis},
	shortjournal = {Algebra Univers.},
	volume = {78},
	number = {1},
	pages = {105--118},
	issn = {1420-8911},
	doi = {10.1007/s00012-017-0446-z}
}

@article{klein1970RelationsCategories,
	title = {Relations in Categories},
	author = {Klein, Aaron},
	date = {1970-12},
	journaltitle = {Illinois Journal of Mathematics},
	volume = {14},
	number = {4},
	pages = {536--550},
	publisher = {Duke University Press},
	issn = {0019-2082, 1945-6581},
	doi = {10.1215/ijm/1256052950}
}

@article{pitts1983AmalgamationInterpolation,
	title = {Amalgamation and Interpolation in the Category of Heyting Algebras},
	author = {Pitts, A. M.},
	date = {1983-08-01},
	journaltitle = {Journal of Pure and Applied Algebra},
	shortjournal = {Journal of Pure and Applied Algebra},
	volume = {29},
	number = {2},
	pages = {155--165},
	issn = {0022-4049},
	doi = {10.1016/0022-4049(83)90104-4}
}

@book{johnstone1982StoneSpaces,
	title = {Stone Spaces},
	author = {Johnstone, P. T.},
	date = {1982},
	series = {Cambridge Studies in Advanced Mathematics},
	number = {3},
	publisher = {Cambridge University Press},
	isbn = {978-0-521-23893-9}
}

@book{joyal1984ExtensionGaloisTheory,
	title = {An Extension of the {{Galois}} Theory of {{Grothendieck}}},
	author = {Joyal, André and Tierney, Myles},
	date = {1984},
	series = {Memoirs of the {{American Mathematical Society}}},
	volume = {51},
	number = {309},
	publisher = {American Mathematical Society},
	issn = {0065-9266, 1947-6221},
	doi = {10.1090/memo/0309},
	isbn = {978-0-8218-2312-5 978-1-4704-0722-3}
}

@article{picado2015NotesProductLocales,
	title = {Notes on the {{Product}} of {{Locales}}},
	author = {Picado, Jorge and Pultr, Aleš},
	date = {2015-04-01},
	journaltitle = {Mathematica Slovaca},
	volume = {65},
	number = {2},
	pages = {247--264},
	issn = {1337-2211, 0139-9918},
	doi = {10.1515/ms-2015-0020}
}

@inproceedings{johnstone1991PreframePresentationsPresent,
	title = {Preframe Presentations Present},
	booktitle = {Category {{Theory}}},
	author = {Johnstone, Peter T and Vickers, Steven},
	editor = {Carboni, Aurelio and Pedicchio, Maria Cristina and Rosolini, Guiseppe},
	date = {1991},
	pages = {193--212},
	publisher = {Springer},
	location = {Berlin, Heidelberg},
	doi = {10.1007/BFb0084221},
	isbn = {978-3-540-46435-8}
}

@article{dowker1977SumsCategoryFrames,
	title = {Sums in the Category of Frames},
	author = {Dowker, C. H. and Strauss, D},
	date = {1977},
	journaltitle = {Houston Journal of Mathematics},
	volume = {3},
	number = {1},
	pages = {17--32}
}

@article{heunen2024OrderedLocales,
	title = {Ordered Locales},
	author = {Heunen, Chris and van der Schaaf, Nesta},
	year = {2024},
	journal = {Journal of Pure and Applied Algebra},
	volume = {228},
	number = {7},
	pages = {107654},
	issn = {0022-4049},
	doi = {10.1016/j.jpaa.2024.107654},
	urldate = {2024-03-12},
	label = {HS}
}

@thesis{schaaf2024TowardsPointFreeSpacetimes,
	title = {Towards Point-Free Spacetimes},
	author = {van der Schaaf, Nesta},
	year = {2024},
	number = {arXiv:2406.15406},
	eprint = {2406.15406},
	publisher = {arXiv},
	doi = {10.48550/arXiv.2406.15406},
	archiveprefix = {arXiv},
	note = {PhD thesis},
	institution = {University of Edinburgh},
	label = {vdS}
}

@book{picado2012FramesLocalesTopology,
	title = {Frames and Locales: Topology without Points},
	shorttitle = {Frames and Locales},
	author = {Picado, Jorge and Pultr, Aleš},
	date = {2012},
	series = {Frontiers in Mathematics},
	publisher = {Birkhäuser},
	location = {Basel},
	isbn = {978-3-0348-0154-6}
}

@book{maclane1994SheavesGeometryLogic,
	title = {Sheaves in {{Geometry}} and {{Logic}}: {{A First Introduction}} to {{Topos Theory}}},
	shorttitle = {Sheaves in {{Geometry}} and {{Logic}}},
	author = {Mac Lane, Saunders and Moerdijk, Ieke},
	date = {1994},
	series = {Universitext},
	publisher = {Springer New York},
	doi = {10.1007/978-1-4612-0927-0},
	isbn = {978-0-387-97710-2 978-1-4612-0927-0}
}
\end{document}